\documentclass[12pt]{article}
\usepackage{epsfig}
\usepackage{amsfonts}
\usepackage{amssymb}
\usepackage{amsmath}
\usepackage{amsthm}
\usepackage{latexsym}
\usepackage{graphicx}
\usepackage{bm}
\usepackage{tabularx}
\usepackage{booktabs} 
\usepackage{enumerate}
\usepackage{caption}
\usepackage[titletoc]{appendix}
\usepackage[numbers]{natbib}
\usepackage{bbm}
\usepackage{url}
\usepackage{subfig}
\usepackage{hyperref}
\usepackage{color}
\usepackage{xcolor}
\usepackage[section]{placeins}
\usepackage[thinlines]{easytable}
\usepackage{afterpage}

\catcode`\@=11

\newcommand{\shortdot}[1]{\raisebox{-0.4pt}{$\stackrel{\bullet}{#1}$}}
\theoremstyle{plain}
\newtheorem{theorem}{Theorem}[section]

\theoremstyle{definition}

\theoremstyle{remark}

\begin{document}

\title{Breaking the Symmetry in Queues with Delayed Information}
\author{   
Philip Doldo \\ Center for Applied Mathematics \\ Cornell University
\\ 657 Rhodes Hall, Ithaca, NY 14853 \\  pmd93@cornell.edu  \\ 
\and
  Jamol Pender \\ School of Operations Research and Information Engineering \\ Center for Applied Mathematics \\ Cornell University
\\ 228 Rhodes Hall, Ithaca, NY 14853 \\  jjp274@cornell.edu  \\ 
\and
Richard Rand \\ Department of Mathematics \\ Sibley School of Mechanical and Aerospace Engineering  \\ Cornell University
\\ 417 Upson Hall, Ithaca, NY 14853 \\  rand@math.cornell.edu  \\ 
 }

%\renewcommand{\textfraction}{0.0001}
%\setlength{\floatsep}{10pt plus 1.0pt minus 5000.0pt}
%%%%%% neither of these fix the issues with the plots taking up too much blank space

\maketitle
\begin{abstract}

Giving customers queue length information about a service system has the potential to influence the decision of a customer to join a queue.  Thus, it is imperative for managers of queueing systems to understand how the information that they provide will affect the performance of the system.  To this end, we construct and analyze a two-dimensional deterministic fluid model that incorporates customer choice behavior based on \textbf{delayed} queue length information.  All of the previous literature assumes that all queues have identical parameters and the underlying dynamical system is symmetric.  However, in this paper, we relax this symmetry assumption by allowing the arrival rates, service rates, and the choice model parameters to be different for each queue.  Our methodology exploits the method of multiple scales and asymptotic analysis to understand how to break the symmetry.  We find that the asymmetry can have a large impact on the underlying dynamics of the queueing system.      
\end{abstract}

%**************************************************************************
%**************************************************************************

\section{Introduction} \label{sec_intro}

In many service systems, customers are given information about the system, which can influence their decision to join the queue.  In many of these systems, either the queue length or the waiting time might be given to customers to estimate the time that they might lose waiting to receive service.  Consequently, it is important for service system managers to understand how the information that they provide to customers can affect the underlying queueing system's dynamics. The most common way that service systems interact with their customers are through delay announcements. 

Delay announcements provide customers with information about the estimated time that a customer will wait for service.  This is usually important when customers have multiple decisions about what service they might want to receive.  For example, if one takes a trip to Disney World, one has the option to take many different rides as seen by \citet{nirenberg2018impact}.  Thus, the information provided by the company will influence park goers to join the line for specific rides.  Since this issue is quite important, the impact of delay announcements on the performance of queueing systems has been studied quite extensively in the applied probability literature.  See for example work by \citet{armony2004customer, guo2007analysis, hassin2007information, armony2009impact, guo2009impacts, jouini2009queueing, jouini2011call, allon2011impact, allon2011we, whitt1999improving}.  Unfortunately, most of the literature assumes that the information that the customer receives is in real time and is 100\% accurate. 

However, there are two important scenarios where real-time information is unreasonable to assume.  First, in reality, most of the information communicated to customers is through some electronic device.  Generally these devices need time to process the wait time information and send it to the customer.  In both processing and sending the information, it is possible that the information experiences some time lag.  The second scenario is where customers must commit to a queue before physically joining the queue.  In this setting, the information itself is not delayed, but the customer must experience a travel delay.  During this delay, the system state will most certainly change by the time the customer arrives to the queue.  From the perspective of Disney World, customers who use the My Disney app, first see the wait times of the rides, choose a ride, and then walk to that ride.  The time the customer spends walking to the ride represents the time delay.  

Unlike the previous literature, this paper takes on the challenge of trying to understand the impact of giving delayed information to customers and how this impacts the dynamics of the queueing system.  However, we are not the first to consider delayed information in the context of queueing systems.  The impact of delayed information has been studied in previous work by \citet{pender2017queues, pender2018analysis, novitzky2019nonlinear,  penderstochastic}.  However, one crucial assumption in all of the previous work is that all queues have the same arrival rates, service rates, and the same choice model function.  Thus, the previous models have a symmetry that is easy to exploit for mathematical and analysis purposes.  In this paper, we attempt to explore the same delayed information queueing model, but instead where the queues have different arrival rates, service rates and different choice functions, thereby \emph{breaking the symmetry}.  This is a significant challenge as symmetry was the key ingredient in previous analyses of our model.   

To this end, we consider an asymmetric two-dimensional system of delay differential equations in which the customer choice model is informed by delayed queue length information. We apply techniques from asymptotic analysis such as the method of multiple scales to analyze how the stability of the system depends on how delayed the information in the choice model is provided to customers. In doing so, we derive a first-order approximation of the asymmetric equilibrium, the critical delay for a Hopf bifurcation to occur, and the amplitude of oscillations when a limit cycle is born.  To find these quantities, we employ Lindstedt's method to derive an analytic expression for an approximation of the amplitude of limit cycles when the delay is close to the critical delay.  Our analysis provides additional insight into how the asymmetry of our model can impact the performance of the underlying queueing system.

 \subsection{Main Contributions of Paper}

The contributions of this work can be summarized as follows.

\begin{itemize}
\item We analyze an asymmetric two-dimensional fluid model that uses delayed queue length information to inform customers about the queue length at each queue. We show how the asymmetry affects the queueing model's equilibrium.
\item Using the method of multiple scales, we derive an approximation for the critical delay, which determines the location of a Hopf bifurcation, in terms of the queueing model parameters.
\item We derive an asymptotic closed form approximation for the amplitudes of the limit cycles that arise when the delay is larger than the critical delay.
\end{itemize}

%**************************************************************************
%**************************************************************************

\subsection{Organization of Paper}

The remainder of this paper is organized as follows. In Section~\ref{sec_modeling}, we review the symmetric two dimensional queueing model and then introduce the asymmetric model that we will analyze. Section~\ref{sec_equilibrium} derives an expression for the new approximate equilibrium point of the asymmetric system.  We demonstrate through several numerical examples to show the equilibrium changes as the queueing model's parameters are varied. Section~\ref{sec_delta_mod} uses the method of multiple scales and asymptotic analysis to find the critical delay at which the stability of the delay differential equations changes.  We also demonstrate numerically that our approximate critical delays are quite accurate and determine the location of Hopf bifurcations.  In Section~\ref{sec_amplitude}, we use Lindstedt's method to approximate the amplitude of limit cycles when the delay is near the critical delay.  We show through numerical examples that our amplitude approximations are accurate near the critical delay, even with the model asymmetry.  Finally in Section~\ref{sec_conclusion}, we conclude with directions for future research related to this work.

\section{Asymmetric Queueing Model} \label{sec_modeling}

In this section, we describe the asymmetric queueing model that we will analyze in this paper.  In previous literature on queueing systems with delayed information, such as \citet{novitzky2019nonlinear} and \citet{pender2018analysis}, a common assumption is that the queueing system is symmetric. This symmetry was assumed for convenience in the analysis since the analysis of an N-dimensional DDE system can be reduced to a one dimensional DDE.  In this paper, our focus is on understanding the impact of asymmetry on the dynamics of the queuing system. In fact, it can be shown that the asymmetric model does not yield an explicit closed form formula for the equilibrium and the equilibrium can only be written as the solution to a fixed point equation.  This is true even in the two dimensional case.  Thus, the asymmetric model presents significant mathematical challenges that the symmetric model does not.  

The symmetric model used in previous literature consists of two infinite-server queues.  The two queues are coupled through the arrival rate function, which is equal to the product of the arrival rate $\lambda > 0$ and the probabilistic choice model for joining each queue. The choice model that determines the probabilities of joining each queue is based on a Multinomial Logit Model (MNL) \citet{ben1999discrete}, \citet{train2009discrete} that makes the decision off of delayed queue length information, as shown in the following system \citet{novitzky2019limiting}.  Customers are served immediately at each queue at rate $\mu > 0$ and therefore the total departure rate at each queue is queue length times the service rate.  The infinite server queue is widely used as a canonical model that represents the best one can hope for, see for example \citet{iglehart1965limiting, fralix2009infinite, daw2018queues, daw2019distributions, daw2019new}. This is because the infinite server queue is a lower bound for multi-server queues without abandonment.  From a dynamical system perspective, it was shown in \citet{novitzky2019limiting} that it is unnecessary to study fluid models with a finite number of servers as the finite server model can be reduced to an infinite server dynamical system model, with modified different parameters.  The two delay differential equations in the symmetric case are given by the following equations

\begin{align}
\shortdot{q}_1(t) &= \lambda  \cdot \frac{\exp(- \theta q_1(t-\Delta) + \alpha )}{\exp(-\theta  q_1(t-\Delta) + \alpha ) + \exp(-\theta  q_2(t-\Delta) + \alpha)} - \mu  q_1(t) \label{symmetric equation 1}\\
\shortdot{q}_2(t) &= \lambda \cdot \frac{\exp(-\theta  q_2(t-\Delta) + \alpha)}{\exp(-\theta  q_1(t-\Delta) + \alpha ) + \exp(-\theta  q_2(t-\Delta) + \alpha)} - \mu q_2(t) \label{symmetric equation 2}
\end{align}

\noindent where we assume that $q_1(t)$ and $q_2(t)$, which represent the queue lengths as functions of time, start with different initial continuous functions on the interval $[-\Delta, 0].$ One should note that if the two queue lengths in the symmetric model start with identical initial functions, then they will remain the same for all time.  Now we will describe the symmetric model's parameters.  The parameter $\lambda$ represents the arrival rate, which is the rate at which customers arrive to each queue.  The parameter $\mu$ is the service rate at which servers will serve each customer in the system.  The parameter $\theta$ is the customer sensitivity to the queue length.  When the parameter $\theta$ is large, then customers are highly sensitive to the queue length.  In fact, when we let $\theta \to \infty$, the MNL model converges to the indicator function for the smallest queue.  In addition, when we let $\theta \to 0$, the MNL model converges to $\frac{1}{N}$ and the system becomes a system of N independent and uncoupled infinite server queues.  The parameter $\alpha$ is the customer preference parameter.  Whichever queue has the largest preference parameter $\alpha$, then customers are more likely to go to that queue regardless of the queue length. The parameter $\alpha$ may initially seem pointless as it cancels in the the symmetric system, however, we include it for clarity because it will not cancel in the asymmetric system.  With those four model parameters, we can break the symmetry by perturbing the parameters associated with the first queue, yielding the following asymmetric queueing system

\begin{align}
\shortdot{q}_1(t) &= (\lambda + \epsilon \hat{\lambda}) \cdot \frac{\exp(- (\theta + \epsilon \hat{\theta}) q_1(t-\Delta) + (\alpha + \epsilon \hat{\alpha}))}{\exp(-(\theta + \epsilon \hat{\theta}) q_1(t-\Delta) + (\alpha + \epsilon \hat{\alpha})) + \exp(-\theta  q_2(t-\Delta) + \alpha)} \nonumber \\ &- (\mu + \epsilon \hat{\mu}) q_1(t) \label{perturbed equation 1} \\
\shortdot{q}_2(t) &= \lambda \cdot \frac{\exp(-\theta  q_2(t-\Delta) + \alpha)}{\exp(-(\theta + \epsilon \hat{\theta}) q_1(t-\Delta) + (\alpha + \epsilon \hat{\alpha})) + \exp(-\theta q_2(t-\Delta) + \alpha)} - \mu q_2(t) \label{perturbed equation 2}
\end{align}

\noindent where $\epsilon$ is assumed to be a small parameter. This is the asymmetric system that we will be concerned with throughout this paper.

Before we move to the analysis of the asymmetric model, we believe that it is important to observe that the asymmetric model can be viewed as a symmetric model where the parameters are uncertain or random.  Thus, our asymmetric model can be used to provide confidence intervals around the symmetric model when the model parameters are unknown.  This is also useful from a statistical perspective when the model parameters are obtained through some inference analysis and they are not exactly symmetric.  In the age of of uncertainty quantification, the asymmetry analysis provides information about the DDEs with random parameters.

%**************************************************************************
%**************************************************************************

\section{Asymptotic Analysis of the Equilibrium } \label{sec_equilibrium}

In this paper, our goal is to analyze the stability of the queueing system as a function of the model parameters and the delayed information and to approximate the amplitude of the limit cycles near the bifurcation point. In order to understand the stability of the queueing model, we must calculate the equilibrium or an approximation equilibrium for our queueing model.  For the symmetric model, \citet{novitzky2019nonlinear} shows that the symmetric model given in Equations \ref{symmetric equation 1}-\ref{symmetric equation 2} has a unique equilibrium point at $q_1 = q_2 = \frac{\lambda}{2 \mu}$. In this section, we explore the effects that the asymmetry has on this equilibrium point. In doing so, we obtain a first-order (in $\epsilon$) approximation of the equilibrium of our perturbed system which is described in Theorem \ref{equilibrium_theorem}.
\begin{theorem}

\label{equilibrium_theorem}

The system of Equations \ref{perturbed equation 1}-\ref{perturbed equation 2} has an approximate (up to order $\epsilon^2$) equilibrium point at $$\left(q_1^*, q_2^* \right)= \left(\frac{\lambda}{2 \mu} + a\epsilon + O(\epsilon^2), \frac{\lambda}{2 \mu} + b \epsilon + O(\epsilon^2)\right)$$ where $$a = \frac{\lambda \theta + 4 \mu}{4 \mu (\lambda \theta + 2 \mu)} \hat{\lambda} + \frac{- \lambda(\lambda \theta + 4 \mu)}{4 \mu^2 (\lambda \theta + 2 \mu)} \hat{\mu} +  \frac{- \lambda^2}{4 \mu (\lambda \theta + 2 \mu)} \hat{\theta} + \frac{\lambda}{2(\lambda \theta + 2 \mu)} \hat{\alpha}$$ and $$b = \frac{\lambda \theta}{4 \mu (\lambda \theta + 2 \mu)} \hat{\lambda} + \frac{- \lambda^2 \theta}{4 \mu^2 (\lambda \theta + 2 \mu)} \hat{\mu} + \frac{\lambda^2}{4 \mu (\lambda \theta + 2 \mu)} \hat{\theta} + \frac{- \lambda}{2 (\lambda \theta + 2 \mu)} \hat{\alpha}.$$

\label{equilibrium theorem}

\end{theorem}

\begin{proof}
If we substitute in the constants $q_1^*$ and $q_2^*$ for $q_1$ and $q_2$ in the system \ref{symmetric equation 1}-\ref{symmetric equation 2}, respectively, we have that $q_1(t) = q_1(t - \Delta) = \frac{\lambda}{2 \mu} + a \epsilon + O(\epsilon^2)$ and $q_2(t) = q_2(t - \Delta) = \frac{\lambda}{2 \mu} + b \epsilon + O(\epsilon^2)$ which gives us 

\begin{align*}
0 &= (\lambda + \epsilon \hat{\lambda}) \left[ 1 + \exp\left(\epsilon\left(\theta (a-b) + \frac{\lambda}{2 \mu} \hat{\theta} - \hat{\alpha}\right)  + O(\epsilon^2)\right) \right]^{-1} - (\mu + \epsilon \hat{\mu}) \left( \frac{\lambda}{2 \mu} + a \epsilon + O(\epsilon^2) \right)\\
&= (\lambda + \epsilon \hat{\lambda}) \left( \frac{1}{2} - \frac{1}{4} \left( \theta (a - b) + \frac{\lambda}{2 \mu} \hat{\theta} - \hat{\alpha}  \right)\epsilon + O(\epsilon^2)  \right)  - (\mu + \epsilon \hat{\mu}) \left( \frac{\lambda}{2 \mu} + a \epsilon + O(\epsilon^2) \right)\\
\end{align*}

\noindent and

\begin{align*}
0 &= \lambda \left[ 1 + \exp\left( \epsilon \left(  \theta(b-a) - \frac{\lambda}{2 \mu} \hat{\theta} + \hat{\alpha} \right) + O(\epsilon^2) \right) \right]^{-1} - \mu \left( \frac{\lambda}{2 \mu} + b \epsilon + O(\epsilon^2) \right)\\
&= \lambda \left( \frac{1}{2} - \frac{1}{4} \left( \theta (b-a) - \frac{\lambda}{2 \mu} \hat{\theta} + \hat{\alpha}  \right) \epsilon + O(\epsilon^2) \right)- \mu \left( \frac{\lambda}{2 \mu} + b \epsilon + O(\epsilon^2) \right).
\end{align*}

\noindent Matching $O(\epsilon)$ terms, we get a system of two equations with two unknowns $a$ and $b$.

\begin{align*}
0 &= 2 \mu \lambda \theta (b - a)  - 8 \mu^2 a - \lambda^2 \hat{\theta} + 2 \mu \lambda \hat{\alpha} + 4 \mu \hat{\lambda} - 4 \lambda \hat{\mu}\\
0 &= 2 \mu \lambda \theta (a-b) - 8 \mu^2 b + \lambda^2 \hat{\theta} - 2 \mu \lambda \hat{\alpha}
\end{align*}

\noindent Solving this two dimensional system of equations gives us the desired values for $a$ and $b$.

\end{proof}

Our expression for the new approximate equilibrium makes it very easy to understand what happens to the equilibrium when only a single parameter is perturbed. We observe that the $\hat{\lambda}$ terms in $a$ and $b$ are both positive so that a positive perturbation in the arrival rate $\lambda$ will cause the equilibrium for each queue to increase, with the first queue's equilibrium increasing more as the $\hat{\lambda}$ term in $a$ is larger than the corresponding term in $b$. One should note this asymmetry in the arrival rate change.  In fact, this is because the increase in the arrival rate is direct to first queue, but is indirect for the second queue.  By similar reasoning, we see that positively perturbing the service rate $\mu$ will cause the equilibrium corresponding to each queue to decrease, with the first queue's equilibrium decreasing more than that of the second queue.  Thus, we observe the effects of increasing either $\hat{\lambda}$ or $\hat{\mu}$ are not symmetric with respect to each queue, but they have the same sign in this model.   However, if we only perturb either $\theta$ or $\alpha$, we observe that one of the queue length's equilibrium will increase while the other will decrease.  Despite the opposite signs of direction, the magnitude of the change is identical and symmetric.  

%Below we show plots of queue length versus time for various parameter values to demonstrate the shift in the equilibrium due to perturbations.

%%%%\begin{figure}[ht!]
%%%%  \centering
%%%%  \includegraphics[scale=.515]{./Code/Paper_Figures/Equilibrium_sym_Fig_1.pdf}
%%%%\vspace{-40mm}
%%%%\captionsetup{justification=centering,margin=2cm}
%%%%  \caption{$\Delta = .25, \lambda=10, \mu=1, \theta=1, \alpha=0$\\ History function is constant with $q_1 = 4.99$ and $q_2 = 5.01$}
%%%%\end{figure}
%%%%
%%%%\begin{figure}[hb!]
%%%%  \centering
%%%%  \includegraphics[scale=.515]{./Code/Paper_Figures/Equilibrium_sym_Fig_2.pdf}
%%%%\vspace{-40mm}
%%%%\captionsetup{justification=centering,margin=2cm}
%%%%  \caption{$\Delta = .4, \lambda=10, \mu=1, \theta=1, \alpha=0$\\ History function is constant with $q_1 = 4.99$ and $q_2 = 5.01$}
%%%%\end{figure}

\subsection{Numerical Verification of Equilibrium}

In this section, we analyze the validity of Theorem \ref{equilibrium theorem} by plotting several numerical examples.  Below we show plots of queue length versus time to illustrate the shift in the equilibrium due to the perturbations of the model parameters. 

In Figure \ref{Fig1}, we consider the symmetric system for two values of $\Delta$, each of which shows us a qualitatively different behavior of the system as the queue lengths decay in one case and grow in the other. This trend will be discussed more in the following section. In this case, the equilibrium is at $q_1 = q_2 = \frac{\lambda}{2 \mu}$.

In Figure \ref{Fig2}, we consider the perturbed system where the only perturbed parameter is the queueing system's arrival rate $\lambda$. Since we increased the arrival rate for the first queue, it makes sense that the equilibrium for $q_1$ increases. However, we observe that the equilibrium for $q_2$ also increases, to a slightly less extent, and this is due to the fact that the probability of joining the second queue depends on the delayed length of the first queue in a way so that if the delayed length of the first queue increases (which of course happens because we increased the arrival rate into the first queue), then the probability of joining the second queue increases.

In Figure \ref{Fig3}, the only perturbed parameter is the service rate $\mu$. Since we increased the service rate for the first queue, we see that the equilibrium for it decreases. We also see that the equilibrium for the second queue decreases because the probability of joining the second queue decreases when the delayed length of the first queue decreases. 

In Figure \ref{Fig4}, the only perturbed parameter is $\theta$. Perturbing this parameter positively causes the probability of joining the first queue to decrease and the probability of joining the second queue to increase which gives us some intuition for why we see the equilibrium for the first queue decrease and the equilibrium for the second queue increase.

In Figure \ref{Fig5}, the only perturbed parameter is $\alpha$. Perturbing this parameter positively causes the probability of joining the first queue to increase and the probability of joining the second queue to decrease, which causes the equilibrium for the first queue to increase and the equilibrium for the second queue to decrease. 

We see that the effects of perturbing $\alpha$ are the opposite of the effects of perturbing $\theta$. Similarly, perturbing $\lambda$ seems to qualitatively affect the equilibrium in a way opposite to how perturbing $\mu$ does.

In Figure \ref{Fig6}, all four of the aforementioned parameters are perturbed and in this case the resulting behavior depends on how big the permutations are for each parameter. 

We summarized the various values used in these figures along with the resulting equilibrium values and approximations and errors in Table \ref{Table1}. A natural concern is how the error of the equilibrium approximation varies as a parameter is perturbed by varying amounts. We explore this by varying $\hat{\lambda}$ while keeping other parameters fixed. The results can be seen in Table \ref{Table2} and they are plotted on the left side of Figure \ref{Fig_amp_error}. We do the same for $\hat{\mu}$ in Table \ref{Table3} and we plotted those results in the right side of Figure \ref{Fig_amp_error}. In both cases, we see that the error increases as the parameter is perturbed more. This is expected as we expect our approximation to get its best results when the perturbations are small.

%\begin{TAB}{r,1cm,2cm}{|c|c|}{|c|c|c|}
%h & h\\
%h & h \\
%h & h\\
%\end{TAB}

%\begin{table}[]
%\begin{tabular}{lllll}
%\hline\\ [-2.5ex]
%$\lambda$ & $\hat{\lambda}$ & $\mu$ & $\hat{\mu}$ &  \\ [0.5ex]
%\hline
%1         & 1               & 1     & 1                                         &  \\
%          &                 &       &                                           &  \\
%          &                 &       &                                           & 
%\end{tabular}
%\end{table}

\begin{table}[] 
\begin{tabular}{ | l | l | l | l | l | l | l | l | l | l | l | l | l | l | l | l |}
\hline\\ [-2.5ex]
$\lambda$ & $\hat{\lambda}$ & $\mu$ & $\hat{\mu}$ & $\theta$ & $\hat{\theta}$ & $\alpha$ & $\hat{\alpha}$ & $\epsilon$ &  $\hat{q}_1$ & $q_1$ & $\hat{q}_1$ error &  $\hat{q}_2$ &  $q_2$ & $\hat{q}_2$ error \\
\hline
10         & 1               & 1     & 0           &    1      &       0         &     0     &      0  &   0.1     & $5.0292$   &  $5.0290$  & $2 \cdot 10^{-4}$ & 5.0208 & 5.0207 & $1 \cdot 10^{-4}$   \\
\hline
10         & 0               & 1     & 0.1           &    1      &       0         &     0     &      0   & 0.1       & $4.9708$   &  $ 4.9710$ & $2 \cdot 10^{-4}$ & 4.9792 & 4.9793  & $1 \cdot 10^{-4}$  \\
\hline
10         & 0               & 1     & 0           &    1      &       1         &     0     &      0    & 0.1      & $ 4.7917$ &  $4.8000$   & 0.0083 & 5.2084  & 5.2000 & 0.0084 \\
\hline
10         & 0               & 1     & 0           &    1      &       0         &     0     &      1 & 0.1         & $5.0417$ &  $ 5.0417 $   & 0.0000 & 4.9583 & 4.9583 & 0.0000 \\
\hline
10         & 1               & 1     & 0.1          &    1      &       1         &     0     &      1 & 0.1         & $ 4.8333$&  $ 4.8400$   & 0.0067 & 5.1667 & 5.1600 & 0.0067 \\
\hline
\end{tabular}
\caption{The analytical expression for the first-order approximation of equilibrium, $\hat{q}$, compared against numerical integration, $q$, for various parameter values}
\label{Table1}
\end{table}

\begin{table}[]
\begin{center}
\begin{tabular}{| l | l | l | l | l | l | l |}
\hline\\ [-2.5ex]
$\hat{\lambda}$ & $\hat{q}_1$ & $q_1$ & $\hat{q}_1$ error & $\hat{q}_2$ & $q_2$ & $\hat{q}_2$ error\\
\hline
1 & $ 5.0292$ & $5.0290$ & $2 \cdot 10^{-4}$ & 5.0208  &  5.0207  & $1 \cdot 10^{-4}$  \\
\hline 
1.5 & $ 5.0438$  & $5.0435$& $3 \cdot 10^{-4}$ & 5.0313 &  5.0311 & $2 \cdot 10^{-4}$  \\
\hline 
2 & $ 5.0583 $  & $ 5.0578 $ & $5 \cdot 10^{-4}$ &  5.0417 & 5.0413  & $4 \cdot 10^{-4}$  \\
\hline 
2.5 & $ 5.0729 $ & $ 5.0722 $ &  $7 \cdot 10^{-4}$ & 5.0521  &  5.0515 & $6 \cdot 10^{-4}$  \\
\hline 
3 & $ 5.0875$  & $ 5.0864 $ & 0.0011 & 5.0625  &  5.0617 &  $8 \cdot 10^{-4}$   \\
\hline 
3.5 & $ 5.1021 $ & $ 5.1006 $ &    0.0015 & 5.0729 & 5.0719 &   0.0010 \\
\hline 
4 & $ 5.1167 $ & $ 5.1147 $ & 0.0020 & 5.0833 & 5.0820 & 0.0013\\
\hline 
4.5 & $ 5.1313$ & $5.1288 $ & 0.0025 & 5.0938 & 5.0920 & 0.0018\\
\hline 
5 & $ 5.1458$ & $ 5.1429\ $ & 0.0029 & 5.1042  & 5.1020  &  0.0022\\
\hline 
\end{tabular}
\end{center}
\caption{How the approximation of the equilibrium point, $\hat{q}$, compares to numerical integration, $q$, as $\hat{\lambda}$ varies while fixing  $\lambda = 10, \mu = 1, \theta = 1, \alpha = 0, \epsilon = .01, \hat{\mu} = \hat{\theta} = \hat{\alpha} = 0$}
\label{Table2}
\end{table}

\begin{table}[]
\begin{center}
\begin{tabular}{| l | l | l | l | l | l | l |}
\hline\\ [-2.5ex]
$\hat{\mu}$ & $\hat{q}_1$ & $q_1$ & $\hat{q}_1$ error & $\hat{q}_2$ & $q_2$ & $\hat{q}_2$ error\\
\hline
0.1 & $4.9708 $  & $ 4.9710 $ & $2 \cdot 10^{-4}$ &  4.9792  &  4.9793 & $1 \cdot 10^{-4}$\\
\hline 
0.15 & $ 4.9562 $  & $ 4.9566 $& $4 \cdot 10^{-4}$ &  4.9688  & 4.9690  &  $2 \cdot 10^{-4}$\\
\hline 
0.2 & $ 4.9417 $ & $ 4.9423 $ & $6 \cdot 10^{-4}$ &  4.9583  & 4.9588  & $5 \cdot 10^{-4}$\\
\hline 
0.25 & $ 4.9271 $ & $ 4.9281 $ & 0.001 & 4.9479   & 4.9487  & $8 \cdot 10^{-4}$\\
\hline 
0.3 & $ 4.9125 $ & $ 4.9140 $ & 0.0015 & 4.9375   & 4.9386  & 0.0011\\
\hline 
0.35 & $ 4.8979 $ & $ 4.9000  $ & 0.0021 & 4.9271   & 4.9285  & 0.0014\\
\hline 
0.4 & $ 4.8833  $ & $ 4.8860  $ & 0.0027 & 4.9167   & 4.9186  & 0.0019\\
\hline 
0.45 & $ 4.8688 $ & $ 4.8721 $ & 0.0033 &  4.9063  & 4.9087  & 0.0024\\
\hline 
0.5 & $4.8542  $ & $ 4.8583 $ & 0.0041 &   4.8958  &  4.8988 & 0.0030 \\
\hline 
\end{tabular}
\end{center}
\caption{How the approximation of the equilibrium point, $\hat{q}$, compares to numerical integration, $q$, as $\hat{\mu}$ varies while fixing  $\lambda = 10, \mu = 1, \theta = 1, \alpha = 0, \epsilon = .01, \hat{\mu} = \hat{\theta} = \hat{\alpha} = 0$}
\label{Table3}
\end{table}

\begin{figure}[ht!]
\vspace{-45mm}
  %\centering
  \hspace{-21mm}~\includegraphics[scale=.55]{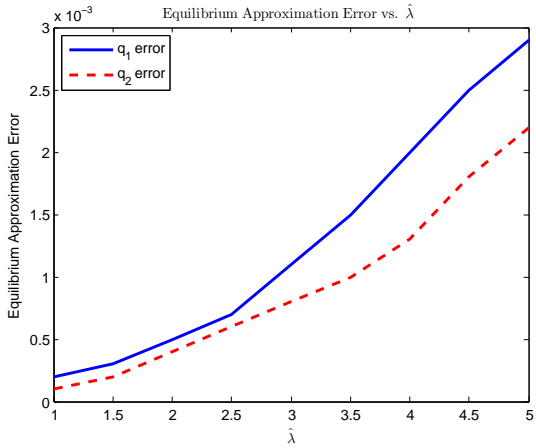}~\hspace{-40mm}~\includegraphics[scale=.55]{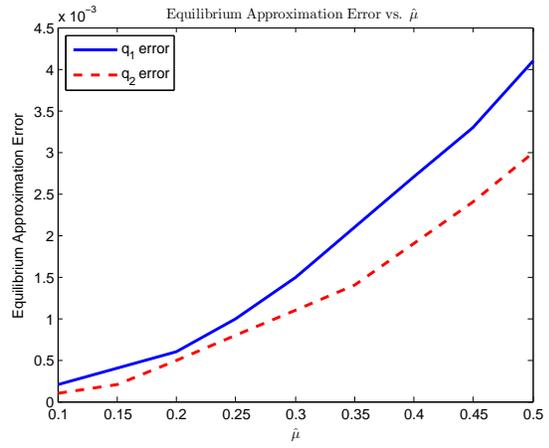}
\captionsetup{justification=centering,margin=2cm}
\vspace{-55mm}
  \caption{Plots of the error of the equilibrium approximation against $\hat{\lambda}$ and $\hat{\mu}$ from Tables \ref{Table2} (Left) and \ref{Table3} (Right)}
\label{Fig_amp_error}
\end{figure}

\begin{figure}[ht!]
\vspace{-35mm}
  %\centering
  \hspace{-5mm}~\includegraphics[scale=.4]{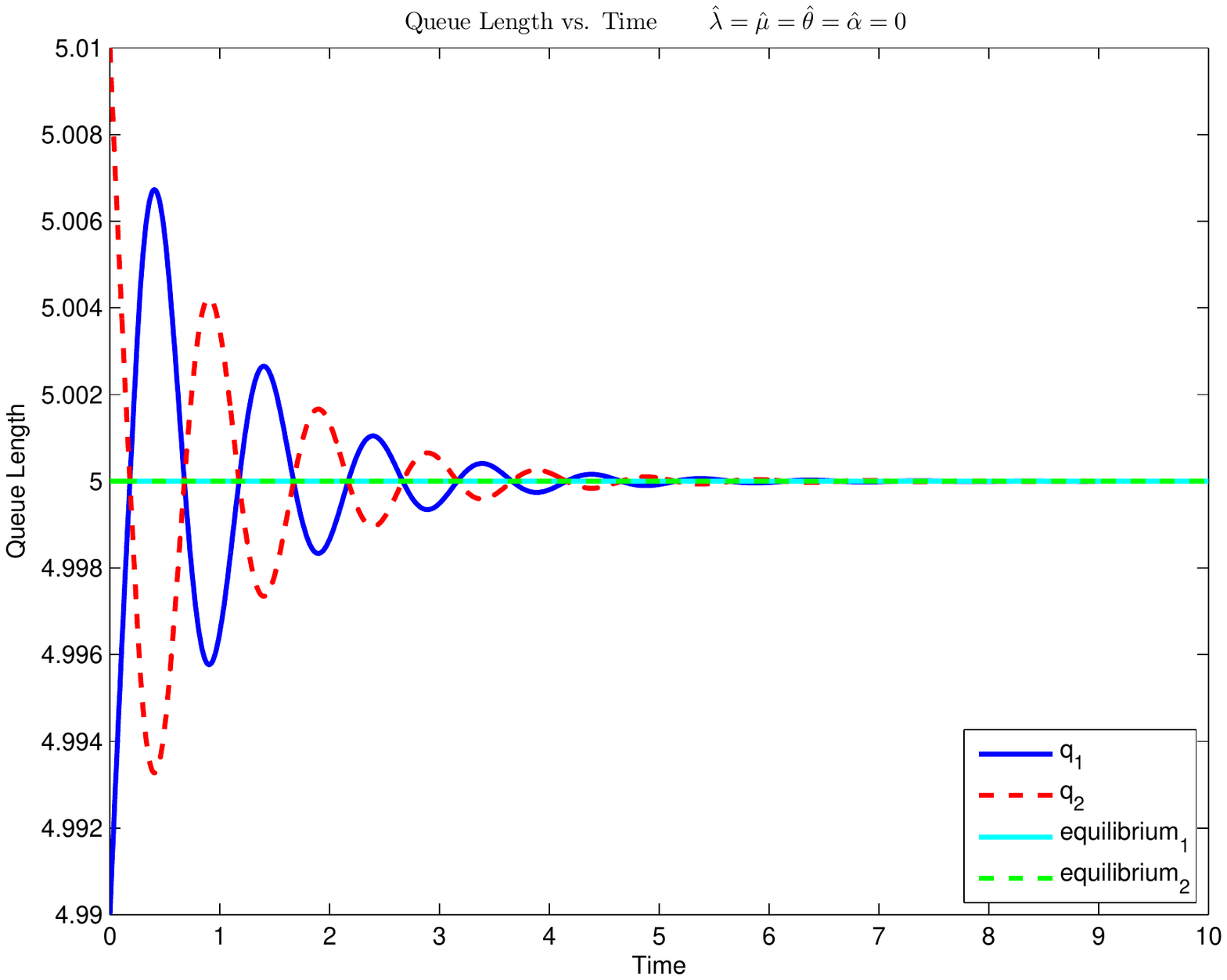}~\hspace{-10mm}~\includegraphics[scale=.4]{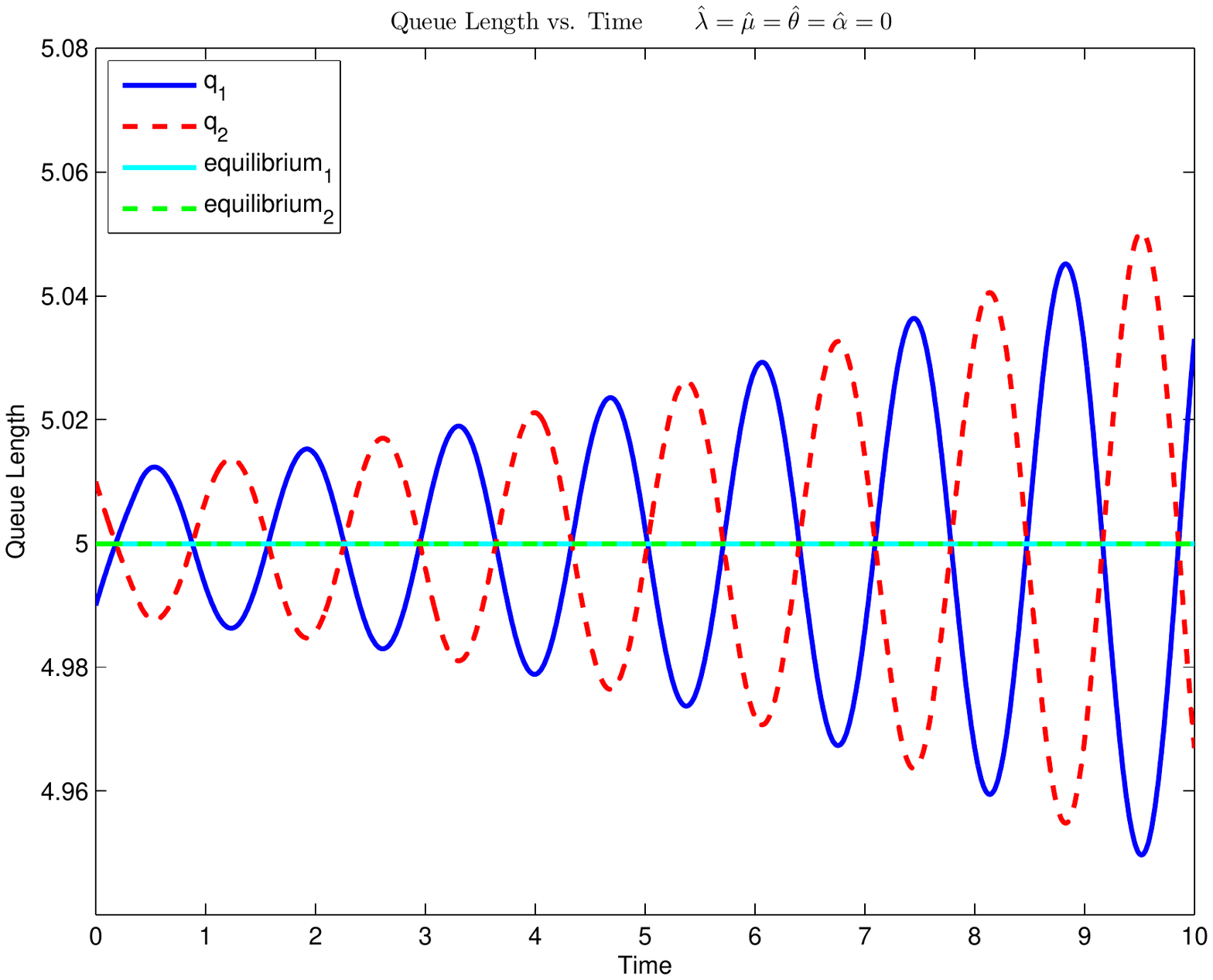}
\captionsetup{justification=centering,margin=2cm}
\vspace{-30mm}
  \caption{$\hat{\lambda} = \hat{\mu} = \hat{\theta} = \hat{\alpha} = 0$, $\lambda=10, \mu=1, \theta=1, \alpha=0$\\ On $[-\Delta, 0]$, $q_1 = 4.99$ and $q_2 = 5.01$,
Left: $\Delta = .25$,
Right: $\Delta = .4$}
\label{Fig1}
\end{figure}

\begin{figure}[ht!]
\vspace{-35mm}
  \hspace{-5mm}~\includegraphics[scale=.4]{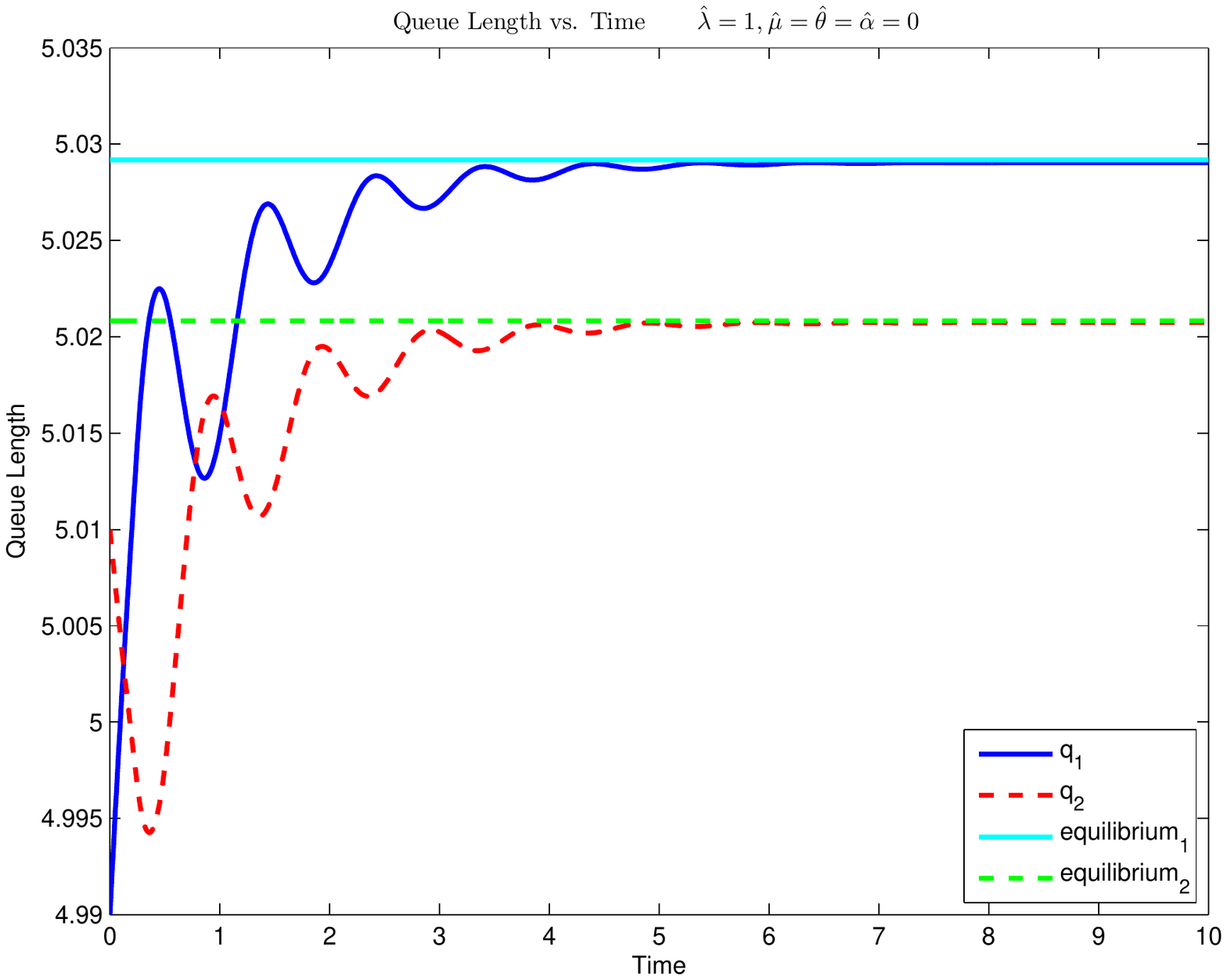}~\hspace{-10mm}~\includegraphics[scale=.4]{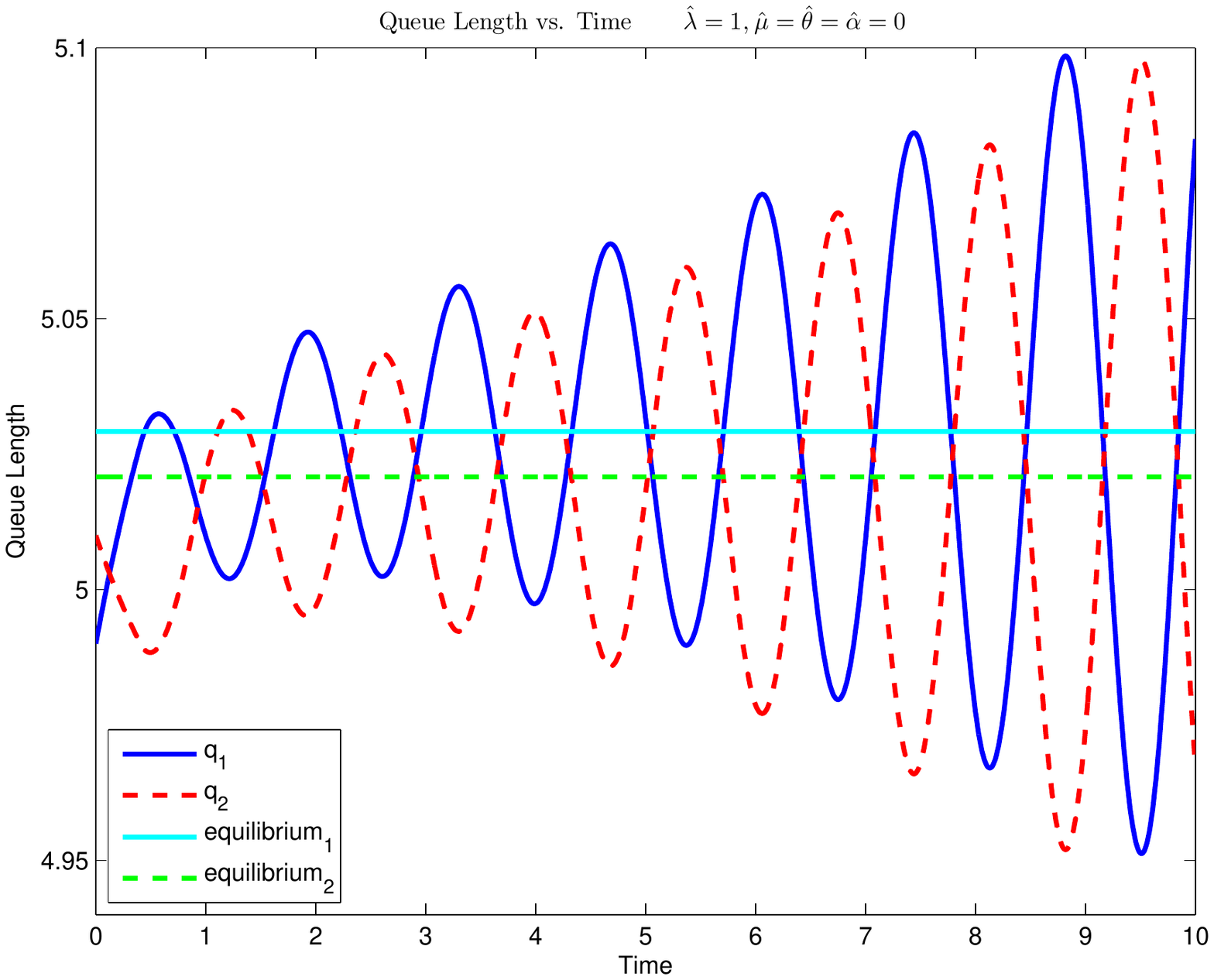}
\captionsetup{justification=centering,margin=2cm}
\vspace{-30mm}
  \caption{$\hat{\lambda} = 1, \hat{\mu} = \hat{\theta} = \hat{\alpha} = 0$, $\epsilon = .1, \lambda=10, \mu=1, \theta=1, \alpha=0$\\ On $[-\Delta, 0]$, $q_1 = 4.99$ and $q_2 = 5.01$,
Left: $\Delta = .25$,
Right: $\Delta = .4$}
\label{Fig2}
\end{figure}

%
%\begin{figure}[ht!]
%  \centering
%\vspace{-55mm}
%  \includegraphics[scale=.7]{./Code/Paper_Figures/Equilibrium_lambda_Fig_1.pdf}
%\vspace{-60mm}
%\captionsetup{justification=centering,margin=2cm}
%  \caption{$\Delta = .25, \epsilon = .1, \lambda=10, \mu=1, \theta=1, \alpha=0$\\ History function is constant with $q_1 = 4.99$ and $q_2 = 5.01$}
%\vspace{-40mm}
%  \includegraphics[scale=.7]{./Code/Paper_Figures/Equilibrium_lambda_Fig_2.pdf}
%\vspace{-60mm}
%  \caption{$\Delta = .4, \epsilon = .1, \lambda=10, \mu=1, \theta=1, \alpha=0$\\ History function is constant with $q_1 = 4.99$ and $q_2 = 5.01$}
%\end{figure}

%%%%%%%%

%\begin{figure}[ht!]
%  \centering
%  \includegraphics[scale=.5]{./Code/Paper_Figures/Equilibrium_mu_Fig_1.pdf}
%\vspace{-40mm}
%\captionsetup{justification=centering,margin=2cm}
%  \caption{$\Delta = .25, \epsilon = .1, \lambda=10, \mu=1, \theta=1, \alpha=0$\\ History function is constant with $q_1 = 4.99$ and $q_2 = 5.01$}
%\end{figure}
%
%\begin{figure}[hb!]
%  \centering
%  \includegraphics[scale=.5]{./Code/Paper_Figures/Equilibrium_mu_Fig_2.pdf}
%\vspace{-40mm}
%\captionsetup{justification=centering,margin=2cm}
%  \caption{$\Delta = .4, \epsilon = .1, \lambda=10, \mu=1, \theta=1, \alpha=0$\\ History function is constant with $q_1 = 4.99$ and $q_2 = 5.01$}
%\end{figure}

%%%%%%%%555

%%%%%%%%%%%%%%%%2 per page
\begin{figure}[ht!]
\vspace{-35mm}
  \hspace{-5mm}~\includegraphics[scale=.4]{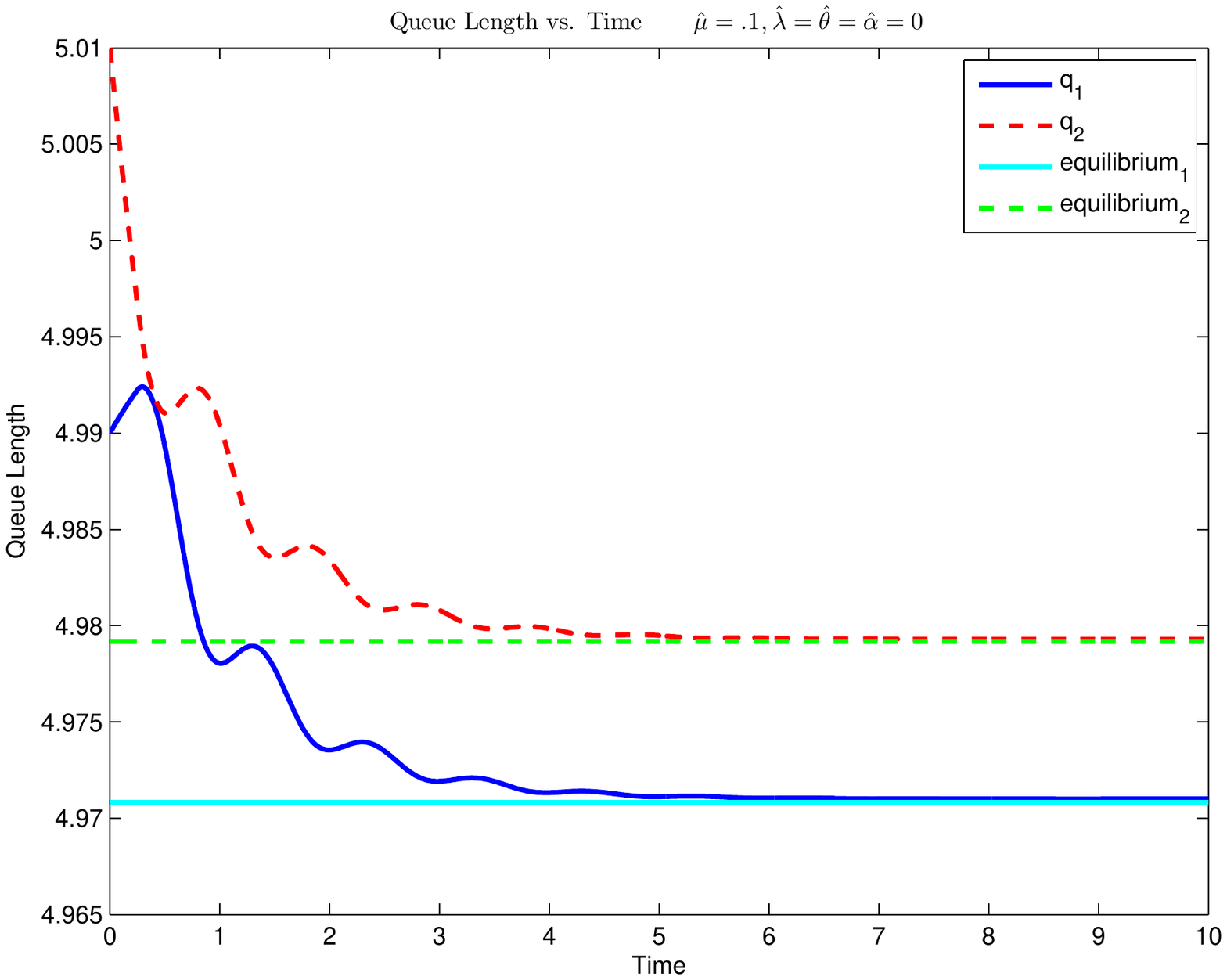}~\hspace{-10mm}~\includegraphics[scale=.4]{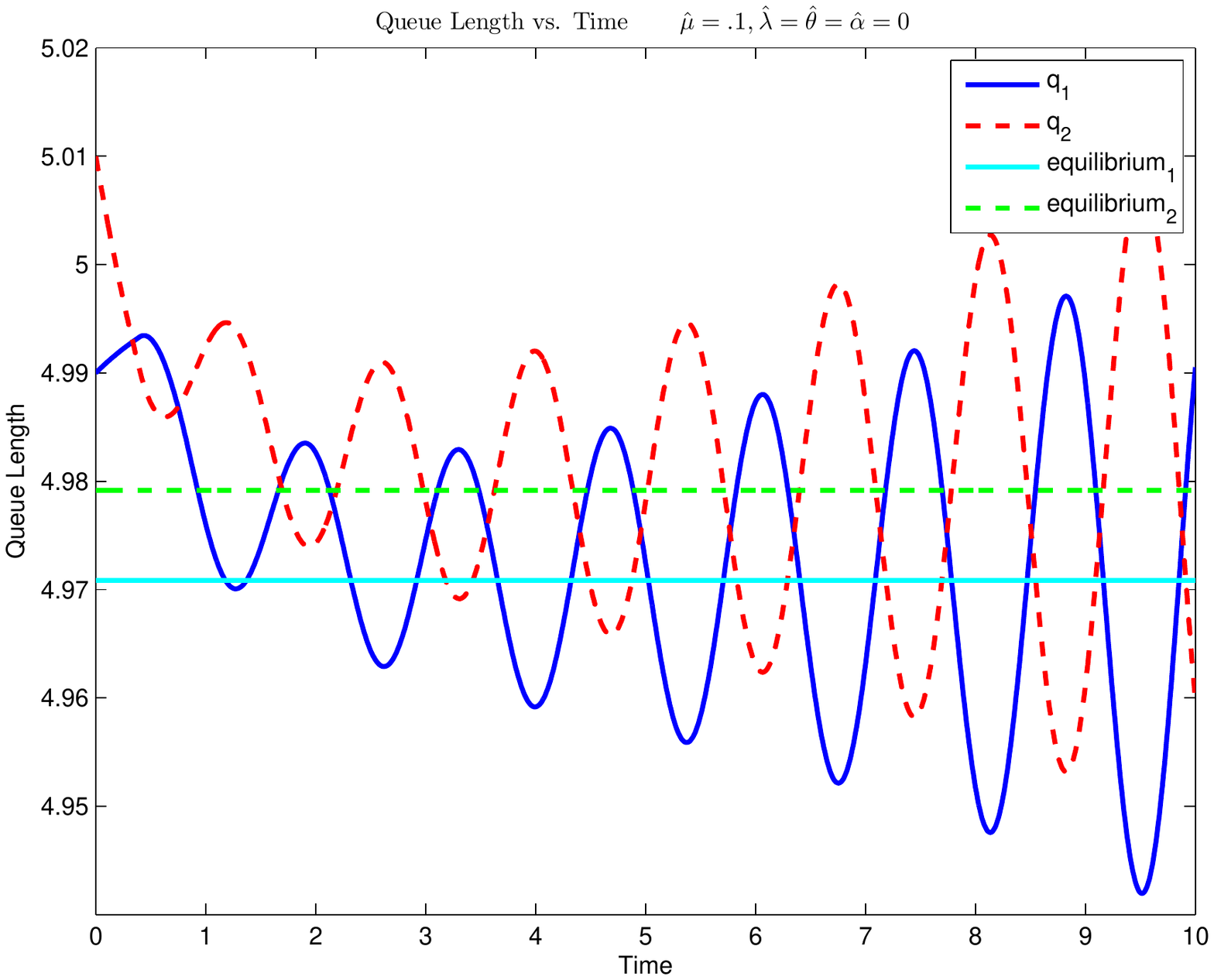}
\captionsetup{justification=centering,margin=2cm}
\vspace{-30mm}
  \caption{$\hat{\mu} = 1, \hat{\lambda}= \hat{\theta} = \hat{\alpha} = 0$, $\epsilon = .1, \lambda=10, \mu=1, \theta=1, \alpha=0$\\ On $[-\Delta, 0]$, $q_1 = 4.99$ and $q_2 = 5.01$,
Left: $\Delta = .25$,
Right: $\Delta = .4$}
\label{Fig3}
\end{figure}

%
%\begin{figure}[ht!]
%  \centering
%\vspace{-55mm}
%  \includegraphics[scale=.7]{./Code/Paper_Figures/Equilibrium_mu_Fig_1.pdf}
%\vspace{-60mm}
%\captionsetup{justification=centering,margin=2cm}
%  \caption{$\Delta = .25, \epsilon = .1, \lambda=10, \mu=1, \theta=1, \alpha=0$\\ History function is constant with $q_1 = 4.99$ and $q_2 = 5.01$}
%\vspace{-40mm}
%  \includegraphics[scale=.7]{./Code/Paper_Figures/Equilibrium_mu_Fig_2.pdf}
%\vspace{-60mm}
%  \caption{$\Delta = .4, \epsilon = .1, \lambda=10, \mu=1, \theta=1, \alpha=0$\\ History function is constant with $q_1 = 4.99$ and $q_2 = 5.01$}
%\end{figure}

%%%%%%%%%%%%

%\begin{figure}[ht!]
%  \centering
%  \includegraphics[scale=.5]{./Code/Paper_Figures/Equilibrium_theta_Fig_1.pdf}
%\vspace{-40mm}
%\captionsetup{justification=centering,margin=2cm}
%  \caption{$\Delta = .25, \epsilon = .1, \lambda=10, \mu=1, \theta=1, \alpha=0$\\ History function is constant with $q_1 = 4.99$ and $q_2 = 5.01$}
%\end{figure}
%
%\begin{figure}[hb!]
%  \centering
%  \includegraphics[scale=.5]{./Code/Paper_Figures/Equilibrium_theta_Fig_2.pdf}
%\vspace{-40mm}
%\captionsetup{justification=centering,margin=2cm}
%  \caption{$\Delta = .4, \epsilon = .1, \lambda=10, \mu=1, \theta=1, \alpha=0$\\ History function is constant with $q_1 = 4.99$ and $q_2 = 5.01$}
%\end{figure}

%%%%%%%%%%%%%%%%%%%

%%%%%%%%%%%%%%%%2 per page
\begin{figure}[ht!]
\vspace{-35mm}
  \hspace{-5mm}~\includegraphics[scale=.4]{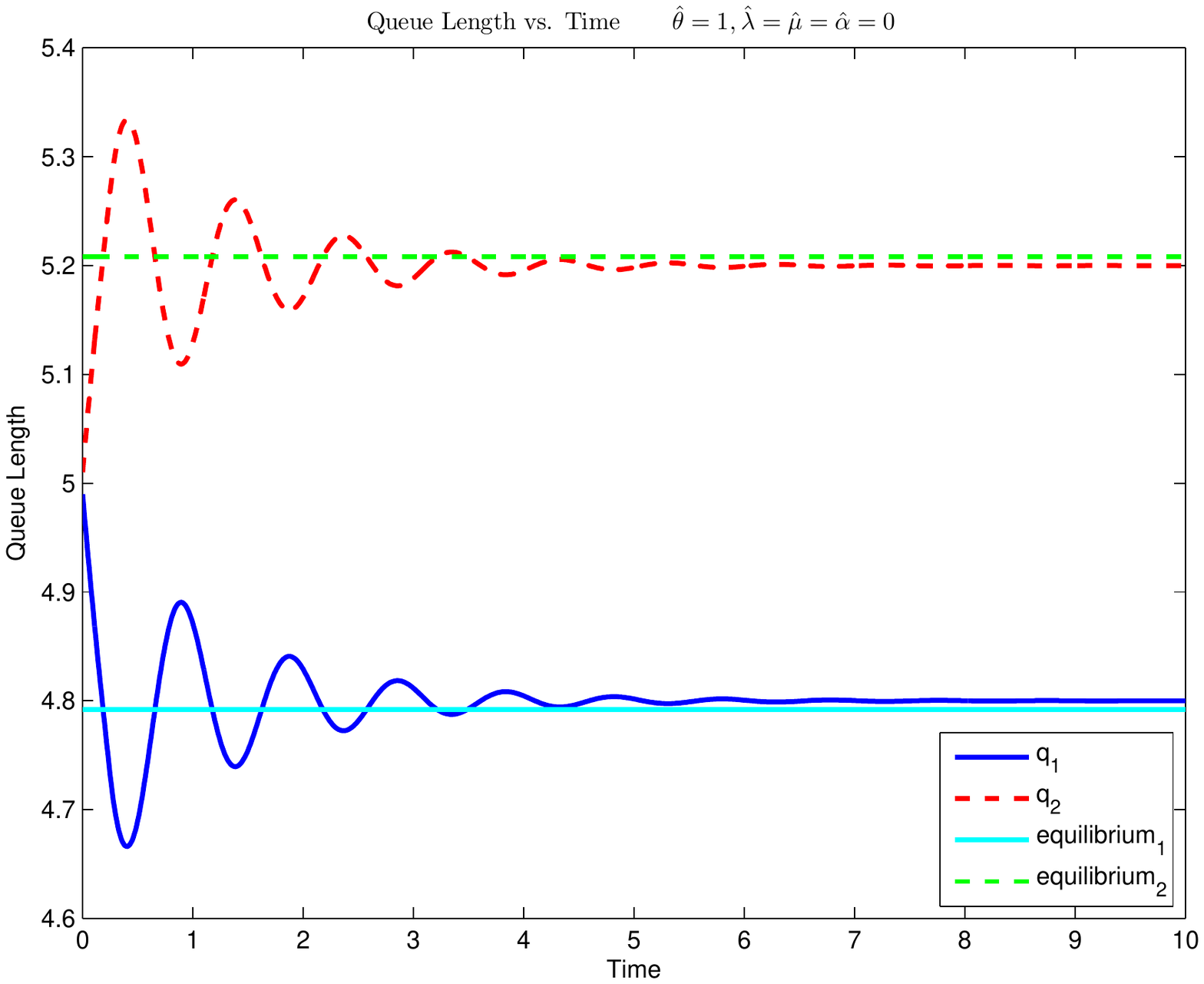}~\hspace{-10mm}~\includegraphics[scale=.4]{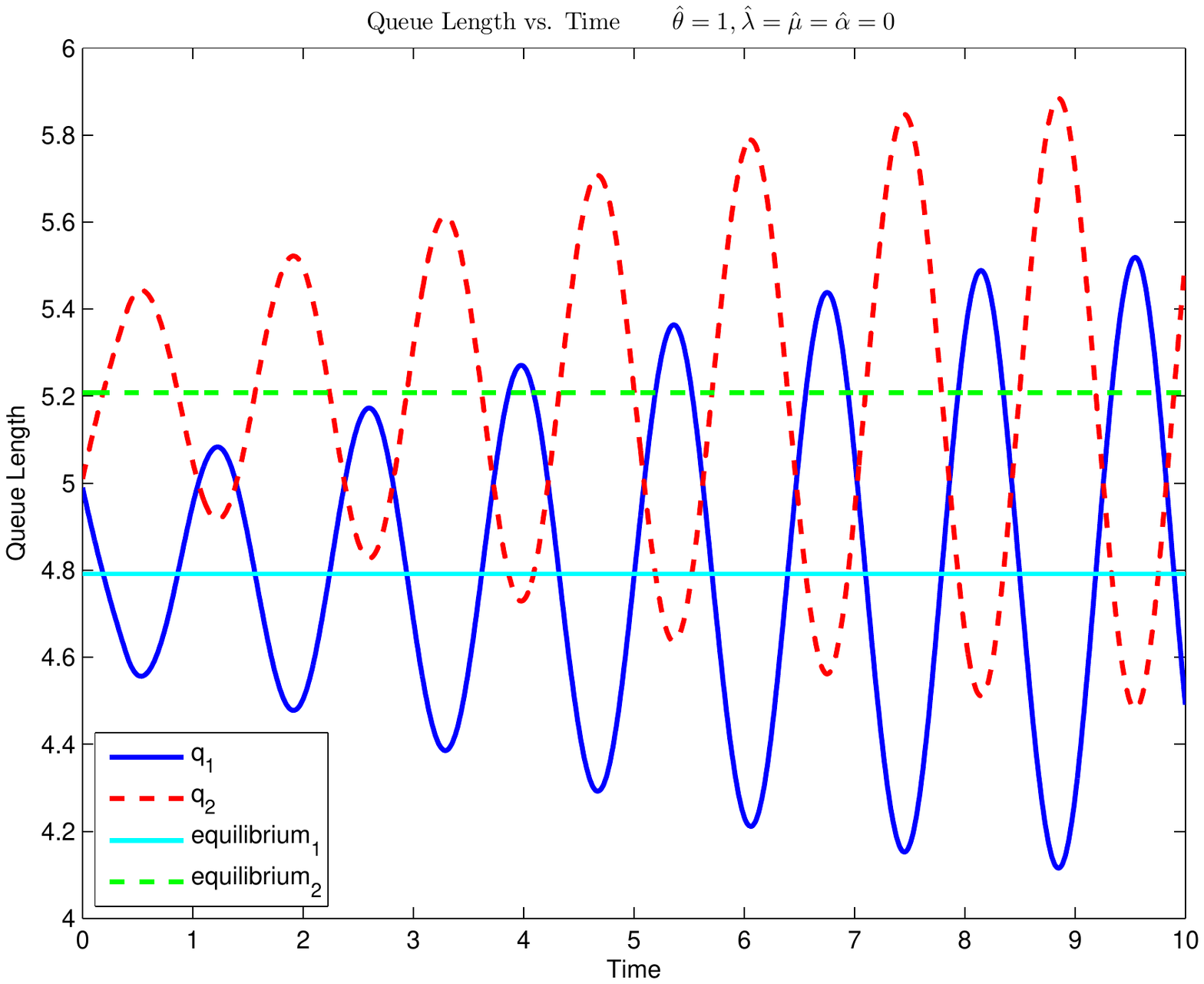}
\captionsetup{justification=centering,margin=2cm}
\vspace{-30mm}
  \caption{$\hat{\theta}=1, \hat{\lambda} = \hat{\mu}  = \hat{\alpha} = 0$, $\epsilon = .1, \lambda=10, \mu=1, \theta=1, \alpha=0$\\ On $[-\Delta, 0]$, $q_1 = 4.99$ and $q_2 = 5.01$,
Left: $\Delta = .25$,
Right: $\Delta = .4$}
\label{Fig4}
\end{figure}

%\begin{figure}[ht!]
%  \centering
%\vspace{-55mm}
%  \includegraphics[scale=.7]{./Code/Paper_Figures/Equilibrium_theta_Fig_1.pdf}
%\vspace{-60mm}
%\captionsetup{justification=centering,margin=2cm}
%  \caption{$\Delta = .25, \epsilon = .1, \lambda=10, \mu=1, \theta=1, \alpha=0$\\ History function is constant with $q_1 = 4.99$ and $q_2 = 5.01$}
%\vspace{-40mm}
%  \includegraphics[scale=.7]{./Code/Paper_Figures/Equilibrium_theta_Fig_2.pdf}
%\vspace{-60mm}
%  \caption{$\Delta = .4, \epsilon = .1, \lambda=10, \mu=1, \theta=1, \alpha=0$\\ History function is constant with $q_1 = 4.99$ and $q_2 = 5.01$}
%\end{figure}

%%%%%%%%%%%%%%%%%%%
%
%\begin{figure}[ht!]
%  \centering
%  \includegraphics[scale=.5]{./Code/Paper_Figures/Equilibrium_alpha_Fig_1.pdf}
%\vspace{-40mm}
%\captionsetup{justification=centering,margin=2cm}
%  \caption{$\Delta = .25, \epsilon = .1, \lambda=10, \mu=1, \theta=1, \alpha=0$\\ History function is constant with $q_1 = 4.99$ and $q_2 = 5.01$}
%\end{figure}
%
%\begin{figure}[hb!]
%  \centering
%  \includegraphics[scale=.5]{./Code/Paper_Figures/Equilibrium_alpha_Fig_2.pdf}
%\vspace{-40mm}
%\captionsetup{justification=centering,margin=2cm}
%  \caption{$\Delta = .4, \epsilon = .1, \lambda=10, \mu=1, \theta=1, \alpha=0$\\ History function is constant with $q_1 = 4.99$ and $q_2 = 5.01$}
%\end{figure}

%%%%%%%%%%%%%%%%%%%%%55
%%%%%%%%%%%%%%%%2 per page
\begin{figure}[ht!]
\vspace{-35mm}
  \hspace{-5mm}~\includegraphics[scale=.4]{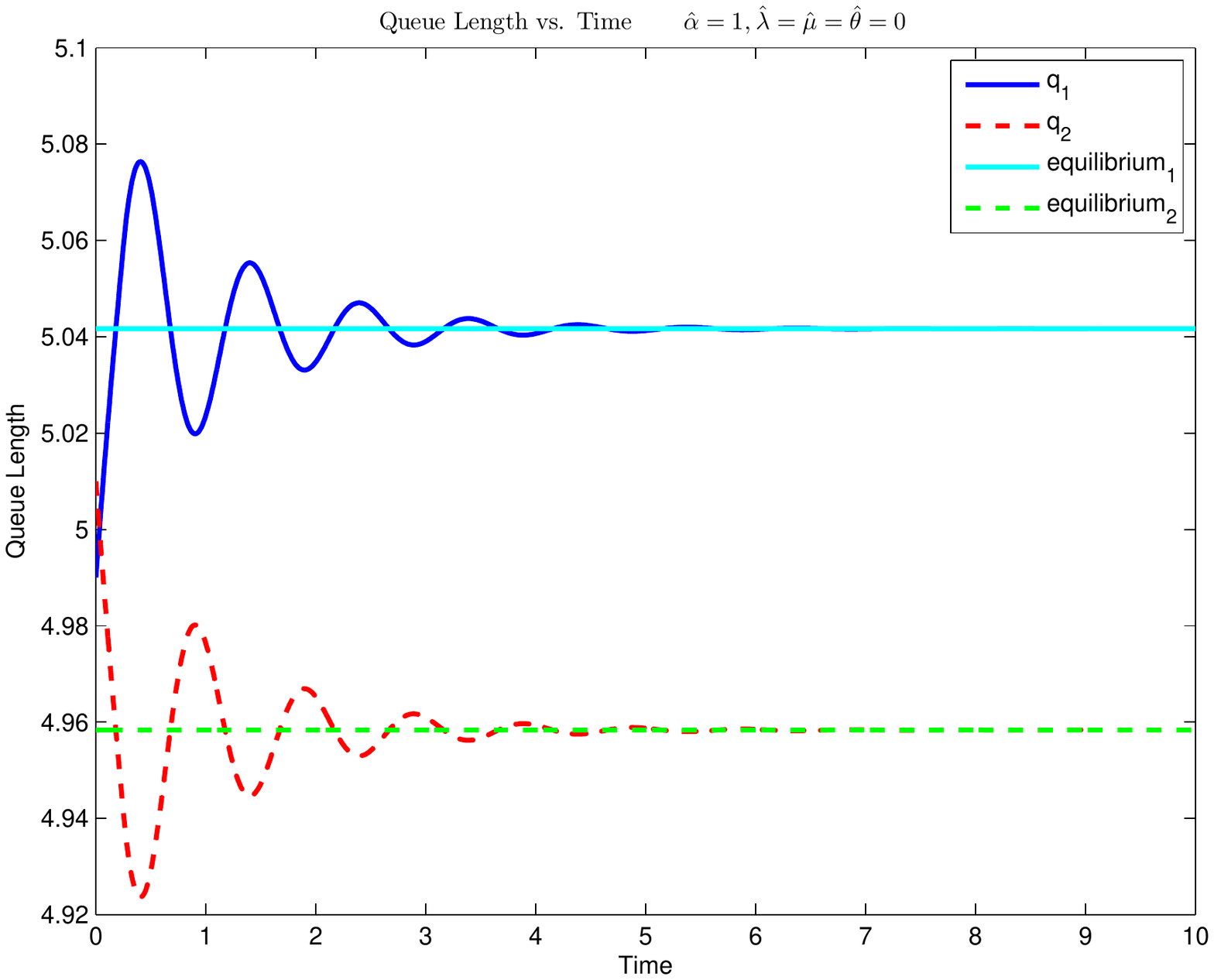}~\hspace{-10mm}~\includegraphics[scale=.4]{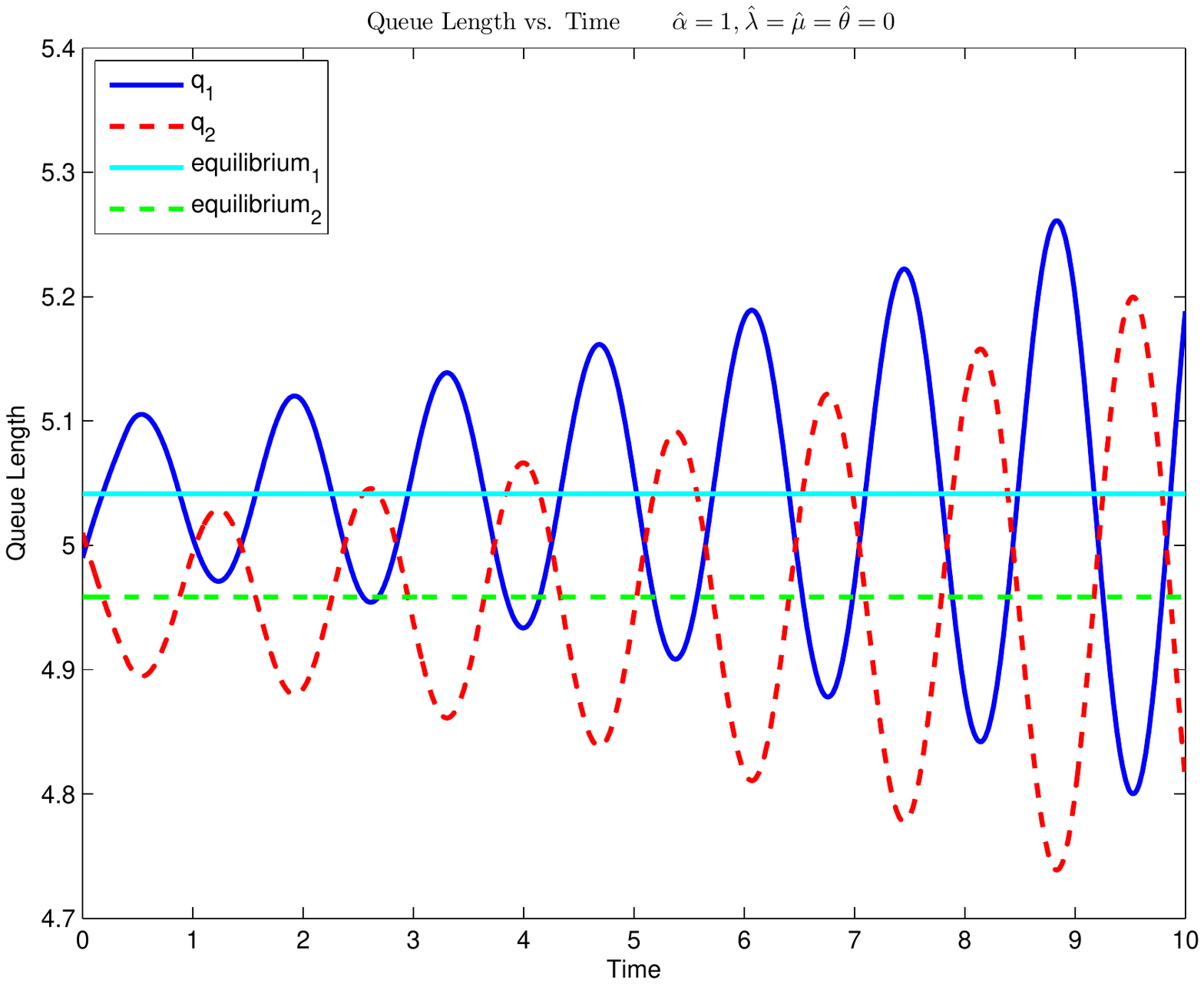}
\captionsetup{justification=centering,margin=2cm}
\vspace{-30mm}
  \caption{$\hat{\alpha} = 1, \hat{\lambda} = \hat{\mu} = \hat{\theta} = 0$, $\epsilon = .1, \lambda=10, \mu=1, \theta=1, \alpha=0$\\ On $[-\Delta, 0]$, $q_1 = 4.99$ and $q_2 = 5.01$,
Left: $\Delta = .25$,
Right: $\Delta = .4$}
\label{Fig5}
\end{figure}

%
%\begin{figure}[ht!]
%  \centering
%\vspace{-55mm}
%  \includegraphics[scale=.7]{./Code/Paper_Figures/Equilibrium_alpha_Fig_1.pdf}
%\vspace{-60mm}
%\captionsetup{justification=centering,margin=2cm}
%  \caption{$\Delta = .25, \epsilon = .1, \lambda=10, \mu=1, \theta=1, \alpha=0$\\ History function is constant with $q_1 = 4.99$ and $q_2 = 5.01$}
%\vspace{-40mm}
%  \includegraphics[scale=.7]{./Code/Paper_Figures/Equilibrium_alpha_Fig_2.pdf}
%\vspace{-60mm}
%  \caption{$\Delta = .4, \epsilon = .1, \lambda=10, \mu=1, \theta=1, \alpha=0$\\ History function is constant with $q_1 = 4.99$ and $q_2 = 5.01$}
%\end{figure}

%%%%%%%%%%%%%%%%%%%%

%\begin{figure}[ht!]
%  \centering
%  \includegraphics[scale=.5]{./Code/Paper_Figures/Equilibrium_all_Fig_1.pdf}
%\vspace{-40mm}
%\captionsetup{justification=centering,margin=2cm}
%  \caption{$\Delta = .25, \epsilon = .1, \lambda=10, \mu=1, \theta=1, \alpha=0$\\ History function is constant with $q_1 = 4.99$ and $q_2 = 5.01$}
%\end{figure}
%
%\begin{figure}[hb!]
%  \centering
%  \includegraphics[scale=.5]{./Code/Paper_Figures/Equilibrium_all_Fig_2.pdf}
%\vspace{-40mm}
%\captionsetup{justification=centering,margin=2cm}
%  \caption{$\Delta = .4, \epsilon = .1, \lambda=10, \mu=1, \theta=1, \alpha=0$\\ History function is constant with $q_1 = 4.99$ and $q_2 = 5.01$}
%\end{figure}

%%%%%%%%%5
%%%%%%%%%%%%%%%%2 per page

\clearpage
\clearpage
\clearpage
\clearpage
\clearpage
\clearpage
\clearpage
\clearpage
\clearpage
\clearpage
\clearpage
\clearpage
\clearpage
\clearpage
\clearpage
\clearpage
\clearpage
\clearpage
\clearpage
\clearpage
\clearpage
\clearpage
\clearpage
\clearpage
\clearpage
\clearpage

\begin{figure}[ht!]
\vspace{-35mm}
  \hspace{-5mm}~\includegraphics[scale=.4]{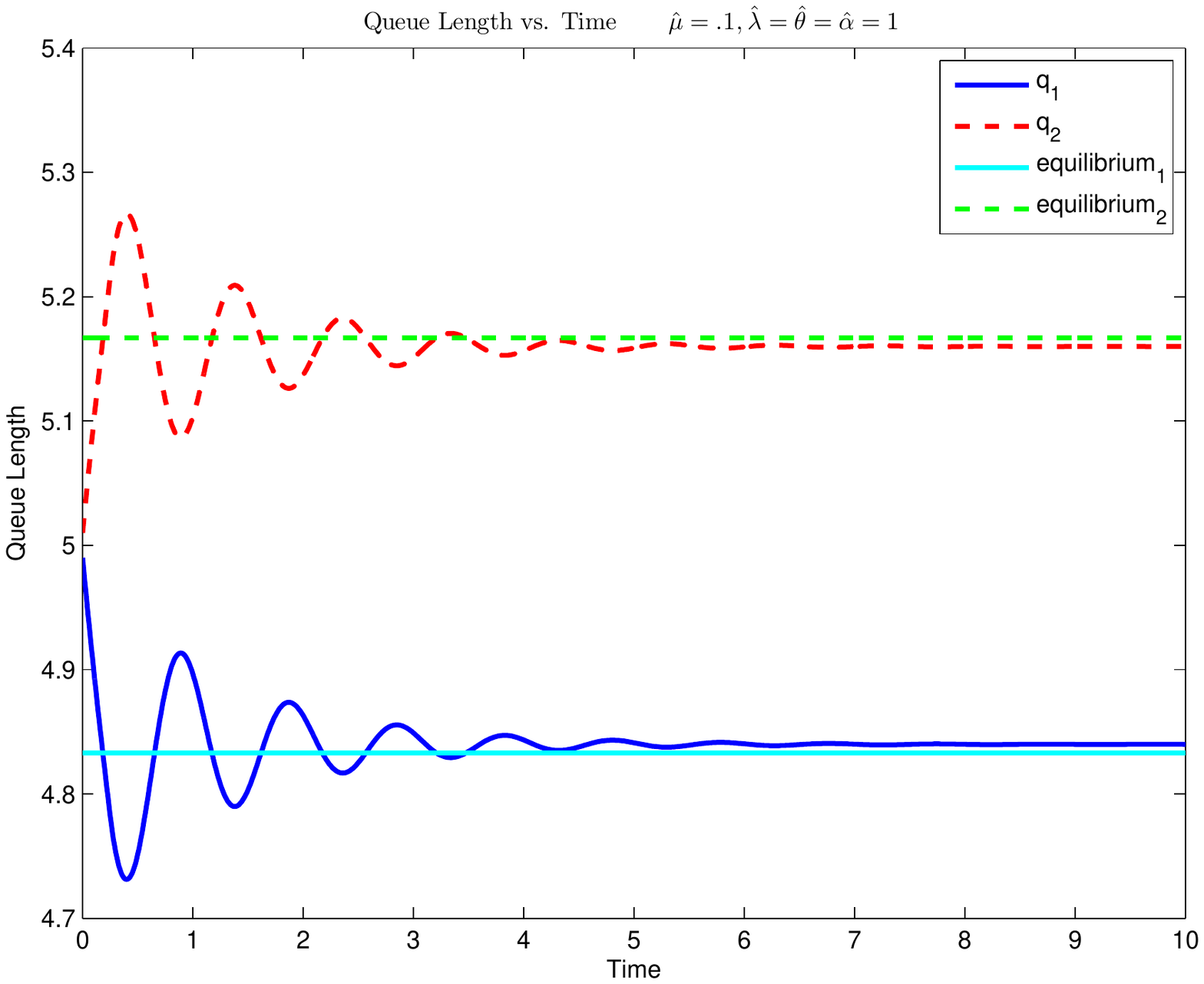}~\hspace{-10mm}~\includegraphics[scale=.4]{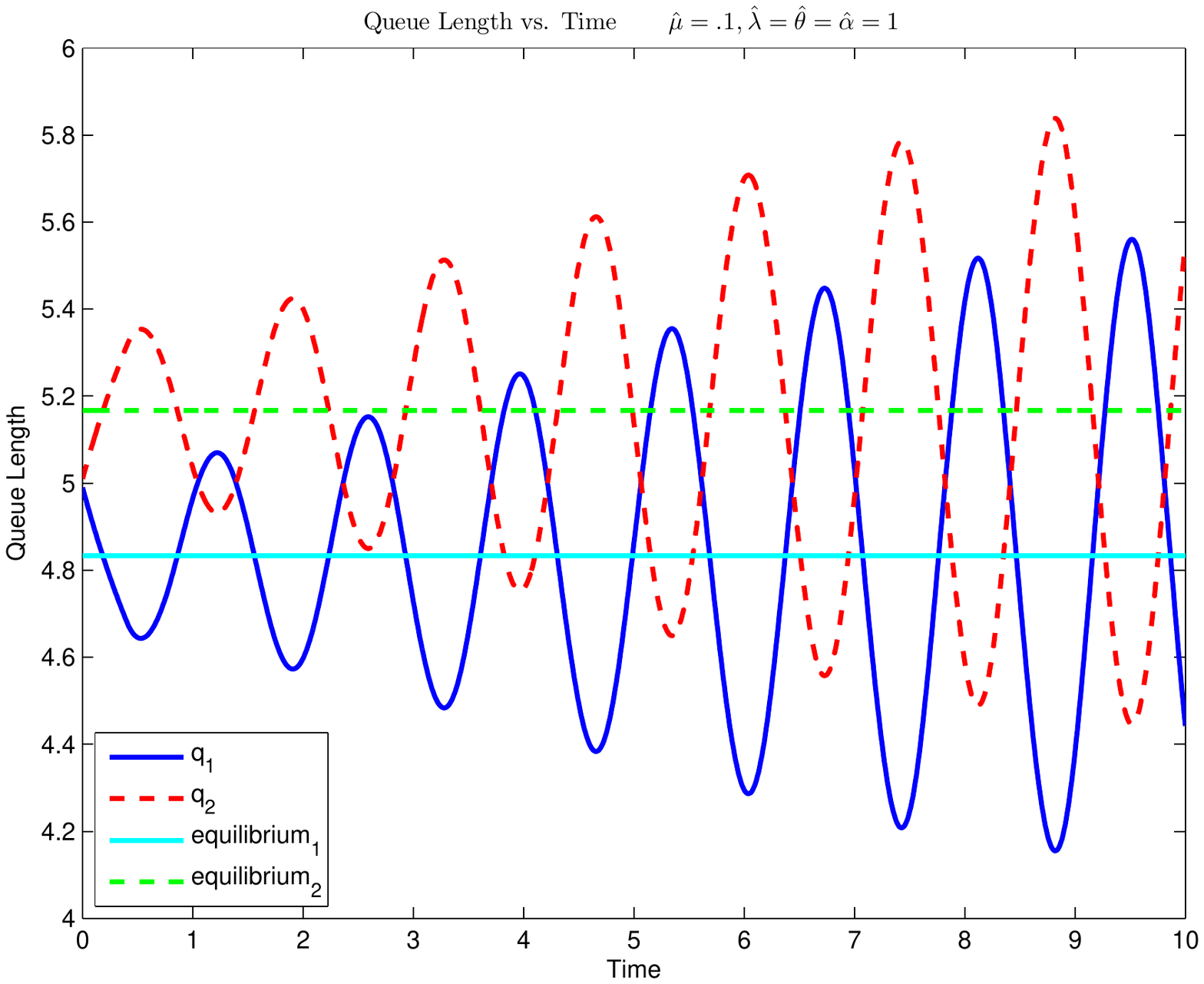}
\captionsetup{justification=centering,margin=2cm}
\vspace{-30mm}
  \caption{$\hat{\mu} = 0.1, \hat{\lambda} = \hat{\theta} = \hat{\alpha} = 1$, $\epsilon = .1, \lambda=10, \mu=1, \theta=1, \alpha=0$\\ On $[-\Delta, 0]$, $q_1 = 4.99$ and $q_2 = 5.01$,
Left: $\Delta = .25$,
Right: $\Delta = .4$}
\label{Fig6}
\end{figure}

%
%\begin{figure}[ht!]
%  \centering
%\vspace{-55mm}
%  \includegraphics[scale=.7]{./Code/Paper_Figures/Equilibrium_all_Fig_1.pdf}
%\vspace{-60mm}
%\captionsetup{justification=centering,margin=2cm}
%  \caption{$\Delta = .25, \epsilon = .1, \lambda=10, \mu=1, \theta=1, \alpha=0$\\ History function is constant with $q_1 = 4.99$ and $q_2 = 5.01$}
%\vspace{-40mm}
%  \includegraphics[scale=.7]{./Code/Paper_Figures/Equilibrium_all_Fig_2.pdf}
%\vspace{-60mm}
%  \caption{$\Delta = .4, \epsilon = .1, \lambda=10, \mu=1, \theta=1, \alpha=0$\\ History function is constant with $q_1 = 4.99$ and $q_2 = 5.01$}
%\end{figure}

%%%%%%%%
%**********************************************************************************************
%**********************************************************************************************

\section{Hopf Bifurcation of Asymmetric Model} \label{sec_delta_mod}

Now that we have derived an approximate equilibrium for our asymmetric queueing model, we can now analyze the stability of this approximate equilibrium. In the symmetric model,  \citet{novitzky2019nonlinear} shows that if $\lambda \theta > 2 \mu$, then the symmetric system given in Equations \ref{symmetric equation 1}-\ref{symmetric equation 2} will exhibit a Hopf bifurcation for values of $\Delta > \Delta_{\text{cr}}$ where $$\Delta_{\text{cr}} = \frac{ \arccos\left( \frac{-2 \mu}{\lambda \theta} \right)}{\omega_{\text{cr}}} \hspace{5mm} \text{ and } \hspace{5mm} \omega_{\text{cr}} = \frac{1}{2} \sqrt{\lambda^2 \theta^2 - 4 \mu^2}.$$ Our goal in this section is to derive an analogous critical delay expression for the asymmetric model, which we will denote as $\Delta_{\text{mod}}$. We will show that the new critical delay, $\Delta_{\text{mod}}$, marks a change in stability for the queueing model and we verify this result using numerical integration of DDEs.  This means that we show from numerical integration that a limit cycle is born at this modified critical value of $\Delta_{mod}$. Our analysis for deriving the approximate critical delay makes use of the method of multiple scales.  We show this result below in Theorem \ref{hopf_eqn}

\begin{theorem} \label{hopf_eqn}
If $ \lambda \theta> 2 \mu$, then, for sufficiently small $\epsilon$, the stability of the queueing system given in Equations \ref{perturbed equation 1}-\ref{perturbed equation 2} changes when $\Delta = \Delta_{\text{mod}}$ where 

\begin{align*}
\Delta_{\text{mod}} &= \Delta_{\text{cr}} - \epsilon \left( \frac{\mu + \Delta_{\text{cr}}(\mu^2 + \omega_{\text{cr}}^2)}{2 \lambda \omega_{\text{cr}}^2} \hat{\lambda} - \frac{1 + \mu \Delta_{\text{cr}}}{2 \omega_{\text{cr}}^2} \hat{\mu} + \frac{\mu + \Delta_{\text{cr}}(\mu^2 + \omega_{\text{cr}}^2)}{2 \theta \omega_{\text{cr}}^2} \hat{\theta} \right) + O(\epsilon^2).
\end{align*}

\label{delta mod theorem}

\end{theorem}

\begin{proof}

We begin by linearizing the system of Equations \ref{perturbed equation 1}-\ref{perturbed equation 2} about the approximate equilibrium point $$(q_1^*, q_2^*) = \left(\frac{\lambda}{2 \mu} + a\epsilon + O(\epsilon^2), \frac{\lambda}{2 \mu} + b \epsilon + O(\epsilon^2)\right)$$ where $a$ and $b$ are as defined in Theorem 2.1. In doing so, we introduce the functions $\tilde{u}_1(t)$ and $\tilde{u}_2(t)$ so that $$q_1(t) = q_1^* + \tilde{u}_1(t), \hspace{5mm} q_2(t) = q_2^* + \tilde{u}_2(t)$$ and we approximate $\overset{\bullet}{\tilde{u}}_1$ and $\overset{\bullet}{\tilde{u}}_2$ by a linear Taylor expansion about the equilibrium point $\tilde{u}_1(t) = \tilde{u}_2(t) = \tilde{u}_1(t - \Delta) = \tilde{u}_2(t - \Delta) = 0$ and we denote the linear approximations by $\overset{\bullet}{u}_1(t)$ and $\overset{\bullet}{u}_2(t)$, respectively, and we Taylor expand coefficients with nonlinear dependence on $\epsilon$ about $\epsilon = 0$ and neglect terms that are $O(\epsilon^2)$ yielding the following first-order (in $\epsilon$) approximation of the linear system.

\begin{align}
\overset{\bullet}{u}_1(t) &= \frac{(\lambda + \hat{\lambda} \epsilon) \theta}{4}\left[ u_2(t - \Delta) - u_1(t - \Delta) \right] - \frac{\lambda \hat{\theta}}{4} \epsilon u_1(t - \Delta) - (\mu + \hat{\mu} \epsilon) u_1(t) \label{u1 equation 45}\\
\overset{\bullet}{u}_2(t) &= \frac{\lambda \theta}{4} [u_1(t- \Delta) - u_2(t - \Delta)] + \frac{\lambda \hat{\theta}}{4} \epsilon u_1(t - \Delta) - \mu u_2(t) \label{u2 equation 46}
\end{align}

\noindent We then proceed by making the change of variables $$v_1(t) = u_1(t) + u_2(t), \hspace{5mm} v_2(t) = u_1(t) - u_2(t)$$ to get the following system

\begin{align}
\overset{\bullet}{v}_1(t) + \left( \mu + \frac{\hat{\mu}}{2} \epsilon \right) v_1(t) &= - \frac{\theta \hat{\lambda}}{4}\epsilon v_2(t - \Delta) - \frac{\hat{\mu}}{2} \epsilon v_2(t) \label{v1 equation 1}\\
\overset{\bullet}{v}_2(t) + \left( \frac{\lambda \theta}{2} + \frac{\theta \hat{\lambda}}{4} \epsilon + \frac{\lambda \hat{\theta}}{4} \epsilon \right) v_2(t - \Delta) + \left( \mu + \frac{\hat{\mu}}{2} \epsilon \right) v_2(t) &= - \frac{\lambda \hat{\theta}}{4} \epsilon v_1(t - \Delta) - \frac{\hat{\mu}}{2} \epsilon v_1(t). \label{v2 equation 1}
\end{align}

\noindent Before proceeding, we introduce new variables $$\xi = t, \hspace{5mm} \eta = \epsilon t$$ to represent a regular time and a slow time, respectively, so we have $$v_i(t) = v_i(\xi, \eta) \hspace{5mm} \text{ and } \hspace{5mm} v_i(t - \Delta) = v_i(\xi - \Delta, \eta - \epsilon \Delta) $$ and the derivatives become $$ \overset{\bullet}{v}_i(t) = \frac{d v_i}{dt} = \frac{\partial  u_i}{\partial \xi} \frac{d \xi}{d t} + \frac{\partial u_i}{\partial \eta} \frac{d \eta}{d t} = \frac{\partial v_i}{\partial \xi} + \epsilon \frac{\partial v_i}{\partial \eta}, \hspace{5mm} i = 1,2. $$ In addition to this change of variables, we expand our functions and detune our delay from the critical delay for the symmetric system as follows. $$v_1(t) = v_{1,0}(t) + \epsilon v_{1,1}(t) + O(\epsilon^2)$$ $$v_2(t) = v_{2,0}(t) + \epsilon v_{2,1}(t) + O(\epsilon^2)$$ $$\Delta = \Delta_{\text{cr}} + \epsilon \Delta_1 + O(\epsilon^2)$$ Taylor expanding our delayed terms yields $$v_i(t-\Delta) = v_i(\xi - \Delta, \eta - \epsilon \Delta) = \bar{v}_i - \epsilon \left( \Delta_1 \frac{\partial \bar{v}_i}{\partial \xi} + \Delta_{\text{cr}} \frac{\partial \bar{v}_i}{\partial \eta}  \right) + O(\epsilon^2)$$ where $\bar{v}_i := v_i(\xi - \Delta_{\text{cr}}, \eta)$ for $ i = 1, 2.$ Applying these expansions to equations \ref{v1 equation 1} and \ref{v2 equation 1} and then collecting $O(1)$ terms and $O(\epsilon)$ terms yields the following four equations.

\begin{align}
\frac{\partial v_{1,0}}{\partial \xi} + \mu v_{1,0} &= 0 \label{v10 equation 1}\\
\frac{\partial v_{2,0}}{\partial \xi} + \frac{\lambda \theta}{2} \bar{v}_{2,0} + \mu v_{2,0} &= 0 \label{v20 equation 1}\\
\frac{\partial v_{1,1}}{\partial \xi} + \mu v_{1,1} &= - \frac{\partial v_{1,0}}{\partial \eta} - \frac{\theta \hat{\lambda}}{4} \bar{v}_{2,0} - \frac{\hat{\mu}}{2} \left( v_{1,0} + v_{2,0} \right) \label{v11 equation 1}\\
\frac{\partial v_{2,1}}{\partial \xi} + \frac{\lambda \theta}{2} \bar{v}_{2,1} + \mu v_{2,1} &= - \frac{ \partial v_{2,0}}{\partial \eta} + \frac{\lambda \theta}{2} \left( \Delta_1 \frac{\partial \bar{v}_{2,0}}{\partial \xi} + \Delta_{\text{cr}} \frac{\partial \bar{v}_{2,0}}{\partial \eta}  \right) \label{v21 equation 1}\\
 &- \frac{\theta \hat{\lambda}}{4} \bar{v}_{2,0} - \frac{\hat{\mu}}{2} \left( v_{1,0} + v_{2,0} \right) - \frac{\lambda \hat{\theta}}{4} \left( \bar{v}_{1,0} + \bar{v_{2,0}} \right) \nonumber
\end{align}

\noindent It is easy to check that $$v_{1,0} = \tilde{c}(\eta)  \exp(- \mu \xi) \hspace{5mm} \text{ and } \hspace{5mm} v_{2,0} = A(\eta) \cos(\omega_{\text{cr}} \xi) + B \sin(\omega_{\text{cr}} \xi)$$ solve Equations \ref{v10 equation 1} and \ref{v20 equation 1}, respectively, and we can rearrange Equation \ref{v20 equation 1} and use our expression for $v_{2,0}$ to observe that $$\bar{v}_{2,0} = - \frac{2}{\lambda \theta} \left[ \frac{\partial v_{2,0}}{\partial \xi} + \mu v_{2,0}  \right] = \frac{2}{\lambda \theta} \left[ -(\mu A + \omega_{\text{cr}} B) \cos(\omega_{\text{cr}} \xi) + (\omega_{\text{cr}} A - \mu B) \sin(\omega_{\text{cr}} \xi) \right].$$ Thus, we have the following expressions for terms in Equations \ref{v11 equation 1} and \ref{v21 equation 1}

\begin{align*}
\frac{\partial v_{1,0}}{\partial \eta} &= \tilde{c}' \exp(-\mu \xi)\\
\frac{\partial v_{2,0}}{\partial \eta} &= A' \cos(\omega_{\text{cr}} \xi) + B' \sin(\omega_{\text{cr}} \xi)\\
\frac{\partial \bar{v}_{2,0}}{\partial \eta} &= \frac{2}{\theta \lambda} \left[ -(\mu A' + \omega_{\text{cr}} B') \cos(\omega_{\text{cr}} \xi) + (\omega_{\text{cr}} A' - \mu B') \sin(\omega_{\text{cr}} \xi) \right]\\
\frac{\partial \bar{v}_{2,0}}{\partial \xi} &= \frac{2  \omega_{\text{cr}}}{\theta \lambda} \left[ ( \omega_{\text{cr}} A - \mu B) \cos( \omega_{\text{cr}} \xi) + (\mu A +  \omega_{\text{cr}} B) \sin( \omega_{\text{cr}} \xi) \right]
\end{align*}

\noindent and Equations \ref{v11 equation 1} and \ref{v21 equation 1} can respectively be rewritten as

\begin{align}
\frac{\partial v_{1,1}}{\partial \xi} + v_{1,1} &= \left(\tilde{c}'(\eta) - \frac{\hat{\mu}}{2} \tilde{c}(\eta)\right)\exp(- \mu \xi) + \cos(\omega_{\text{cr}} \xi) \left[ \frac{\hat{\lambda}}{2 \lambda} (\mu A(\eta) + \omega_{\text{cr}} B(\eta)) - \frac{\hat{\mu}}{2} A(\eta) \right]  \nonumber\\
&+ \sin(\omega_{\text{cr}} \xi) \left[ - \frac{\hat{\lambda}}{2 \lambda} (\omega_{\text{cr}} A(\eta) - \mu B(\eta)) - \frac{\hat{\mu}}{2} B(\eta)  \right] \label{v11 equation 2}\\
\frac{\partial v_{2,1}}{\partial \xi} &+ \frac{\lambda \theta}{2} \bar{v}_{2,1} + \mu v_{2,1} = - \frac{\hat{\mu}}{2} \tilde{c} \exp(- \mu \xi) \nonumber\\
&+ \cos(\omega_{\text{cr}} \xi) \Bigg[A'(\eta)(- \mu \Delta_{\text{cr}} - 1) + B'(\eta) (- \omega_{\text{cr}} \Delta_{\text{cr}}) \nonumber \\
&+ A(\eta) \left(\omega_{\text{cr}}^2 \Delta_1 + \frac{\mu \hat{\lambda}}{2 \lambda} - \frac{\hat{\mu}}{2} + \frac{\mu \hat{\theta}}{2 \theta}\right) + B(\eta)\left(- \mu \omega_{\text{cr}} \Delta_1 + \frac{\omega_{\text{cr}} \hat{\lambda}}{2 \lambda} + \frac{\omega_{\text{cr}} \hat{\theta}}{2 \theta}\right) \Bigg] \nonumber\\
&+ \sin(\omega_{\text{cr}} \xi) \Bigg[ A'(\eta)\left( \omega_{\text{cr}} \Delta_{\text{cr}} \right) + B'(\eta) (-\mu \Delta_{\text{cr}} - 1) \nonumber\\
&+ A(\eta) \left( \mu \omega_{\text{cr}} \Delta_1 - \frac{\omega_{\text{cr}} \hat{\lambda}}{2 \lambda} - \frac{\omega_{\text{cr}} \hat{\theta}}{2 \theta} \right) + B(\eta) \left( \omega_{\text{cr}}^2 \Delta_1 + \frac{\mu \hat{\lambda}}{2 \lambda} - \frac{\hat{\mu}}{2} + \frac{\mu \hat{\theta}}{2 \theta} \right) \Bigg] . \label{v21 equation 2}
\end{align}

We observe that the general homogeneous solutions for $v_{1,1}$ and $v_{2,1}$ are the same as the general homogeneous solutions for $v_{1,0}$ and $v_{2,0}$, respectively. In both Equations \ref{v11 equation 2} and \ref{v21 equation 2}, there are terms present in the inhomogeneous part that are not linearly independent of the corresponding homogeneous solution. It is easy to see, by the method of undetermined coefficients for example (which introduces a factor of $\xi$ on to terms in the particular solution that correspond to terms in the inhomogeneity that are linearly dependent with a homogeneous solution), that such terms will give rise to secular terms in the particular solutions to each equation. We want to set terms in the inhomogeneities that introduce secular solutions equal to zero because our asymptotic expansions would otherwise become invalid for large time as the series would no longer be asymptotic when $\xi = O(\frac{1}{\epsilon})$, for example, at which point $O(\epsilon)$ terms in the series would become $O(1)$. Equating the coefficients of these terms in the inhomogeneities equal to zero yields the following equations.

\begin{align}
\frac{d \tilde{c}}{d \eta} &= \frac{\hat{\mu}}{2} \tilde{c} \label{secular equation 1}\\
\frac{d A}{d \eta} &= K_1 A(\eta) + K_2 B(\eta) \label{secular equation 2}\\
\frac{d B}{d \eta} &= K_3 A(\eta) + K_4 B(\eta) \label{secular equation 3}
\end{align}

Solving Equation \ref{secular equation 1} gives us that $\tilde{c}(\eta) = \tilde{k} \exp(\frac{\hat{\mu}}{2} \eta)$ and therefore $v_{1,0} = \tilde{k} \exp(\frac{\hat{\mu}}{2} \eta - \mu \xi)$ which decays to $0$ for sufficiently small $\epsilon$. We observe that the system of Equations \ref{secular equation 2}-\ref{secular equation 3} is in the form $$c_1 A' + c_2 B' + c_3 A + c_4 B = 0$$ $$-c_2 A' + c_1 B' - c_4 A + c_3 B = 0$$ where $$c_1 = - \mu \Delta_{\text{cr}} - 1, \hspace{5mm} c_2 = -  \omega_{\text{cr}} \Delta_{\text{cr}}, \hspace{5mm} c_3 =  \omega_{\text{cr}}^2 \Delta_1 + \frac{\mu \hat{\lambda}}{2 \lambda}  - \frac{\hat{\mu}}{2} + \frac{\mu \hat{\theta}}{2 \theta}, \hspace{5mm} c_4 = -  \mu \omega_{\text{cr}} \Delta_1 + \frac{\omega_{\text{cr}} \hat{\lambda}}{2 \lambda} + \frac{\omega_{\text{cr}} \hat{\theta}}{2 \theta}.$$  This tells us that 

\begin{align*}
K_1 &= K_4 = \frac{-(c_1 c_3 + c_2 c_4)}{c_1^2 + c_2^2}\\
K_2 &= - K_3 = \frac{c_2 c_3 - c_1 c_4}{c_1^2 + c_2^2}.
\end{align*}

\noindent So, we have the linear system 
\begin{align}
\begin{bmatrix}
\frac{d A}{d \eta} \\
\frac{d B}{d \eta}
\end{bmatrix} = \begin{bmatrix}
K_1 & - K_3 \\
K_3 & K_1
\end{bmatrix} \begin{bmatrix}
A\\
B
\end{bmatrix}. \label{slow flow system}
\end{align}

\noindent Recall that $$v_{2,0} = A(\eta) \cos(\omega_{\text{cr}} \xi) + B(\eta) \sin(\omega_{\text{cr}} \xi)$$ so that $A(\eta)$ and $B(\eta)$ represent the amplitudes of each term in $v_{2,0}$.  Thus, the equilibrium point $A(\eta) = B(\eta) = 0$ of this linear system corresponds to when $v_{2,0} = 0$ and it also corresponds to when the sinusoidal terms in the inhomogeneity in \ref{v11 equation 2} are equal to zero which would make $v_{1,1}$ decay for sufficiently small $\epsilon$. Because of this, the stability of the equilibrium point $(A,B) = (0,0)$ to \ref{slow flow system} corresponds to the stability of the DDE system given in Equations \ref{perturbed equation 1}-\ref{perturbed equation 2}>.  Thus, our problem reduces to analyzing the stability of a linear system. 

\noindent  Now we define the following matrix $$K = \begin{bmatrix}
K_1 & - K_3 \\
K_3 & K_1
\end{bmatrix}.$$

\noindent Note that since we assumed $\lambda \theta > 2 \mu$, we have that each entry of $K$ is real. To analyze the stability of Equation \ref{slow flow system}, we need to determine whether the real parts of the eigenvalues of $K$ are positive or negative. However, keep in mind that the entries of $K$ depend on $\Delta_1$, so if we can find conditions on what the value of $\Delta_1$ must be in order for the real parts of the eigenvalues of $K$ to change sign, then we'll essentially have found an approximation (up to $O(\epsilon)$ terms) for the critical value of $\Delta$ for which the stability of our DDE system given in Equations \ref{perturbed equation 1}-\ref{perturbed equation 2} changes and a Hopf bifurcation occurs.  In particular, we note the special structure of this matrix $K$ (it is actually the matrix representation of the complex number $K_1 + i K_3$) and see that it has eigenvalues $K_1 \pm i K_3$. Thus, $K_1$ is the real part of both of the eigenvalues of $K$, so we want to find conditions on $\Delta_1$ under which $K_1$ is positive or negative. We see that $\text{sgn}(K_1) = -\text{sgn}(c_1 c_3 + c_2 c_4) $. 

\begin{align*}
c_1 c_3 + c_2 c_4 &= ( - \mu \Delta_{\text{cr}} - 1) \left( \omega_{\text{cr}}^2 \Delta_1 + \frac{\mu \hat{\lambda}}{2 \lambda}  - \frac{\hat{\mu}}{2} + \frac{\mu \hat{\theta}}{2 \theta} \right) + (-  \omega_{\text{cr}} \Delta_{\text{cr}}) \left( -  \mu \omega_{\text{cr}} \Delta_1 + \frac{\omega_{\text{cr}} \hat{\lambda}}{2 \lambda} + \frac{\omega_{\text{cr}} \hat{\theta}}{2 \theta} \right)\\
&= - \omega_{\text{cr}}^2 \Delta_1 - \left[ \frac{\mu + \Delta_{\text{cr}}(\mu^2 + \omega_{\text{cr}}^2)}{2 \lambda} \hat{\lambda} - \frac{1 + \mu \Delta_{\text{cr}}}{2} \hat{\mu} + \frac{\mu + \Delta_{\text{cr}}(\mu^2 + \omega_{\text{cr}}^2)}{2 \theta} \hat{\theta} \right]
\end{align*}

\noindent So we see that $K_1 < 0$ when $$\Delta_1 < - \left( \frac{\mu + \Delta_{\text{cr}}(\mu^2 + \omega_{\text{cr}}^2)}{2 \lambda \omega_{\text{cr}}^2} \hat{\lambda} - \frac{1 + \mu \Delta_{\text{cr}}}{2 \omega_{\text{cr}}^2} \hat{\mu} + \frac{\mu + \Delta_{\text{cr}}(\mu^2 + \omega_{\text{cr}}^2)}{2 \theta \omega_{\text{cr}}^2} \hat{\theta} \right)$$ and $K_1 > 0$ when $$\Delta_1 >  - \left( \frac{\mu + \Delta_{\text{cr}}(\mu^2 + \omega_{\text{cr}}^2)}{2 \lambda \omega_{\text{cr}}^2} \hat{\lambda} - \frac{1 + \mu \Delta_{\text{cr}}}{2 \omega_{\text{cr}}^2} \hat{\mu} + \frac{\mu + \Delta_{\text{cr}}(\mu^2 + \omega_{\text{cr}}^2)}{2 \theta \omega_{\text{cr}}^2} \hat{\theta} \right)$$ and since $\Delta = \Delta_{\text{cr}} + \epsilon \Delta_1 + O(\epsilon^2)$, we see that the critical value of $\Delta$ for which the stability of our perturbed DDE system given  in Equations \ref{perturbed equation 1}-\ref{perturbed equation 2} changes is $$\Delta_{\text{mod}} = \Delta_{\text{cr}} - \epsilon \left( \frac{\mu + \Delta_{\text{cr}}(\mu^2 + \omega_{\text{cr}}^2)}{2 \lambda \omega_{\text{cr}}^2} \hat{\lambda} - \frac{1 + \mu \Delta_{\text{cr}}}{2 \omega_{\text{cr}}^2} \hat{\mu} + \frac{\mu + \Delta_{\text{cr}}(\mu^2 + \omega_{\text{cr}}^2)}{2 \theta \omega_{\text{cr}}^2} \hat{\theta} \right) + O(\epsilon^2).$$

\end{proof}

%\textcolor{blue}{Philip, how did you make this observation below?  One should give a proof or more insight that these two expressions are indeed the same.}
An observation we can make that more clearly relates this expression of $\Delta_{\text{mod}}$ to the form of $\Delta_{\text{cr}}$ given in \citet{novitzky2019nonlinear} is that $$\Delta_{\text{mod}} = \frac{\arccos \left(  \frac{ -2 \left(  \mu + \frac{\hat{\mu} \epsilon}{2} \right) }{ \left( \lambda + \frac{\hat{\lambda} \epsilon}{2}  \right)  \left(  \theta + \frac{\hat{\theta} \epsilon}{2} \right) }  \right)}{\omega_{\text{mod}}} + O(\epsilon^2) \hspace{5mm} \text{ where }$$ $$\omega_{\text{mod}} = \frac{1}{2} \sqrt{ \left( \lambda + \frac{\hat{\lambda} \epsilon}{2} \right)^2  \left( \theta + \frac{\hat{\theta} \epsilon}{2}  \right)^2 - 4 \left( \mu + \frac{\hat{\mu} \epsilon}{2}   \right)^2    }$$ which can be seen by Taylor expanding about $\epsilon = 0$. To give some intuition regarding how this observation was made, consider the case where we only perturb the arrival rate $\lambda$. Linearizing the system, neglecting $O(\epsilon^2)$ terms, we'll obtain equations \ref{u1 equation 45} and \ref{u2 equation 46}, but with $\hat{\theta} = \hat{\mu} = 0$.  

\begin{align}
\overset{\bullet}{u}_1(t) &= \frac{(\lambda + \hat{\lambda} \epsilon) \theta}{4}\left[ u_2(t - \Delta) - u_1(t - \Delta) \right] - \mu  u_1(t) \\
\overset{\bullet}{u}_2(t) &= \frac{\lambda \theta}{4} [u_1(t- \Delta) - u_2(t - \Delta)]  - \mu u_2(t) 
\end{align}

Applying the transformation $v_1(t) = u_1(t) + u_2(t)$ and $v_2(t) = u_1(t) - u_2(t)$ gives us equations \ref{v1 equation 1} and \ref{v2 equation 1} except with $\hat{\theta} = \hat{\mu} = 0$.

\begin{align}
\overset{\bullet}{v}_1(t) +  \mu  v_1(t) &= - \frac{\theta \hat{\lambda}}{4}\epsilon v_2(t - \Delta) \\
\overset{\bullet}{v}_2(t) + \left( \frac{\lambda \theta}{2} + \frac{\theta \hat{\lambda}}{4} \epsilon  \right) v_2(t - \Delta) + \mu  v_2(t) &= 0
\end{align}

We see that the coupling that was present in equations \ref{v1 equation 1} and \ref{v2 equation 1} has been simplified. We see that the homogeneous solution to the equation for $v_1(t)$ decays with time. Thus, if $v_2(t)$ is stable, then the particular solution for the $v_1(t)$ equation will be stable and if $v_2(t)$ is unstable then the system is unstable. Because of this, we can restrict our attention to the equation for $v_2(t).$ Letting $v_2(t) = e^{rt}$ gives us the characteristic equation $$r + C e^{- r \Delta} + \mu = 0$$ where $C =   \frac{\lambda \theta}{2} + \frac{\theta \hat{\lambda}}{4} \epsilon  $. The system is stable when $r$ has negative real part and it is unstable when $r$ has positive real part, so the change in stability occurs when $r$ crosses the imaginary axis, so we let $r = i \omega$ for some real $\omega$. Collecting real and imaginary parts gives us the equations 

\begin{align}
\sin(\omega \Delta) &= \frac{\omega}{C}\\
 \cos(\omega \Delta) &= - \frac{\mu}{C}
\end{align}

\noindent and since $\cos^2(\omega \Delta) + \sin^2(\omega \Delta) = 1$, we are able to get $$\Delta_{\text{mod}} = \frac{\arccos \left( - \frac{\mu}{C} \right) }{\omega_{\text{mod}}}, \hspace{5mm} \omega_{\text{mod}} = \sqrt{C^2 - \mu^2}.$$ We see that $$C = \frac{\lambda \theta}{2} + \frac{\theta \hat{\lambda}}{4} \epsilon   =  \frac{\theta \left(  \lambda + \frac{\hat{\lambda}}{2} \epsilon \right)}{2}$$ so we get $$\Delta_{\text{mod}} = \frac{2 \arccos \left(  - \frac{2 \mu}{ \left( \lambda + \frac{\hat{\lambda} \epsilon }{2} \right) \theta  }  \right)}{\sqrt{ \left( \lambda + \frac{\hat{\lambda}}{2} \epsilon \right)^2 \theta^2 - 4 \mu^2  }}.$$

While this isn't a rigorous approach to obtaining the expression we got for $\Delta_{\text{mod}}$ with all of the parameters perturbed as the coupling in the general case causes complications, it should at least give some intuition for why we considered the expression above. The perturbations to the parameters end up being divided by $2$ in each case due to the transformation from $u_1$ and $u_2$ to $v_1$ and $v_2$. This also has to do with the fact that the perturbation terms for the $\lambda$ and $\mu$ cases only appear in a single equation in equations \ref{v1 equation 1} and \ref{v2 equation 1}, so when forming the equation for $v_2(t)$, these terms do not get the factor of $2$ that terms that were in both equations (but differed by a factor of -1) got. Also, even though the $\theta$ perturbation is in both equations, it is multiplying a $u_1$ term in both cases which transforms to $$u_1 = \frac{v_1 + v_2}{2}$$ which introduces a factor of $\frac{1}{2}.$

Another important observation to make is that our expression for $\Delta_{\text{mod}}$ appears to not depend on $\hat{\alpha}$ up to first order. However, if we collect $O(\epsilon^2)$ terms when linearizing the system, we can see the contributions from $\hat{\alpha}$. We illustrate this in Theorem \ref{delta_mod_alpha_theorem} by considering the special case where $\hat{\lambda} = \hat{\mu} = \hat{\theta} = 0$ for ease of calculation.

\begin{theorem}
\label{delta_mod_alpha_theorem}
When $\hat{\lambda} = \hat{\mu} = \hat{\theta} = 0$, we have that $$\Delta_{\text{mod}} =  \Delta_{\text{cr}} + \frac{4 \mu^3 + \mu^2 \lambda^2 \theta^2 \Delta_{\text{cr}}}{4 \omega_{\text{cr}}^2 (\lambda \theta + 2 \mu)^2} \epsilon^2 \hat{\alpha}^2 + O(\epsilon^3).$$
\end{theorem}

\begin{proof}
 It can be shown by a calculation similar to the one in Section~\ref{sec_equilibrium} that the equilibrium point in this special case is $$(q_1^*, q_2^*) = \left( \frac{\lambda}{2 \mu} + a_1 \epsilon + a_2 \epsilon^2 + O(\epsilon^3),   \frac{\lambda}{2 \mu} + b_1 \epsilon + b_2 \epsilon^2 + O(\epsilon^3)  \right)$$ where $$a_1 = \frac{\lambda}{2 (\lambda \theta + 2 \mu)}, \hspace{5mm} b_1 = \frac{- \lambda}{2 (\lambda \theta + 2 \mu)}$$ and $a_2 = b_2 = 0$. Linearizing about this equlibrium point and neglecting $O(\epsilon^3)$ terms gives us the following linear system.

\begin{align*}
\overset{\bullet}{u}_1(t) &= \left( \frac{\lambda \theta}{4} - \frac{1}{16} \left[ (a_1 - b_1)^2 \lambda \theta^3 + 2 (b_1 - a_1) \lambda \theta^2 \hat{\alpha} + \lambda \theta \hat{\alpha}^2  \right] \epsilon^2  \right) [u_2(t - \Delta) - u_1(t - \Delta)] - \mu u_1(t)\\
\overset{\bullet}{u}_2(t) &= \left( \frac{\lambda \theta}{4} - \frac{1}{16} \left[ (a_1 - b_1)^2 \lambda \theta^3 + 2 (b_1 - a_1) \lambda \theta^2 \hat{\alpha} + \lambda \theta \hat{\alpha}^2  \right] \epsilon^2  \right) [u_1(t - \Delta) - u_2(t - \Delta)] - \mu u_2(t)\\
\end{align*}

\noindent Using the transformation $$v_1(t) = u_1(t) + u_2(t), \hspace{5mm} v_2(t) = u_1(t) - u_2(t)$$ gives us the system 

\begin{align}
\overset{\bullet}{v}_1(t) &+ \mu v_1(t) = 0\\
\overset{\bullet}{v}_2(t) &+ \left( \frac{\lambda \theta}{2} - \frac{1}{8} \left[ (a_1 - b_1)^2 \lambda \theta^3 + 2 (b_1 - a_1) \lambda \theta^2 \hat{\alpha} + \lambda \theta \hat{\alpha}^2  \right] \epsilon^2  \right) v_2(t - \Delta) + \mu v_2(t) = 0
\end{align}

\noindent We see that $v_1(t)$ decays with time and thus we restrict our analysis to (3.20). Letting $v_2(t) = e^{rt}$, we get the characteristic equation $$r + D e^{- r \Delta} + \mu = 0$$ where we let $D = \left( \frac{\lambda \theta}{2} - \frac{1}{8} \left[ (a_1 - b_1)^2 \lambda \theta^3 + 2 (b_1 - a_1) \lambda \theta^2 \hat{\alpha} + \lambda \theta \hat{\alpha}^2  \right] \epsilon^2  \right)$ for ease of notation. If $r$ has negative real part, then we have stability and we have instability when $r$ has positive real part. Thus, the change in stability occurs when $r$ crosses the imaginary axis, that is when $r = i \omega$ for some real $\omega$. Letting $r = i \omega$ and collecting real and imaginary parts gives us the following two equations

\begin{align}
\sin(\omega \Delta) &= \frac{\omega}{D}\\
\cos(\omega \Delta) &= - \frac{\mu}{D}
\end{align}

\noindent so that, using the fact that $\cos^2(\omega \Delta) + \sin^2(\omega \Delta) = 1$, we get $$\Delta_{\text{mod}} = \frac{\arccos \left( - \frac{\mu}{D} \right)}{\omega_{\text{mod}}}, \hspace{5mm} \omega_{\text{mod}} = \sqrt{D^2 - \mu^2}.$$ Taylor expanding $\Delta_{\text{mod}}$ about $\epsilon = 0$ gives us the result $$\Delta_{\text{mod}} =  \Delta_{\text{cr}} + \frac{4 \mu^3 + \mu^2 \lambda^2 \theta^2 \Delta_{\text{cr}}}{4 \omega_{\text{cr}}^2 (\lambda \theta + 2 \mu)^2} \epsilon^2 \hat{\alpha}^2 + O(\epsilon^3)$$ when $\hat{\lambda} = \hat{\mu} = \hat{\theta} = 0.$

\end{proof}
%\textcolor{blue}{Philip, you should make this result above a new theorem.}

\subsection{Numerical Verification of Hopf Bifurcation}

Below we show plots of queue length versus time for various cases to demonstrate the bifurcation in $\Delta$. In each figure below, we consider our system with various parameters either being perturbed or not perturbed from symmetry. In each case, we consider having a delay $.05$ below and $.05$ above the corresponding critical delay value. In each case, we see that the queue length amplitudes decay to equilibrium values when the delay $\Delta$ is below the critical value. When we increase the delay to be above the critical delay, we see oscillations increase and approach a fixed amplitude forming a limit cycle. This suggests that we have a Hopf bifurcation at the critical delay. This observation prompts us to consider the amplitudes of limit cycles in Section \ref{sec_amplitude}.

%\textcolor{blue}{Philip, you need to go into more detail and explain each plot accordingly.}

In Figure \ref{Fig7}, we consider the symmetric case and we see that the amplitudes of the queues oscillate and decay when the delay is below the critical delta and approach some limiting amplitude when the delay is above the critical delta. In Figure \ref{Fig8}, we consider the case where only the arrival rate $\lambda$ is perturbed positively. In this case, we see that increasing the arrival rate in one of the queues causes the critical delay to be less than the critical delay in the symmetric case. In Figure \ref{Fig9}, the service rate $\mu$ is the only perturbed parameter and we are able to see that increasing the service rate in one of the queues leads to an increase in the critical delay. These observations tell us that increasing the inflow of customers in one queue makes the system more susceptible to oscillations caused by delayed information whereas increasing the service rate in one of the queues helps to mitigate this issue. In Figure \ref{Fig10}, the only perturbed parameter is $\theta$ and we see that increasing the value of $\theta$ corresponding to one of the queues causes a decrease in the critical delay. We note that increasing the value of the $\theta$ corresponding to one of the queues increases the number of arrivals into that queue and thus it makes sense that it impacts the critical delay in the same direction that perturbing the arrival rate does. In Figure \ref{Fig11}, the only perturbed parameter is $\alpha$. We see that the critical delay is roughly the same as the critical delay in the symmetric case which isn't surprising based on the result of Theorem \ref{delta_mod_alpha_theorem} which says that perturbing $\alpha$ only affects the value of the critical delay if we include $O(\epsilon^2)$ terms. Figure \ref{Fig12} is an example where all four of the parameters we have discussed were perturbed from symmetry. The impact that perturbing all four of these parameters has on the critical delay will depend on how much each parameter is perturbed by.

%\begin{figure}[ht!]
%  \centering
%  \includegraphics[scale=.515]{./Code/Paper_Figures/Fig_sym2.pdf}
%\vspace{-40mm}
%\captionsetup{justification=centering,margin=2cm}
%  \caption{$\Delta_{\text{mod}} = \Delta_{\text{cr}} \approx .3617, \Delta = \Delta_{\text{mod}} - .05$\\
%$\lambda = 10, \mu=1, \theta=1, \alpha=0$\\ History function is constant with $q_1 = 4.99$ and $q_2 = 5.01$}
%\end{figure}
%
%\begin{figure}[hb!]
%  \centering
%  \includegraphics[scale=.515]{./Code/Paper_Figures/Fig_sym1.pdf}
%\vspace{-40mm}
%\captionsetup{justification=centering,margin=2cm}
%  \caption{$\Delta_{\text{mod}} = \Delta_{\text{cr}} \approx .3617, \Delta = \Delta_{\text{mod}} + .05$\\
%$\lambda = 10, \mu=1, \theta=1, \alpha=0$\\ History function is constant with $q_1 = 4.99$ and $q_2 = 5.01$}
%\end{figure}

%%%%%%%%%%%%%%%%%%%

\vspace{30mm}

\begin{figure}[ht!]
\vspace{-35mm}
  \hspace{-5mm}~\includegraphics[scale=.4]{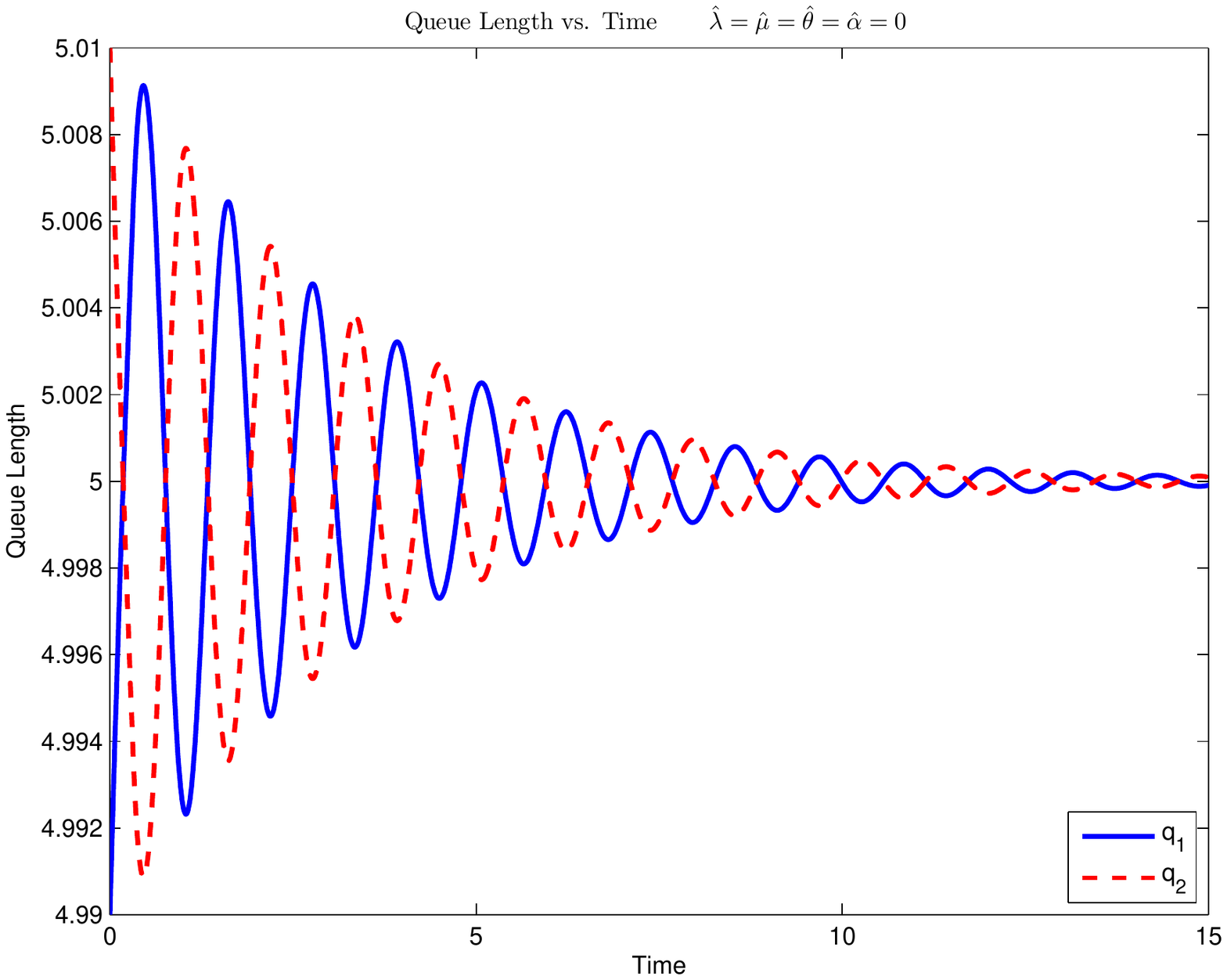}~\hspace{-10mm}~\includegraphics[scale=.4]{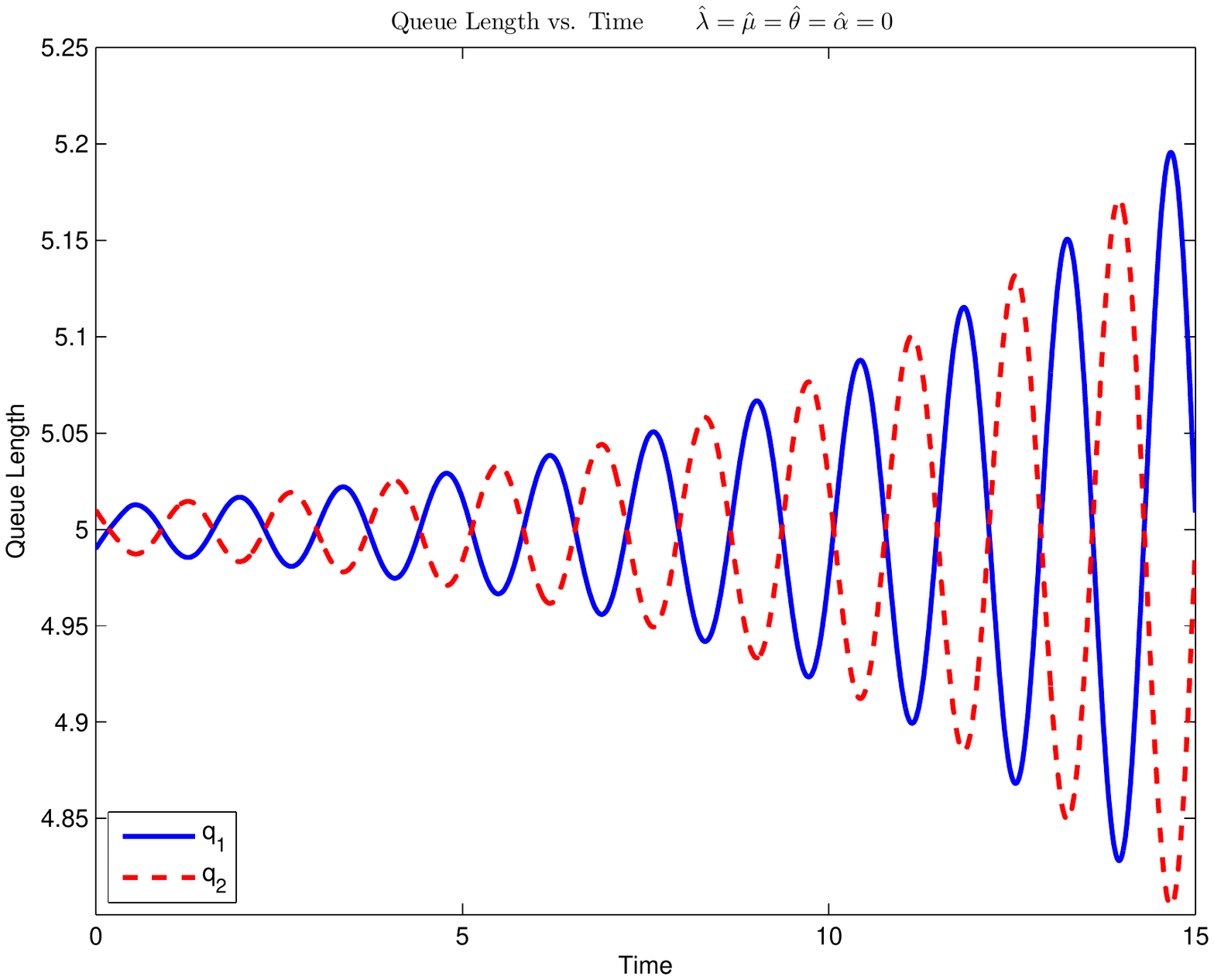}
%\captionsetup{justification=centering,margin=2cm}
\vspace{-30mm}
  \caption{$\hat{\lambda} = \hat{\mu} = \hat{\theta} = \hat{\alpha} = 0$,  $\lambda = 10,\mu=1, \theta=1, \alpha=0$, $\Delta_{\text{mod}} = \Delta_{\text{cr}} \approx .3617$\\ On $[-\Delta, 0]$, $q_1 = 4.99$, $q_2 = 5.01$, Left: $\Delta = \Delta_{\text{mod}} - .05$, Right: $\Delta = \Delta_{\text{mod}} + .05$}
\label{Fig7}
\end{figure}

%\begin{figure}[ht!]
%  \centering
%\vspace{-55mm}
%  \includegraphics[scale=.7]{./Code/Paper_Figures/Fig_sym2.pdf}
%\vspace{-50mm}
%\captionsetup{justification=centering,margin=2cm}
%  \caption{$\Delta_{\text{mod}} = \Delta_{\text{cr}} \approx .3617, \Delta = \Delta_{\text{mod}} - .05$\\
%$\lambda = 10, \mu=1, \theta=1, \alpha=0$\\ History function is constant with $q_1 = 4.99$ and $q_2 = 5.01$}
%\vspace{-40mm}
%  \includegraphics[scale=.7]{./Code/Paper_Figures/Fig_sym1.pdf}
%\vspace{-50mm}
%  \caption{$\Delta_{\text{mod}} = \Delta_{\text{cr}} \approx .3617, \Delta = \Delta_{\text{mod}} + .05$\\
%$\lambda = 10, \mu=1, \theta=1, \alpha=0$\\ History function is constant with $q_1 = 4.99$ and $q_2 = 5.01$}
%\end{figure}

%%%%%%%%%%%%%%%%%%
%
%\begin{figure}[ht!]
%  \centering
%  \includegraphics[scale=.515]{./Code/Paper_Figures/Fig_1.pdf}
%\vspace{-40mm}
%\captionsetup{justification=centering,margin=2cm}
%  \caption{$\Delta_{\text{mod}} \approx .3596, \Delta = \Delta_{\text{mod}} - .05$\\
%$\epsilon = .1, \lambda = 10, \mu=1, \theta=1, \alpha=0$\\ History function is constant with $q_1 = 4.99$ and $q_2 = 5.01$}
%\end{figure}
%
%\begin{figure}[hb!]
%  \centering
%  \includegraphics[scale=.515]{./Code/Paper_Figures/Fig_5.pdf}
%\vspace{-40mm}
%\captionsetup{justification=centering,margin=2cm}
%  \caption{$\Delta_{\text{mod}} \approx .3596, \Delta = \Delta_{\text{mod}} + .05$\\
%$\epsilon = .1, \lambda = 10, \mu=1, \theta=1, \alpha=0$\\ History function is constant with $q_1 = 4.99$ and $q_2 = 5.01$}
%\end{figure}

%%%%%%%%%%%%%%%%

\begin{figure}[ht!]
\vspace{-35mm}
  \hspace{-5mm}~\includegraphics[scale=.4]{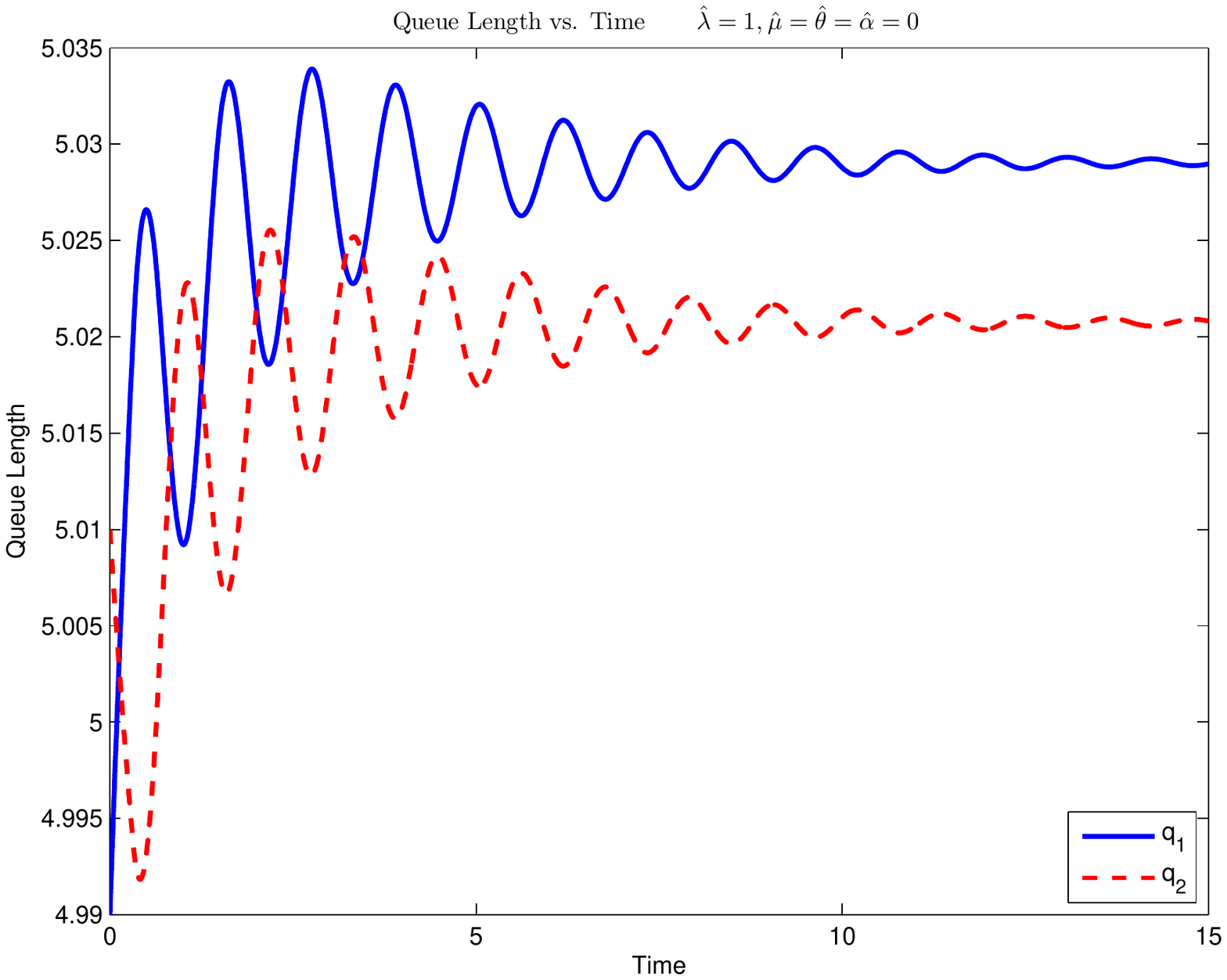}~\hspace{-10mm}~\includegraphics[scale=.4]{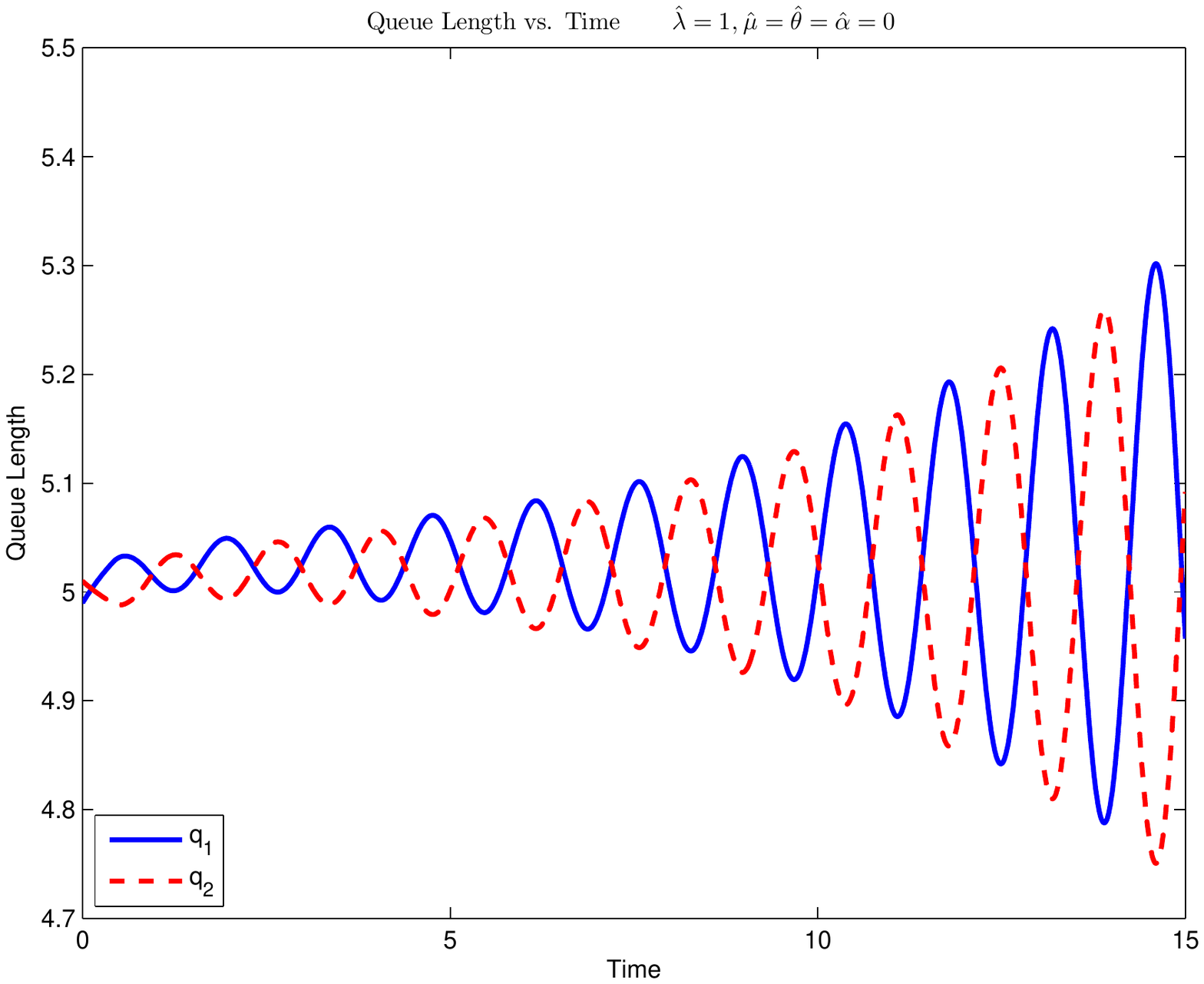}
%\captionsetup{justification=centering,margin=2cm}
\vspace{-30mm}
  \caption{$\hat{\lambda} = 1, \hat{\mu} = \hat{\theta} = \hat{\alpha} = 0$, $\Delta_{\text{mod}} \approx .3596$
$\epsilon = .1, \lambda = 10, \mu=1, \theta=1, \alpha=0$\\ On $[-\Delta, 0]$, $q_1 = 4.99$, $q_2 = 5.01$, Left: $\Delta = \Delta_{\text{mod}} - .05$, Right: $\Delta = \Delta_{\text{mod}} + .05$}
\label{Fig8}
\end{figure}

%\begin{figure}[ht!]
%  \centering
%\vspace{-55mm}
%  \includegraphics[scale=.7]{./Code/Paper_Figures/Fig_1.pdf}
%\vspace{-50mm}
%\captionsetup{justification=centering,margin=2cm}
%  \caption{$\Delta_{\text{mod}} \approx .3596, \Delta = \Delta_{\text{mod}} - .05$\\
%$\epsilon = .1, \lambda = 10, \mu=1, \theta=1, \alpha=0$\\ History function is constant with $q_1 = 4.99$ and $q_2 = 5.01$}
%\vspace{-40mm}
%  \includegraphics[scale=.7]{./Code/Paper_Figures/Fig_5.pdf}
%\vspace{-50mm}
%  \caption{$\Delta_{\text{mod}} \approx .3596, \Delta = \Delta_{\text{mod}} + .05$\\
%$\epsilon = .1, \lambda = 10, \mu=1, \theta=1, \alpha=0$\\ History function is constant with $q_1 = 4.99$ and $q_2 = 5.01$}
%\end{figure}

%%%%%%%%%%%%%%%%%%%%%

%\begin{figure}[ht!]
%  \centering
%  \includegraphics[scale=.515]{./Code/Paper_Figures/Fig_2.pdf}
%\vspace{-40mm}
%\captionsetup{justification=centering,margin=2cm}
%  \caption{$\Delta_{\text{mod}} \approx .3646, \Delta = \Delta_{\text{mod}} - .05$\\
%$\epsilon = .1, \lambda = 10, \mu=1, \theta=1, \alpha=0$\\ History function is constant with $q_1 = 4.99$ and $q_2 = 5.01$}
%\end{figure}
%
%\begin{figure}[hb!]
%  \centering
%  \includegraphics[scale=.515]{./Code/Paper_Figures/Fig_6.pdf}
%\vspace{-40mm}
%\captionsetup{justification=centering,margin=2cm}
%  \caption{$\Delta_{\text{mod}} \approx .3646, \Delta = \Delta_{\text{mod}} + .05$\\
%$\epsilon = .1, \lambda = 10, \mu=1, \theta=1, \alpha=0$\\ History function is constant with $q_1 = 4.99$ and $q_2 = 5.01$}
%\end{figure}

%%%%%%%%%%%%%

\begin{figure}[ht!]
\vspace{-35mm}
  \hspace{-5mm}~\includegraphics[scale=.4]{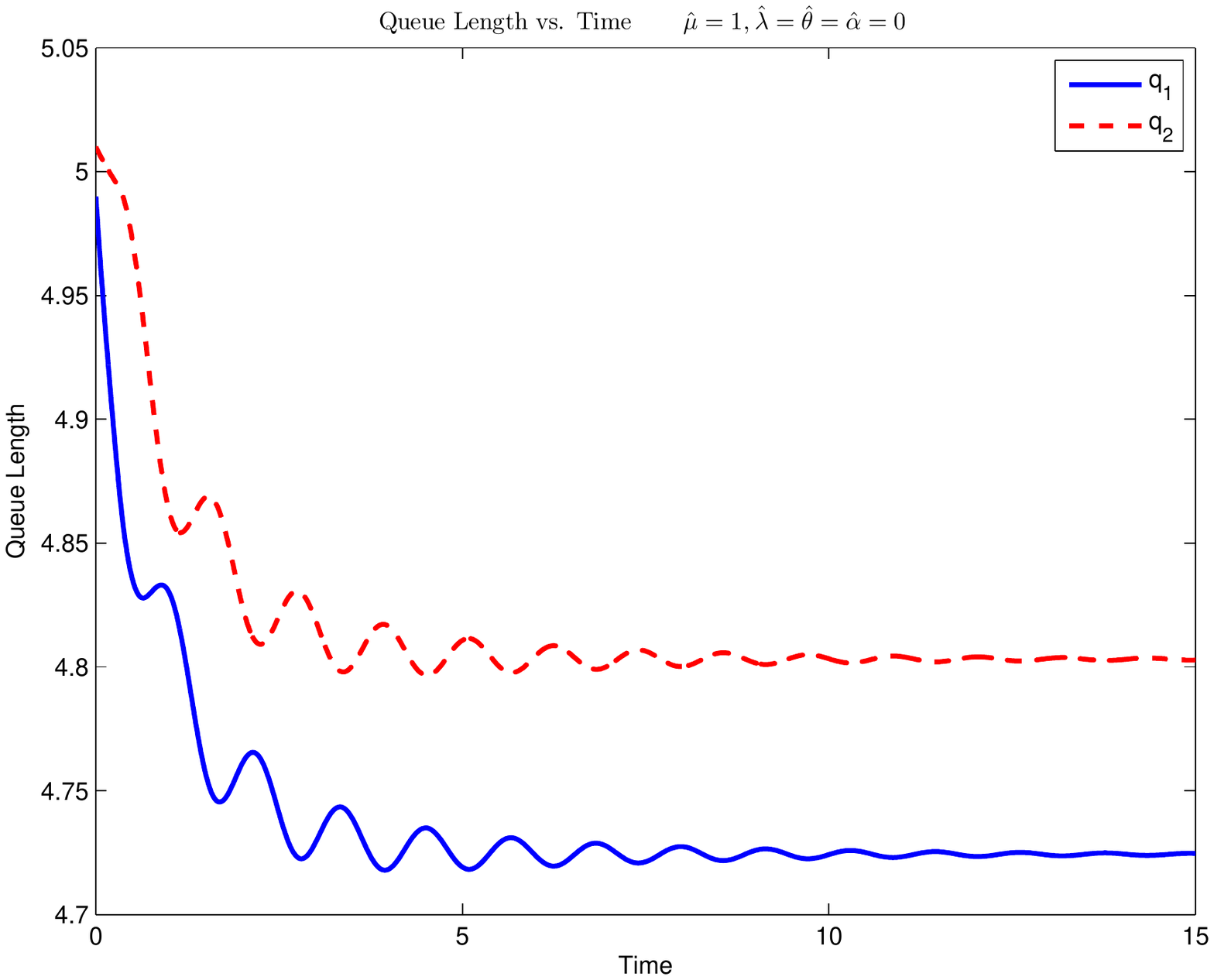}~\hspace{-10mm}~\includegraphics[scale=.4]{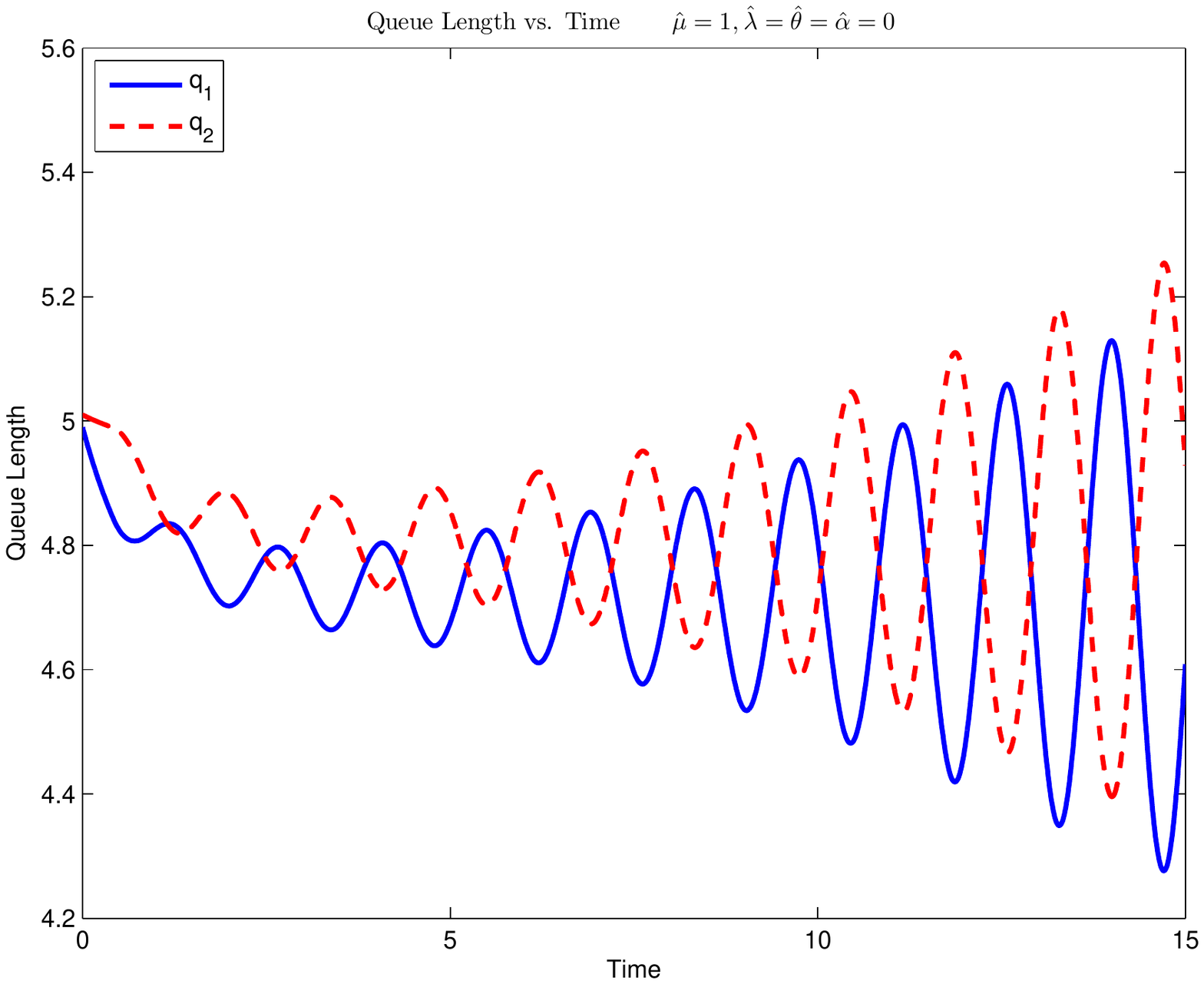}
%\captionsetup{justification=centering,margin=2cm}
\vspace{-30mm}
  \caption{$\hat{\mu} = 1, \hat{\lambda}  = \hat{\theta} = \hat{\alpha} = 0$, $\Delta_{\text{mod}} \approx .3646$
$\epsilon = .1, \lambda = 10, \mu=1, \theta=1, \alpha=0$\\ On $[-\Delta, 0]$, $q_1 = 4.99$, $q_2 = 5.01$, Left: $\Delta = \Delta_{\text{mod}} - .05$, Right: $\Delta = \Delta_{\text{mod}} + .05$}
\label{Fig9}
\end{figure}

%\begin{figure}[ht!]
%  \centering
%\vspace{-55mm}
%  \includegraphics[scale=.7]{./Code/Paper_Figures/Fig_2.pdf}
%\vspace{-50mm}
%\captionsetup{justification=centering,margin=2cm}
%  \caption{$\Delta_{\text{mod}} \approx .3646, \Delta = \Delta_{\text{mod}} - .05$\\
%$\epsilon = .1, \lambda = 10, \mu=1, \theta=1, \alpha=0$\\ History function is constant with $q_1 = 4.99$ and $q_2 = 5.01$}
%\vspace{-40mm}
%  \includegraphics[scale=.7]{./Code/Paper_Figures/Fig_6.pdf}
%\vspace{-50mm}
%  \caption{$\Delta_{\text{mod}} \approx .3646, \Delta = \Delta_{\text{mod}} + .05$\\
%$\epsilon = .1, \lambda = 10, \mu=1, \theta=1, \alpha=0$\\ History function is constant with $q_1 = 4.99$ and $q_2 = 5.01$}
%\end{figure}

%%%%%%%%%%%%%%%%%

%\begin{figure}[ht!]
%  \centering
%  \includegraphics[scale=.515]{./Code/Paper_Figures/Fig_3.pdf}
%\vspace{-40mm}
%\captionsetup{justification=centering,margin=2cm}
%  \caption{$\Delta_{\text{mod}} \approx .3408, \Delta = \Delta_{\text{mod}} - .05$\\
%$\epsilon = .1, \lambda = 10, \mu=1, \theta=1, \alpha=0$\\ History function is constant with $q_1 = 4.99$ and $q_2 = 5.01$}
%\end{figure}
%
%\begin{figure}[hb!]
%  \centering
%  \includegraphics[scale=.515]{./Code/Paper_Figures/Fig_7.pdf}
%\vspace{-40mm}
%\captionsetup{justification=centering,margin=2cm}
%  \caption{$\Delta_{\text{mod}} \approx .3408, \Delta = \Delta_{\text{mod}} + .05$\\
%$\epsilon = .1, \lambda = 10, \mu=1, \theta=1, \alpha=0$\\ History function is constant with $q_1 = 4.99$ and $q_2 = 5.01$}
%\end{figure}

%%%%%%%%%%%%%%%

\begin{figure}[ht!]
\vspace{-35mm}
  \hspace{-5mm}~\includegraphics[scale=.4]{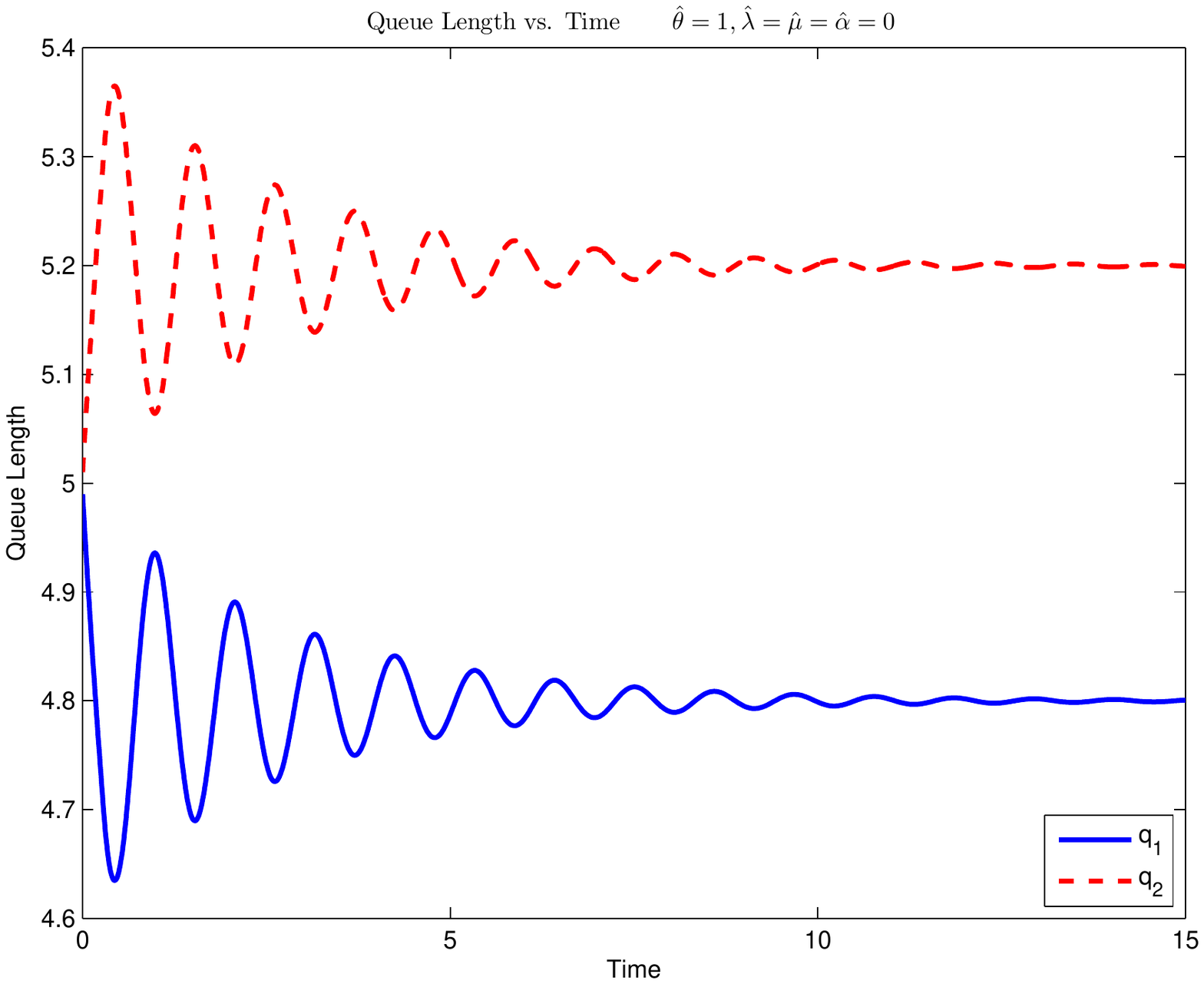}~\hspace{-10mm}~\includegraphics[scale=.4]{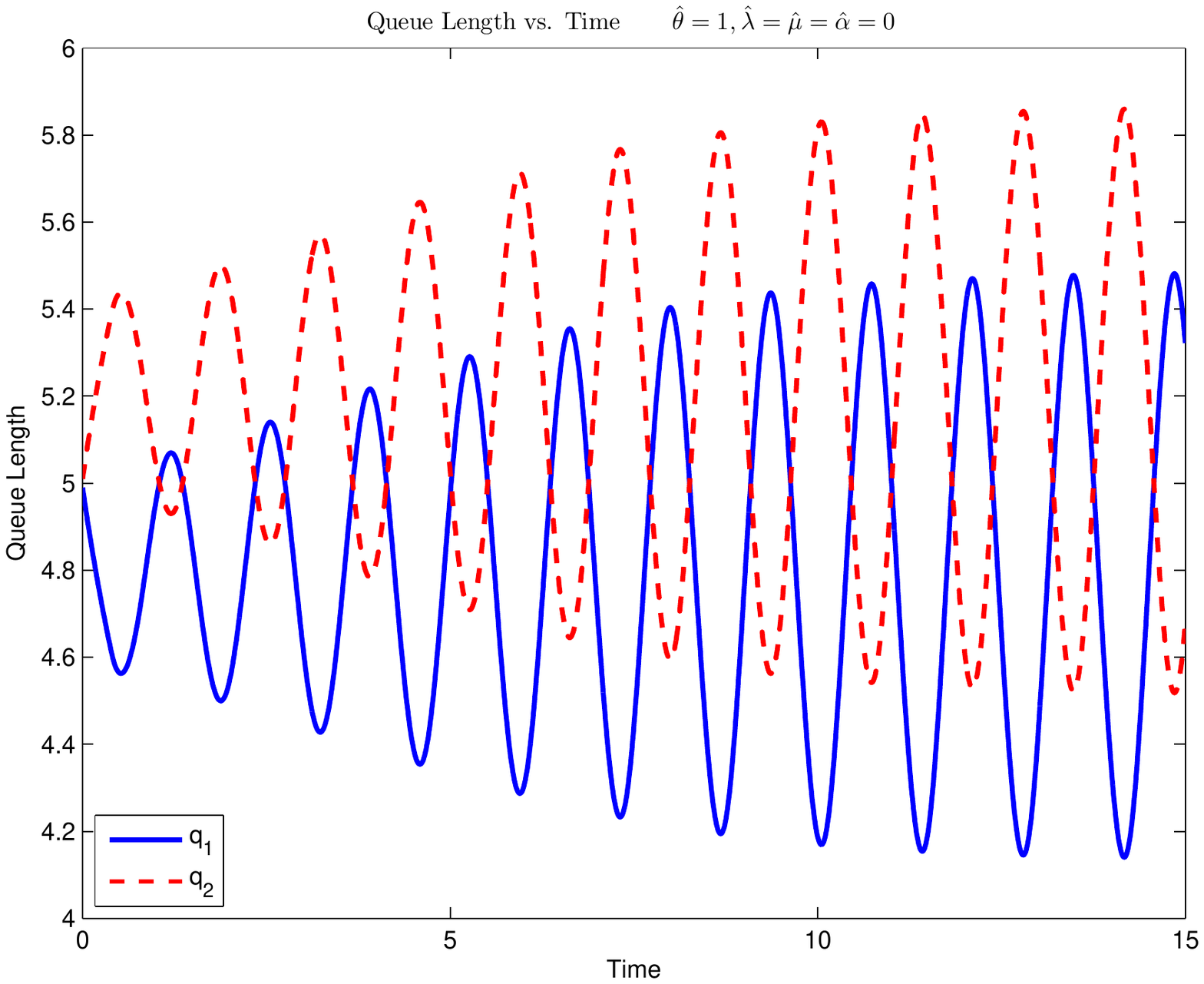}
%\captionsetup{justification=centering,margin=2cm}
\vspace{-30mm}
  \caption{$\hat{\theta} = 1, \hat{\lambda} = \hat{\mu} = \hat{\alpha} = 0$, $\Delta_{\text{mod}} \approx .3408$ $\epsilon = .1, \lambda = 10, \mu=1, \theta=1, \alpha=0$\\ On $[-\Delta, 0]$, $q_1 = 4.99$, $q_2 = 5.01$, Left: $\Delta = \Delta_{\text{mod}} - .05$, Right: $\Delta = \Delta_{\text{mod}} + .05$}
\label{Fig10}
\end{figure}

%\begin{figure}[ht!]
%  \centering
%\vspace{-55mm}
%  \includegraphics[scale=.7]{./Code/Paper_Figures/Fig_3.pdf}
%\vspace{-50mm}
%\captionsetup{justification=centering,margin=2cm}
%  \caption{$\Delta_{\text{mod}} \approx .3408, \Delta = \Delta_{\text{mod}} - .05$\\
%$\epsilon = .1, \lambda = 10, \mu=1, \theta=1, \alpha=0$\\ History function is constant with $q_1 = 4.99$ and $q_2 = 5.01$}
%\vspace{-40mm}
%  \includegraphics[scale=.7]{./Code/Paper_Figures/Fig_7.pdf}
%\vspace{-50mm}
%  \caption{$\Delta_{\text{mod}} \approx .3408, \Delta = \Delta_{\text{mod}} + .05$\\
%$\epsilon = .1, \lambda = 10, \mu=1, \theta=1, \alpha=0$\\ History function is constant with $q_1 = 4.99$ and $q_2 = 5.01$}
%\end{figure}

%%%%%%%%%%%%%%

%\begin{figure}[ht!]
%  \centering
%  \includegraphics[scale=.515]{./Code/Paper_Figures/Fig_4.pdf}
%\vspace{-40mm}
%\captionsetup{justification=centering,margin=2cm}
%  \caption{$\Delta_{\text{mod}} \approx .3617, \Delta = \Delta_{\text{mod}} - .05$\\
%$\epsilon = .1, \lambda = 10, \mu=1, \theta=1, \alpha=0$\\ History function is constant with $q_1 = 4.99$ and $q_2 = 5.01$}
%\end{figure}
%
%\begin{figure}[hb!]
%  \centering
%  \includegraphics[scale=.515]{./Code/Paper_Figures/Fig_8.pdf}
%\vspace{-40mm}
%\captionsetup{justification=centering,margin=2cm}
%  \caption{$\Delta_{\text{mod}} \approx .3617, \Delta = \Delta_{\text{mod}} + .05$\\
%$\epsilon = .1, \lambda = 10, \mu=1, \theta=1, \alpha=0$\\ History function is constant with $q_1 = 4.99$ and $q_2 = 5.01$}
%\end{figure}

%%%%%%%

\begin{figure}[ht!]
\vspace{-35mm}
  \hspace{-5mm}~\includegraphics[scale=.4]{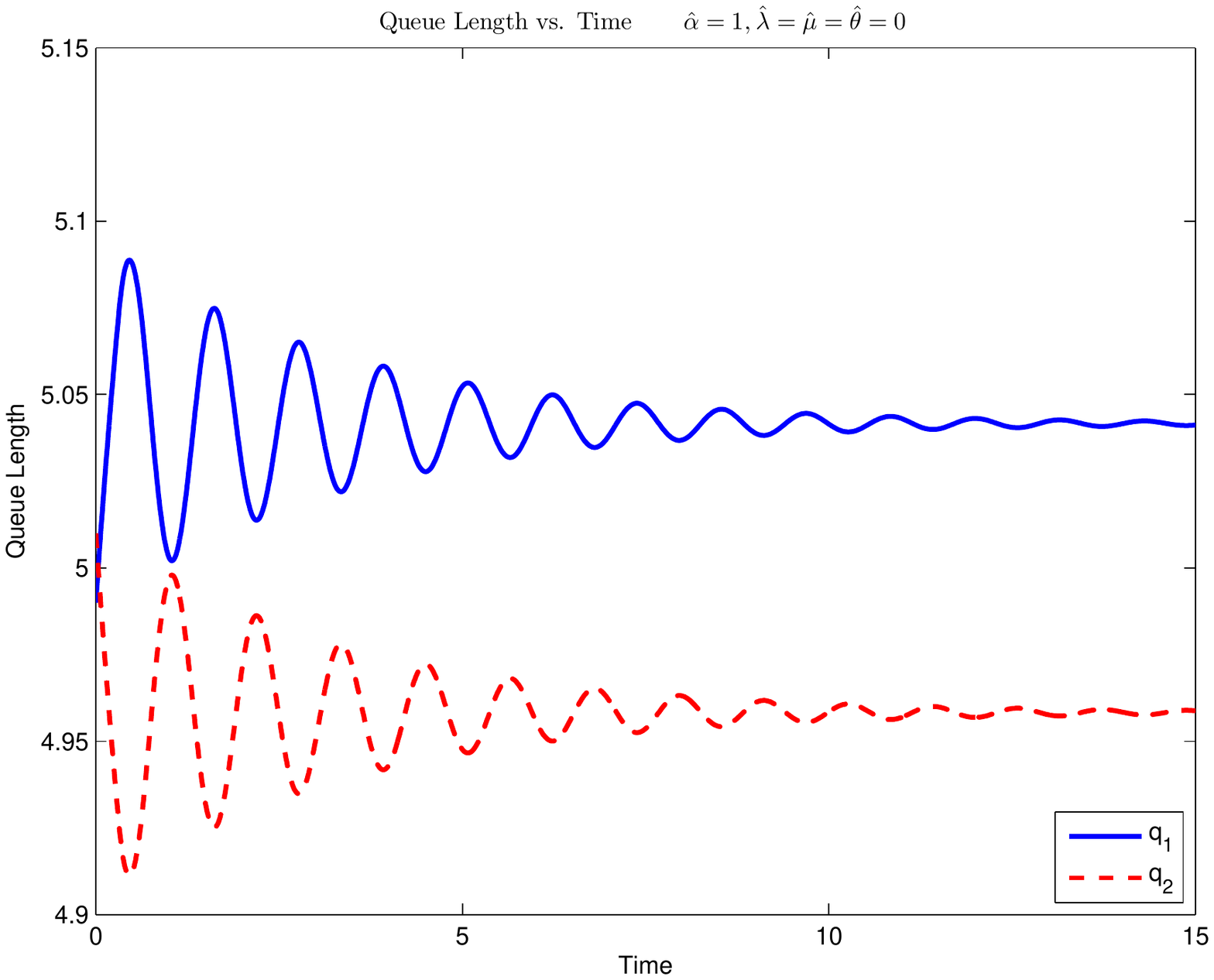}~\hspace{-10mm}~\includegraphics[scale=.4]{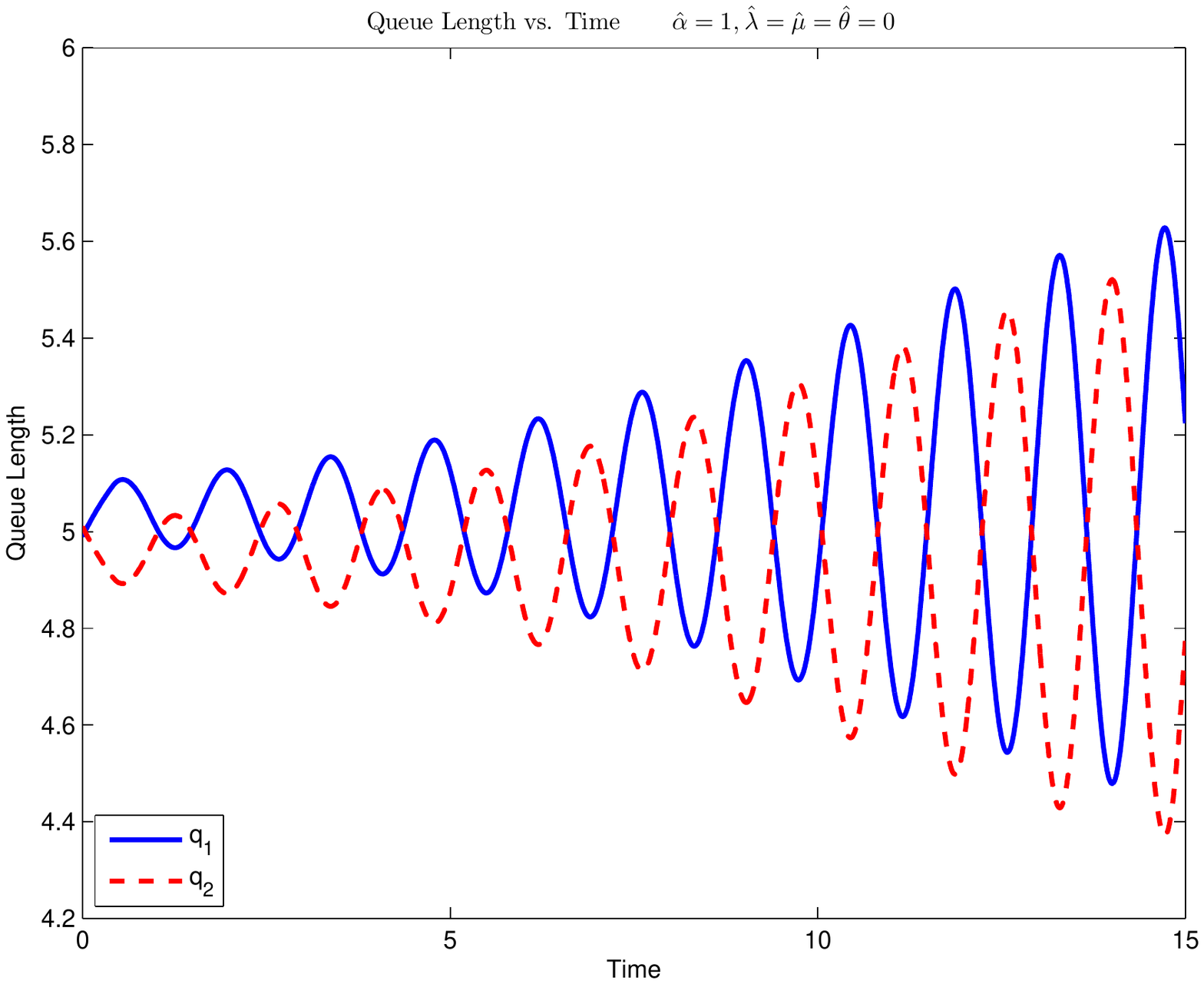}
%\captionsetup{justification=centering,margin=2cm}
\vspace{-30mm}
  \caption{$\hat{\alpha} = 1, \hat{\lambda} = \hat{\mu} = \hat{\theta} = 0$, $\Delta_{\text{mod}} \approx .3617$
$\epsilon = .1, \lambda = 10, \mu=1, \theta=1, \alpha=0$\\ On $[-\Delta, 0]$, $q_1 = 4.99$, $q_2 = 5.01$, Left: $\Delta = \Delta_{\text{mod}} - .05$, Right: $\Delta = \Delta_{\text{mod}} + .05$}
\label{Fig11}
\end{figure}

%
%\begin{figure}[ht!]
%  \centering
%\vspace{-55mm}
%  \includegraphics[scale=.7]{./Code/Paper_Figures/Fig_4.pdf}
%\vspace{-50mm}
%\captionsetup{justification=centering,margin=2cm}
%  \caption{$\Delta_{\text{mod}} \approx .3617, \Delta = \Delta_{\text{mod}} - .05$\\
%$\epsilon = .1, \lambda = 10, \mu=1, \theta=1, \alpha=0$\\ History function is constant with $q_1 = 4.99$ and $q_2 = 5.01$}
%\vspace{-40mm}
%  \includegraphics[scale=.7]{./Code/Paper_Figures/Fig_8.pdf}
%\vspace{-50mm}
%  \caption{$\Delta_{\text{mod}} \approx .3617, \Delta = \Delta_{\text{mod}} + .05$\\
%$\epsilon = .1, \lambda = 10, \mu=1, \theta=1, \alpha=0$\\ History function is constant with $q_1 = 4.99$ and $q_2 = 5.01$}
%\end{figure}

%%%%%%%%

%\begin{figure}[ht!]
%  \centering
%  \includegraphics[scale=.515]{./Code/Paper_Figures/Fig_general2.pdf}
%\vspace{-40mm}
%\captionsetup{justification=centering,margin=2cm}
%  \caption{$\Delta_{\text{mod}} \approx .3416, \Delta = \Delta_{\text{mod}} - .05$\\
%$\epsilon = .1, \lambda = 10, \mu=1, \theta=1, \alpha=0$\\ History function is constant with $q_1 = 4.99$ and $q_2 = 5.01$}
%\end{figure}
%
%\begin{figure}[hb!]
%  \centering
%  \includegraphics[scale=.515]{./Code/Paper_Figures/Fig_general1.pdf}
%\vspace{-40mm}
%\captionsetup{justification=centering,margin=2cm}
%  \caption{$\Delta_{\text{mod}} \approx .3416, \Delta = \Delta_{\text{mod}} + .05$\\
%$\epsilon = .1, \lambda = 10, \mu=1, \theta=1, \alpha=0$\\ History function is constant with $q_1 = 4.99$ and $q_2 = 5.01$}
%\end{figure}

%%%%%%%

\begin{figure}[ht!]
\vspace{-35mm}
  \hspace{-5mm}~\includegraphics[scale=.4]{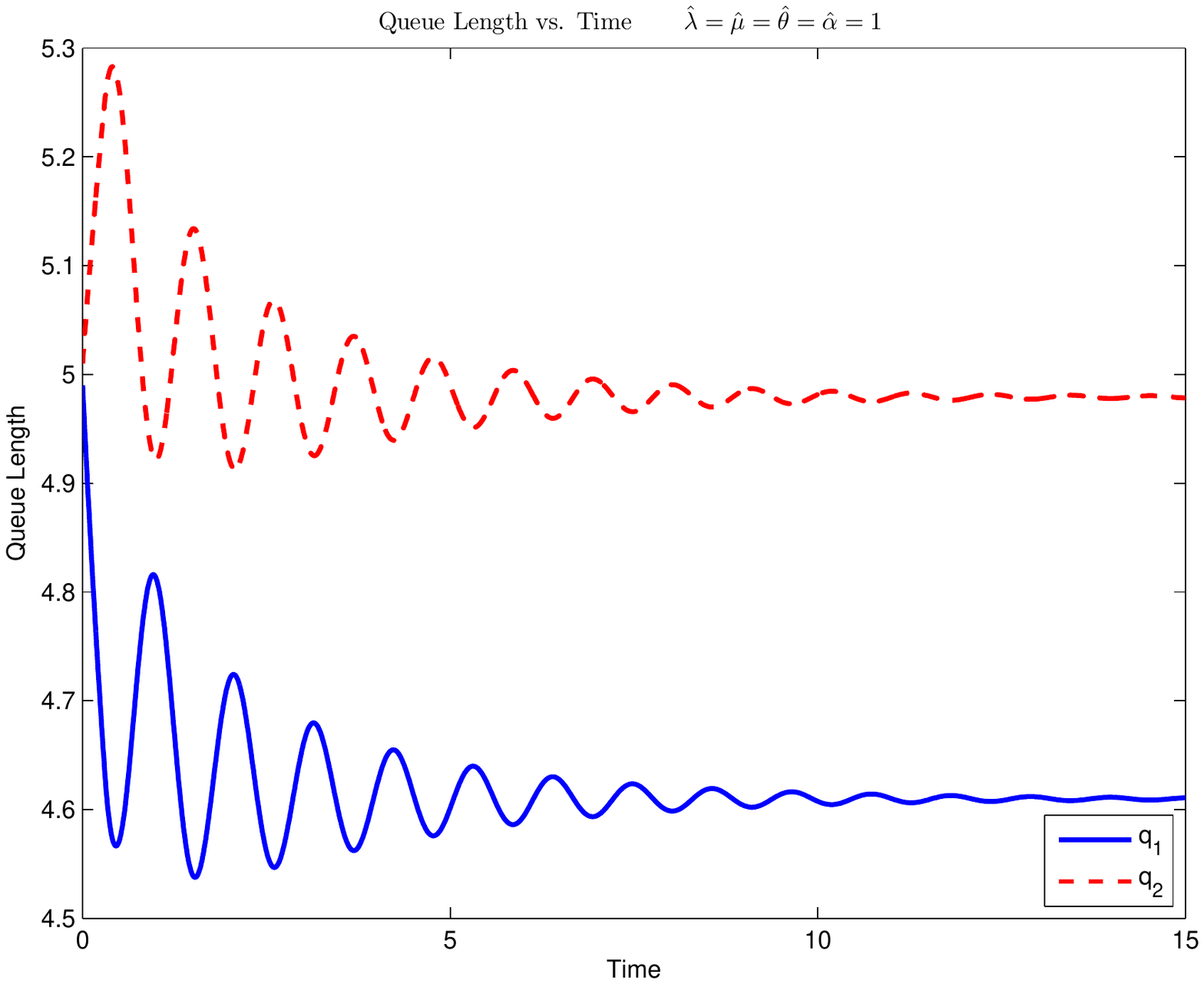}~\hspace{-10mm}~\includegraphics[scale=.4]{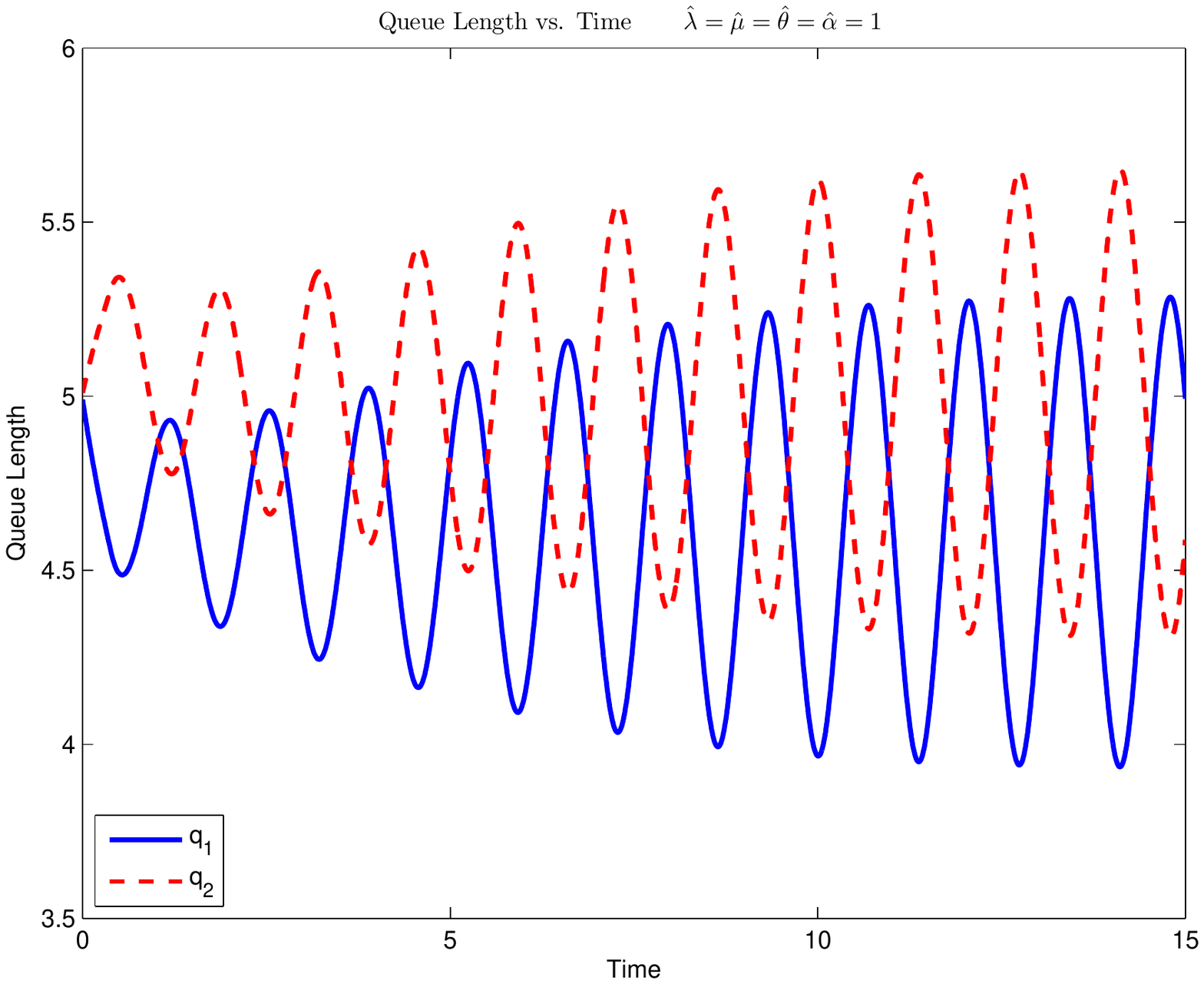}
%\captionsetup{justification=centering,margin=2cm}
\vspace{-30mm}
  \caption{$\hat{\lambda} = \hat{\mu} = \hat{\theta} = \hat{\alpha} = 1$, $\Delta_{\text{mod}} \approx .3416$
$\epsilon = .1, \lambda = 10, \mu=1, \theta=1, \alpha=0$\\ On $[-\Delta, 0]$, $q_1 = 4.99$, $q_2 = 5.01$, Left: $\Delta = \Delta_{\text{mod}} - .05$, Right: $\Delta = \Delta_{\text{mod}} + .05$}
\label{Fig12}
\end{figure}

%
%\begin{figure}[ht!]
%  \centering
%\vspace{-55mm}
%  \includegraphics[scale=.7]{./Code/Paper_Figures/Fig_general2.pdf}
%\vspace{-50mm}
%\captionsetup{justification=centering,margin=2cm}
%  \caption{$\Delta_{\text{mod}} \approx .3416, \Delta = \Delta_{\text{mod}} - .05$\\
%$\epsilon = .1, \lambda = 10, \mu=1, \theta=1, \alpha=0$\\ History function is constant with $q_1 = 4.99$ and $q_2 = 5.01$}
%\vspace{-40mm}
%  \includegraphics[scale=.7]{./Code/Paper_Figures/Fig_general1.pdf}
%\vspace{-50mm}
%  \caption{$\Delta_{\text{mod}} \approx .3416, \Delta = \Delta_{\text{mod}} + .05$\\
%$\epsilon = .1, \lambda = 10, \mu=1, \theta=1, \alpha=0$\\ History function is constant with $q_1 = 4.99$ and $q_2 = 5.01$}
%\end{figure}

%%%%%%%%%%%

%**********************************************************************************************
%**********************************************************************************************

\section{Amplitude of Limit Cycle} \label{sec_amplitude}

In the previous section, we observed that increasing the delay past the critical value $\Delta_{\text{mod}}$ causes oscillations in the queue lengths and ultimately gives rise to a limit cycle. In this section, we aim to approximate the amplitude of the limit cycle when the delay is close to $\Delta_{\text{mod}}$. To do this we will resort to using Lindstedt's method.

\begin{theorem}

Using Lindstedt's method, we obtain the following approximation, $\tilde{A}$, of the amplitude of limit cycles near the critical delay $$\tilde{A} = \sqrt{\Delta - \Delta_{\text{mod}}}  \  \sqrt{ \frac{8 (\bar{\lambda}^2 \bar{\theta}^2 - 4 \bar{\mu}^2)^2}{2 \bar{\lambda}^2 \bar{\mu} \theta^2 \bar{\theta}^2 - 8 \bar{\mu}^3 \theta^2 + \left[ \lambda \theta^2 \bar{\lambda} \bar{\theta} + \frac{4 \lambda \theta^2}{\bar{\lambda} \bar{\theta}} \bar{\mu}^2 (\theta - 1)  \right] \sqrt{\bar{\lambda}^2 \bar{\theta}^2 - 4 \bar{\mu}^2} \arccos \left(- \frac{2 \bar{\mu}}{\bar{\lambda} \bar{\theta}} \right)     }   }.$$

\end{theorem}

\begin{proof}

We will expand the system of Equations \ref{perturbed equation 1}-\ref{perturbed equation 2} about the approximate  equilibrium point $$(q_1^*, q_2^*) = \left(\frac{\lambda}{2 \mu} + a\epsilon + O(\epsilon^2), \frac{\lambda}{2 \mu} + b \epsilon + O(\epsilon^2)\right)$$ where $a$ and $b$ are as defined in Theorem \ref{equilibrium theorem}.  However, unlike the equilibrium and stability calculations, we need to Taylor expand to third order in order to find the amplitude. By Taylor expanding to third order (cubic) and dropping $O(\epsilon^2)$ terms leaves us with the following cubic system of DDEs

\begin{align*}
\overset{\bullet}{u}_1(t) &= -(\mu + \hat{\mu} \epsilon) u_1(t) - \frac{1}{4} \left( \lambda \theta + (\lambda \hat{\theta} + \hat{\lambda} \theta) \epsilon \right) u_1(t - \Delta) + \frac{1}{4} \left(  \lambda \theta + \theta \hat{\lambda} \epsilon \right)u_2(t - \Delta)\\ &+ \frac{1}{32 \mu} (\lambda^2 \theta^2 \hat{\theta} + 2 \mu (a-b) \lambda \theta^3 - 2 \mu \hat{\alpha} \lambda \theta^2) \epsilon (u_1(t- \Delta) - u_2(t - \Delta))^2\\ &+  \frac{1}{48} \left( \lambda \theta^3 + (3 \lambda \theta^2 \hat{\theta} + \theta^3 \hat{\lambda}) \epsilon \right) u_1^3(t-\Delta) - \frac{1}{48} \left( \lambda \theta^3 + \theta^3 \hat{\lambda} \epsilon  \right) u_2^3(t - \Delta)\\ &- \frac{1}{16} (\lambda \theta^3 + (2 \lambda \theta^2 \hat{\theta} + \theta^3 \hat{\lambda})\epsilon) u_1^2(t - \Delta) u_2(t-\Delta) + \frac{1}{16} (\lambda \theta^3 + (\lambda \theta^2 \hat{\theta} + \theta^3 \hat{\lambda}) \epsilon) u_1(t - \Delta) u_2^2(t-\Delta)\\
\overset{\bullet}{u}_2(t) & = - \mu u_2(t) + \frac{1}{4} (\lambda \theta + \lambda \hat{\theta} \epsilon) u_1(t-\Delta) - \frac{1}{4} \lambda \theta u_2(t-\Delta)\\ &+ \frac{1}{32 \mu} (\lambda^2 \theta^2 \hat{\theta} + 2 \mu (a-b) \lambda \theta^3 - 2 \mu \hat{\alpha} \lambda \theta^2) \epsilon (u_1^2(t-\Delta) + 2 u_1(t-\Delta) u_2(t-\Delta) - u_2^2(t-\Delta)) \\ &- \frac{1}{48} (\lambda \theta^2 + 3 \lambda \theta^2 \hat{\theta} \epsilon) u_1^3(t - \Delta) + \frac{1}{48} \lambda \theta^3 u_2^3(t - \Delta) \\ &+ \frac{1}{16} (\lambda \theta^3 + 2 \lambda \theta^2 \hat{\theta} \epsilon) u_1^2(t - \Delta) u_2(t - \Delta) - \frac{1}{16} (\lambda \theta^3 + \lambda \theta^2 \hat{\theta} \epsilon) u_1(t - \Delta) u_2^2(t - \Delta).
\end{align*}

\noindent As we did in the linear case, we make the change of variables $$v_1(t) = u_1(t) + u_2(t), \hspace{5mm} v_2(t) = u_1(t) - u_2(t)$$ and we let $$v_1(t) = v_{1,0}(t) + \epsilon v_{1,1}(t) + O(\epsilon^2)$$ $$v_2(t) = v_{2,0}(t) + \epsilon v_{2,1}(t) + O(\epsilon^2).$$ Collecting $O(1)$ terms gives us the following two equations

\begin{align}
\overset{\bullet}{v}_{1,0}(t) + \mu v_{1,0}(t) &= 0 \label{v10 equation 3}\\
\overset{\bullet}{v}_{2,0}(t) + \frac{\lambda \theta}{2} v_{2,0}(t-\Delta) - \frac{\lambda \theta^3}{24} v_{2,0}^3(t-\Delta)+ \mu v_{2,0}(t) &= 0 \label{v20 equation 3}
\end{align}

\noindent and collecting $O(\epsilon)$ terms gives us 

\begin{align}
\overset{\bullet}{v}_{1,1}(t) &+ \mu v_{1,1}(t) = - \frac{\theta \hat{\lambda}}{4} v_{2,0}(t - \Delta) + \frac{\theta^3 \hat{\lambda}}{48} v_{2,0}^3(t - \Delta) - \frac{\hat{\mu}}{2} (v_{1,0}(t) + v_{2,0}(t) ) \label{v11 equation} \\
\overset{\bullet}{v}_{2,1}(t) &+ \frac{\lambda \theta}{2} v_{2,1}(t - \Delta) - \frac{\lambda \theta^3}{8} v_{2,0}^2(t - \Delta) v_{2,1}(t - \Delta) + \mu v_{2,1} = - \left( \frac{\theta \hat{\lambda}}{4} + \frac{\lambda \hat{\theta}}{4} \right) v_{2,0}(t - \Delta) \nonumber \\ &- \frac{\lambda \hat{\theta}}{4} v_{1,0}(t - \Delta)  + \left( \frac{2 \lambda \theta^3 \mu (a - b) + \lambda^2 \theta^2 \hat{\theta} - 2 \lambda \theta^2 \mu \hat{\alpha}}{16 \mu}  \right) v_{2,0}^2(t - \Delta) \nonumber \\ &+ \frac{\theta^3 \hat{\lambda}}{48} v_{2,0}^3(t -\Delta) + \frac{\lambda \theta^2 \hat{\theta}}{16} [ v_{1,0}(t - \Delta) v_{2,0}^2(t - \Delta) + v_{2,0}^3(t - \Delta)   ] - \frac{\hat{\mu}}{2}(v_{1,0}(t) + v_{2,0}(t)).
\end{align}

\noindent We observe that we can directly solve Equation \ref{v10 equation 3}  and its solution is given by $v_{1,0}(t) = \hat{c} \exp(- \mu t)$, for some constant $\hat{c}$, meaning $$v_{1}(t) = u_1(t) + u_2(t) = \hat{c} \exp(- \mu t) + O(\epsilon)$$ so that for small $\epsilon$ and large time, we have that $$u_1(t) \approx - u_2(t)$$ which tells us that the amplitudes of the queue lengths are approximately symmetric (which is expected given that we're perturbing a symmetric model) and thus $$v_{2,0}(t) = u_1(t) - u_2(t) + O(\epsilon) \approx 2 u_1(t) $$ gives us approximately twice the amplitude of the limit cycle. Because of this, we narrow our interest to Equation \ref{v20 equation 3}. Since we are interested in the amplitudes of limit cycles near the bifurcation point, we are working under the assumption that $\Delta - \Delta_{\text{mod}}$ is small. Letting  $\Delta_{0} = \Delta_{\text{mod}}$ and $\omega_0 = \omega_{\text{mod}}$, we use the following transformations.

 $$\tau = \omega t, \hspace{5mm} v_{2,0}(t) = \sqrt{\epsilon} v(t)$$ $$v(t) = v_0(t) + \epsilon v_1(t) + \cdots, \hspace{5mm} \Delta = \Delta_0 + \epsilon \Delta_1 + \cdots, \hspace{5mm} \omega = \omega_0 + \epsilon \omega_1 + \cdots$$ Matching powers of $\epsilon$ and dropping higher ordered terms, we get two equations.

\begin{align}
\omega_0 v_0'(\tau) + \frac{\lambda \theta}{2} v_0(\tau - \omega_0 \Delta_0) + \mu v_0(\tau) &= 0 \label{eqn_2} \\ 
\omega_0 v_1'(\tau) + \frac{\lambda \theta}{2} v_1(\tau - \omega_0 \Delta_0) + \mu v_1(\tau) &= - \omega_1 v_0'(\tau) \nonumber \\ &+ \frac{\lambda \theta}{2}(\omega_0 \Delta_1 + \omega_1 \Delta_0) v_0'(\tau - \omega_0 \Delta_0) \nonumber \\ &+ \frac{\lambda \theta^3}{24} v_0^3(\tau - \omega_0 \Delta_0) \label{lindstedt inhomogeneous equation 1}
\end{align}

\noindent Noting that $v_0(\tau) = A \sin(\tau)$ satisfies Equation \ref{eqn_2} and that the homogeneous form of Equation \ref{lindstedt inhomogeneous equation 1} is the same as that of Equation \ref{eqn_2}, we substitute in $v_0(\tau) = A \sin(\tau)$ into the inhomogeneity of Equation \ref{lindstedt inhomogeneous equation 1} and set the terms that would introduce secular terms in the particular solution equal to $0$. Doing this yields a system of two equations and two unknowns: $A$ and $\omega_1$. Solving for these unknowns, we obtain %\textcolor{blue}{Philip, where is the result for $\omega_1$.  You mention two unknowns and two equations.  I have added an empty equation slot for it.}

\begin{equation}
A = \sqrt{ \frac{8 \Delta_1 (\bar{\lambda}^2 \bar{\theta}^2 - 4 \bar{\mu}^2)^2}{2 \bar{\lambda}^2 \bar{\mu} \theta^2 \bar{\theta}^2 - 8 \bar{\mu}^3 \theta^2 + \left[ \lambda \theta^2 \bar{\lambda} \bar{\theta} + \frac{4 \lambda \theta^2}{\bar{\lambda} \bar{\theta}} \bar{\mu}^2 (\theta - 1)  \right] \sqrt{\bar{\lambda}^2 \bar{\theta}^2 - 4 \bar{\mu}^2} \arccos \left(- \frac{2 \bar{\mu}}{\bar{\lambda} \bar{\theta}} \right)     }   }
\end{equation} 
and 
\begin{equation}
\omega_1  =  - \frac{1}{\Delta_{0}} \left( \omega_0 \Delta_1 + \frac{A^2 \theta^2 \cos(\omega_0 \Delta_0)}{16 \sin(\omega_0 \Delta_0)}  \right)
\end{equation} 
where

$$\bar{\lambda} := \lambda + \frac{\hat{\lambda} \epsilon}{2}, \hspace{5mm} \bar{\mu} := \mu + \frac{\hat{\mu} \epsilon}{2}, \hspace{5mm} \bar{\theta} := \theta + \frac{\hat{\theta} \epsilon}{2}.$$ We assume that $\Delta - \Delta_0 \approx \epsilon$ which implies that $\Delta_1 \approx 1$ and since $v_{2,0}(t) = \sqrt{\epsilon} v(t)$, our approximation of the amplitude, $\tilde{A}$, is $$\tilde{A} = \sqrt{\Delta - \Delta_{\text{mod}}} \sqrt{ \frac{8 (\bar{\lambda}^2 \bar{\theta}^2 - 4 \bar{\mu}^2)^2}{2 \bar{\lambda}^2 \bar{\mu} \theta^2 \bar{\theta}^2 - 8 \bar{\mu}^3 \theta^2 + \left[ \lambda \theta^2 \bar{\lambda} \bar{\theta} + \frac{4 \lambda \theta^2}{\bar{\lambda} \bar{\theta}} \bar{\mu}^2 (\theta - 1)  \right] \sqrt{\bar{\lambda}^2 \bar{\theta}^2 - 4 \bar{\mu}^2} \arccos \left(- \frac{2 \bar{\mu}}{\bar{\lambda} \bar{\theta}} \right)     }   }.$$

\end{proof}

We now demonstrate numerically how well this approximation matches the actual amplitude of the limit cycle. In each figure below, we will plot the queue lengths against time and approximate the maximum and minimum values of each queue length by adding or subtracting $\frac{1}{2} \tilde{A}$ from the equilibrium for each queue. The cycle lines are the approximations corresponding to $q_1$ and the dashed green lines are the approximations corresponding to $q_2$.

\begin{figure}[ht!]
\vspace{-50mm}
  \hspace{-27mm}~\includegraphics[scale=.6]{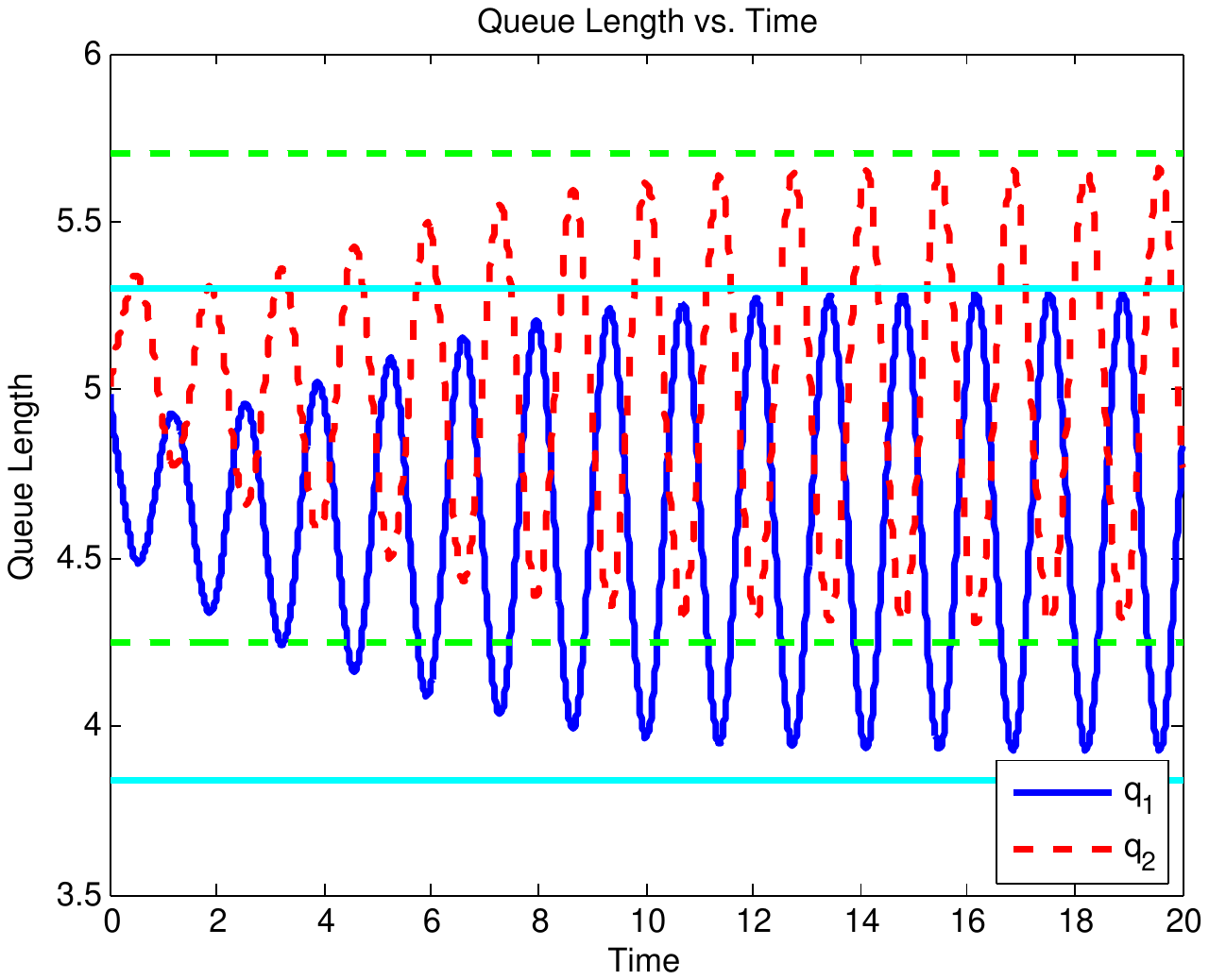}~\hspace{-40mm}~\includegraphics[scale=.6]{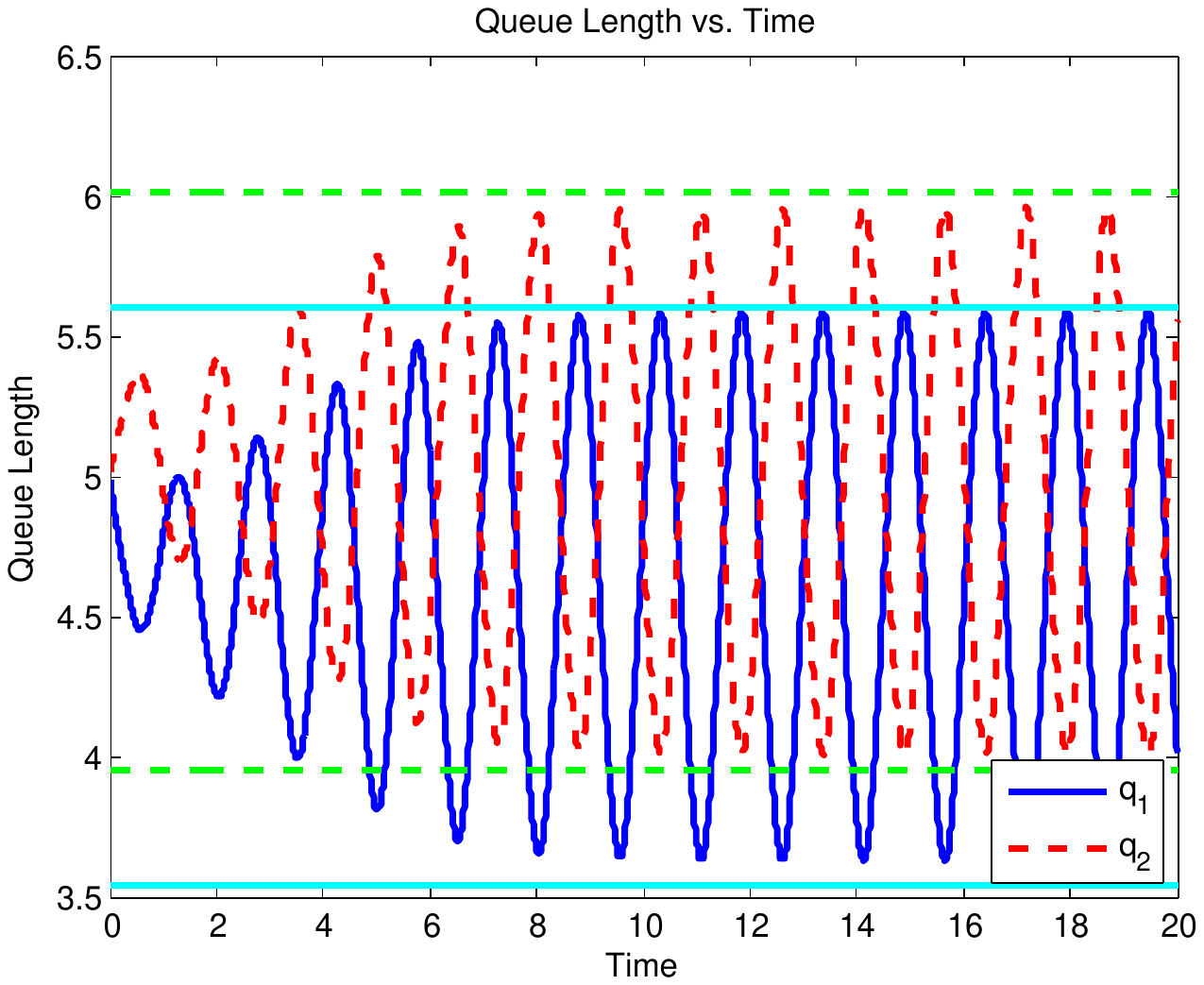}
%\captionsetup{justification=centering,margin=2cm}
\vspace{-55mm}
  \caption{ \\$\hat{\lambda} = \hat{\mu} = \hat{\theta} = \hat{\alpha} = 1, \epsilon = .1, \lambda=10, \mu=1, \theta=1, \alpha=1$, on $[-\Delta, 0]$ $q_1 = 4.99$ and $q_2 = 5.01$\\
Left: $\Delta - \Delta_{\text{mod}} = .05$, $ \text{Amplitudes:   } q_1 \approx  1.3593, q_2  \approx 1.3524$, $\text{Approximation } \approx 1.4562$\\
Right: $\Delta - \Delta_{\text{mod}} \approx .1$ $ \text{Amplitudes:   } q_1 \approx  1.9576, q_2  \approx 1.9503$, $\text{Approximation } \approx 2.0594$}
\label{Fig13}
\end{figure}

\begin{figure}[ht!]
\vspace{-50mm}
  \hspace{-27mm}~\includegraphics[scale=.6]{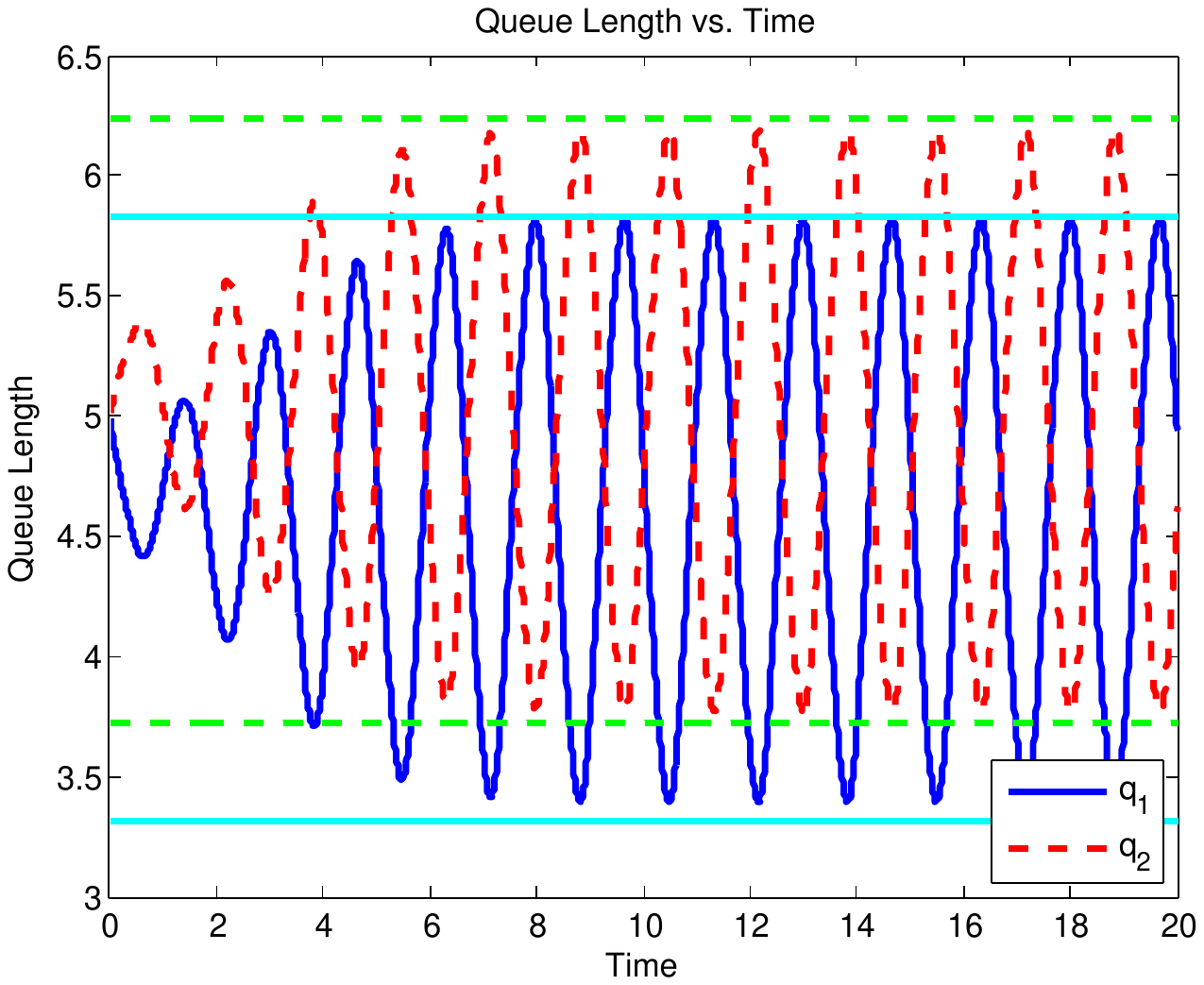}~\hspace{-40mm}~\includegraphics[scale=.6]{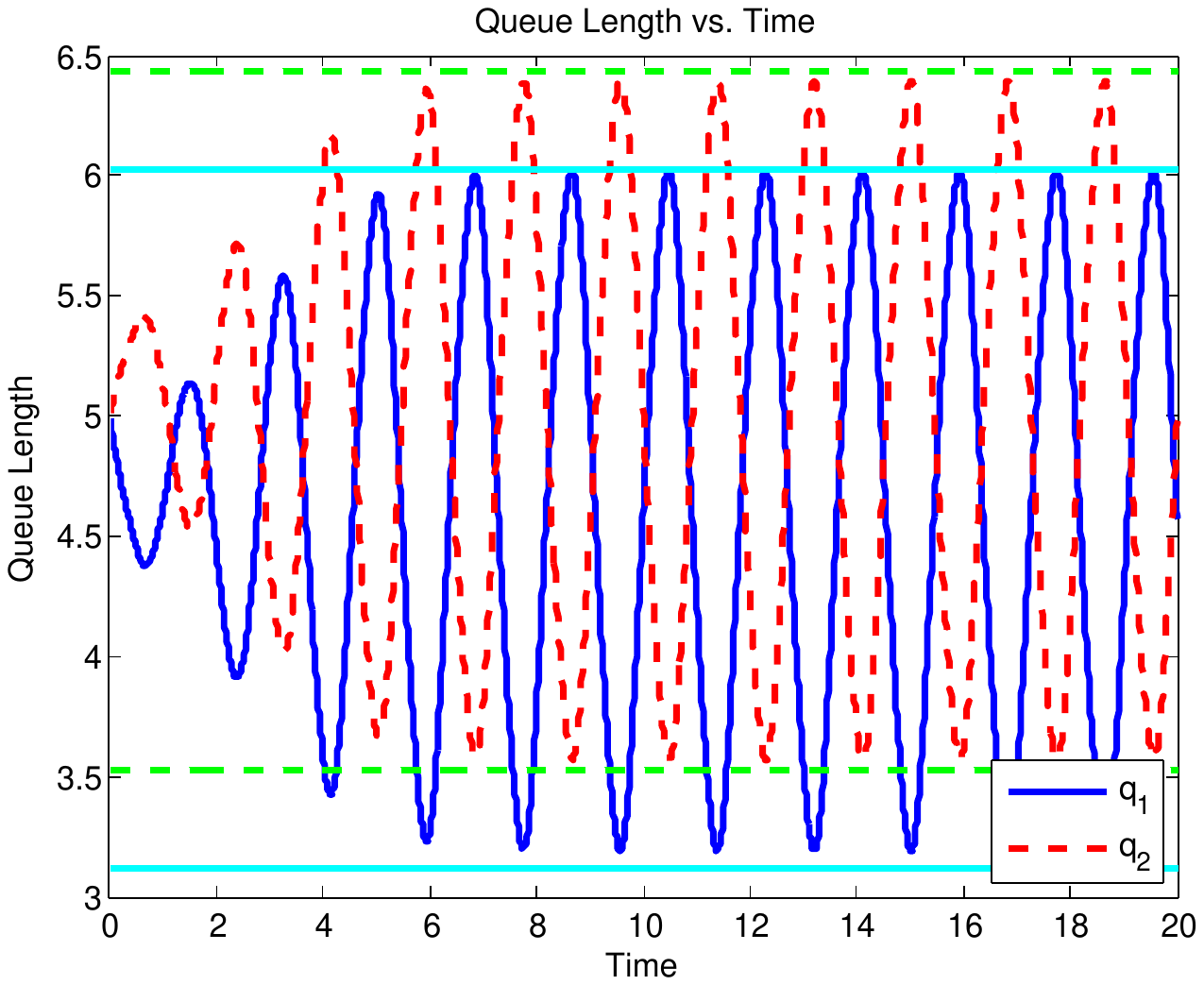}
%\captionsetup{justification=centering,margin=2cm}
\vspace{-55mm}
  \caption{ \\$\hat{\lambda} = \hat{\mu} = \hat{\theta} = \hat{\alpha} = 1, \epsilon = .1, \lambda=10, \mu=1, \theta=1, \alpha=1$, on $[-\Delta, 0]$ $q_1 = 4.99$ and $q_2 = 5.01$\\
Left: $\Delta - \Delta_{\text{mod}} = .15$, $ \text{Amplitudes:   } q_1 \approx  2.4265, q_2  \approx 2.4208$, $\text{Approximation } \approx 2.5222$\\
Right: $\Delta - \Delta_{\text{mod}} \approx .2$, $\text{Amplitudes:   } q_1 \approx  2.8292, q_2  \approx 2.8268$, $\text{Approximation } \approx 2.9124$}
\label{Fig14}
\end{figure}

In Figure \ref{Fig13} and Figure \ref{Fig14}, all four of the parameters are perturbed with $\epsilon = .1$ and we consider the cases when $\Delta - \Delta_{\text{mod}}$ is equal to .05, .1, .15, and .2. As we increase the delay past the critical delay, we see the amplitude increase as our amplitude approximation would expect. Indeed, our approximation of the amplitude is proportional to $\sqrt{\Delta - \Delta_{\text{mod}}}$, so increasing the difference between the delay and the critical delay will result in an increase in the approximation of the amplitude. Keep in mind that our approximation of the amplitude is really only an $O(1)$ approximation as our calculation was based off of Equation \ref{v20 equation 3} and did not rely on the equations obtained by collecting $O(\epsilon)$ terms, which made the analysis more manageable. Consequently, it is not particularly surprising that we see a noticeable amount of error between the actual amplitude and our approximation of the amplitude. The error seems to be around .1 for all four cases.

\section{Conclusion and Future Research} \label{sec_conclusion}

In this paper, we analyze a two-dimensional fluid model that incorporates customer choice which depends on delayed queue length information. This model is different from those considered in previous literature because of the asymmetry we introduced by perturbing four of the model parameters corresponding to one of the two queues. Analyzing this model allows us to explore the impact that breaking the symmetry has on the dynamics of the queueing system. We see how perturbing different model parameters can have different effects on the system's equilibrium, which we find a first-order approximation for. We consider the stability of this equilibrium to derive a first-order approximation for the critical delay at which the system exhibits a change in stability. Numerical experiments suggest that a Hopf bifurcation occurs at this critical delay as a limit cycle appears to be born when the delay is increased past the critical delay value we derived. We employ Lindstedt's method to get an $O(1)$ approximation for the amplitude of limit cycles near the bifurcation point. 

There are several extensions that could be made to this work. One extension would be to consider a system generalized to have $N > 2$ queues where parameters corresponding to each queue have different perturbations so that no two queues in the system are symmetric to each other. Another potential extension would be to consider different choice functions or to simply change the information that the choice model depends on. Such analysis could provide a better understanding of how providing customers with different types of information affects the dynamics of the system. One could also consider a queueing system with time-varying arrival rates, as in \citet{pender2018analysis}, but with asymmetry introduced to the system. We plan on exploring some of these extensions in future work.

\section*{Acknowledgements}
We would like to thank the Center for Applied Mathematics at Cornell University for sponsoring Philip Doldo's research. Finally, we acknowledge the gracious support of the National Science Foundation (NSF) for Jamol Pender’s Career Award CMMI \# 1751975.

%**************************************************************************
%**************************************************************************

%\bibliographystyle{plainnat}

\bibliography{Breaking_Symmetry_2}

\bibliographystyle{plainnat}

\end{document}